\documentclass{amsart}

\usepackage[margin=1in]{geometry}
\usepackage{amsmath}
\usepackage[alphabetic,nobysame]{amsrefs}
\usepackage{amssymb, color}
\usepackage{mathrsfs}
\usepackage{graphicx}
\usepackage{float}
\usepackage{epsf}
\usepackage[norelsize]{algorithm2e}

\graphicspath{{Figures/}}

\numberwithin{equation}{section}

\newtheorem{lem}{Lemma}[section]
\newtheorem{thm}{Theorem}[section]
\newtheorem{prop}[thm]{Proposition}
\newtheorem{cor}[thm]{Corollary}
\theoremstyle{remark}
\newtheorem{rmk}{Remark}[section]

\newcommand{\bd}{\boldsymbol}
\renewcommand{\tilde}{\widetilde}
\renewcommand{\hat}{\widehat}

\newcommand{\nn}{\nonumber}
\newcommand{\Ni}{\noindent}
\newcommand{\R}{{\mathbb R}}

\newcommand{\del}{\partial}
\newcommand{\dx}{ \, {\rm d} x}
\newcommand{\dv}{ \, {\rm d} v}

\newcommand{\La}{\left\langle}
\newcommand{\Ra}{\right\rangle}
\newcommand{\Id}{{\bf{1}}}

\newcommand{\ii}{I}

\newcommand{\CalB}{{\mathcal{B}}}
\newcommand{\CalD}{{\mathcal{D}}}

\newcommand{\CalL}{{\mathcal{L}}}
\newcommand{\CalP}{{\mathcal{P}}}
\newcommand{\CalT}{{\mathcal{T}}}

\newcommand{\Vecf}{f}
\newcommand{\Vf}{f}
\newcommand{\Vg}{g}

\newcommand{\VU}{\vec U}
\newcommand{\VecL}{{\CalL}}

\newcommand{\Vecphi}{{{\phi}}}
\newcommand{\Vphi}{\phi}

\newcommand{\NullL} {{\rm Null} \, \VecL}

\newcommand{\mc}[1]{\mathcal{#1}}

\newcommand{\TT}{\mathrm{T}}

\newcommand{\VV}{\mathbb{V}}
\newcommand{\Be}{{\bf e}}

\newcommand{\ud}{\,\mathrm{d}}
\newcommand{\rd}{\mathrm{d}}

\newcommand{\norm}[1]{\left\lVert#1\right\rVert}

\newcommand{\viint}[2]{\left\langle#1, \, #2 \,\right\rangle}
\newcommand{\vpran}[1]{\left(#1\right)}

\DeclareMathOperator{\Span}{span}
\DeclareMathOperator\erf{erf}
\DeclareMathOperator\diag{diag}
\DeclareMathOperator\rank{rank}

\begin{document}

\title{A convergent method for linear half-space kinetic equations}

\author{Qin Li}
\address{Computing and Mathematical Sciences, California Institute of Technology, 1200 E California Blvd. MC 305-16, Pasadena, CA 91125 USA. Present address: Department of Mathematics, University of Wisconsin-Madison, Madison, WI, 53705 USA.}
\email{qinli@math.wisc.edu}
\author{Jianfeng Lu}
\address{Departments of Mathematics, Physics, and Chemistry, Duke University, Box 90320, Durham, NC 27708 USA.}
\email{jianfeng@math.duke.edu}
\author{Weiran Sun}
\address{Department of Mathematics, Simon Fraser University, 8888 University Dr., Burnaby, BC V5A 1S6, Canada}
\email{weirans@sfu.ca}

\date{\today}
\thanks{We would like to express our gratitude to the NSF grant RNMS11-07444 (KI-Net), whose activities initiated our collaboration.  The research of Q.L. was supported in part by the AFOSR MURI grant FA9550-09-1-0613 and the National Science Foundation under award DMS-1318377. The research of J.L.~was supported in part by the Alfred P.~Sloan Foundation and the National Science Foundation under award DMS-1312659. The research of W.S.~was supported in part by the Simon Fraser University President's Research Start-up Grant PRSG-877723 and NSERC Discovery Individual Grant \#611626.}\thanks{J.L.~would also like to thank Zheng Chen, Jian-Guo Liu, Chi-Wang Shu for helpful discussions. W.S.~would like to thank Cory Hauck for pointing out the reference [ES12].}

\begin{abstract}
 We give a unified proof for the well-posedness of a
  class of linear half-space equations with general incoming data and
  construct a Galerkin method to numerically resolve this type of
  equations in a systematic way. Our main strategy in both analysis
  and numerics includes three steps: adding damping terms to the
  original half-space equation, using an inf-sup argument and even-odd
  decomposition to establish the well-posedness of the damped
  equation, and then recovering solutions to the original half-space
  equation.  The proposed numerical methods for the damped equation is
  shown to be quasi-optimal and the numerical error of approximations
  to the original equation is controlled by that of the damped
  equation. This efficient solution to the half-space problem is
  useful for kinetic-fluid coupling simulations.
\end{abstract}

\maketitle

\section{Introduction}
In this paper we propose a Galerkin method for computing a class of
half-space kinetic equation with given incoming data:
\begin{equation}\label{eq:half-space-1}
  \begin{aligned}
     (v_1+u) \partial_x &f + \mc{L} f = 0, \quad && x \in [0, +\infty), \,\, v \in \VV \subseteq \R^d  \,, \\
    f\big|_{x=0} &= \phi(v), \quad && v_1+u > 0 \,. \\
  \end{aligned}
\end{equation}
where $u \in \R$ is a given constant, $x$ is the spatial variable and $v$ is the velocity variable. Typical examples for the velocity space $\VV$ are $\VV = [-1, 1]$ and $\VV = \R^d$. The density function $f$ is vector-valued when the system has multiple species. The integral operator $\CalL$ only acts on the velocity variable $v$. The specific structure and main assumptions regarding $\CalL$ will be given in Section~\ref{sec:setting}. 

In asymptotic analysis, half-space equations arise as leading-order boundary-layer equations for kinetic equations with multi-scales. Their solutions bridge the gap between the fluid and kinetic
boundary conditions. One motivation of our work is to study the kinetic-fluid coupling using the domain-decomposition method, where the half-space equation serves as the intermediate equation between the fluid and kinetic regimes. In this case, understanding the well-posedness of \eqref{eq:half-space-1} and constructing accurate and efficient numerical schemes to resolve it will provide explicit characterization of the couplings. 

In the literature the well-posedness of equation
\eqref{eq:half-space-1} has long been investigated \cites{BardosSantosSentis:84, BLP:79,
BardosYang:12, CoronGolseSulem:88, Golse:08, UkaiYangYu:03} for
various models. For example, when $\CalL$ is the linearized Boltzmann operator, the well-posedness 
of such half-space equation is fully proved in the fundamental work by
Coron, Golse, and Sulem \cite{CoronGolseSulem:88}. In this work, it is
shown that depending on the choices of $u$, one needs to prescribe
various numbers of additional boundary conditions such that
\eqref{eq:half-space-1} is well-posed. These numbers of boundary
conditions correspond to the counting of the incoming Euler
characteristics at $x = \infty$. 
The proof in \cite{CoronGolseSulem:88} relies mainly on the energy
method. Subsequently, a different proof using a variational
formulation of \eqref{eq:half-space-1} for the linearized Boltzmann
equation is given in
\cite{UkaiYangYu:03}. 
The key idea in \cite{UkaiYangYu:03} is to revise
\eqref{eq:half-space-1} by adding certain damping terms. The revised
collision operator thus obtained is coercive and it enforces the
end-state of $f$ at $x=\infty$ to be zero. By the conservation
properties of $\CalL$, the authors then show that
\eqref{eq:half-space-1} is well-posed for a large class of incoming
data. One restriction in \cite{UkaiYangYu:03} is that $u$ cannot be
chosen in the way such that the Mach number of the system is $-1, 1$,
or $0$. This restriction was later removed in \cite{Golse:08}.

The variational formulation is also a common tool in proving the well-posedness of the neutron transport equations over general bounded domains $\Omega$ in $\R^{d}_x$. 
There is a vast literature in this direction and we will only review some of the main framework and results in \cite{EggerSchlottbom:12} which are most relevant to us. In \cite{EggerSchlottbom:12}, the linear operator $\CalL$ is the subcritical neutron transport operator. Hence it has a trivial null space. The main novelty of \cite{EggerSchlottbom:12} is that one decomposes the
solution $f$ into its even and odd parts in $v$ and imposes different
regularities for these two parts. Using this mixed regularity, the
authors of \cite{EggerSchlottbom:12} 
write the kinetic equation into a variational form and verify that the bilinear operator involved
satisfies an inf-sup condition over a properly chosen function
space. Moreover, they show that for appropriately constructed Galerkin
approximations, the bilinear operator satisfies the inf-sup condition
over finite-dimensional approximation spaces as well. This then shows
the Galerkin approximation is quasi-optimal. Note that the
even-odd parity was widely used for transport equations, see for
example \cite{JinPareschiToscani:01}.

There are two main goals in our paper: first, we will generalize the
analysis in \cites{EggerSchlottbom:12, Golse:08, UkaiYangYu:03} to
obtain a unified proof for the well-posedness of half-space equations
in the form of \eqref{eq:half-space-1}. Second, we will develop a
systematic Galerkin method to numerically resolve
\eqref{eq:half-space-1} and obtain accuracy estimates for our scheme.

We now briefly explain our main results and compare them with previous ones in the literature. In terms of analysis, we show that with appropriate additional boundary conditions at $x = \infty$ given in \cite{CoronGolseSulem:88}, equation \eqref{eq:half-space-1} has a unique solution.  The basic framework we use is the  even-odd variational formulation developed in \cite{EggerSchlottbom:12}. Compared with \cite{EggerSchlottbom:12}, here we allow the linear operator $\CalL$ to have a nontrivial null space and the background velocity $u$ to be any arbitrary constant for general models. The number of additional boundary conditions will change with $u$. 

Due to the loss of coercivity of $\CalL$, if one directly applies the variational method in \cite{EggerSchlottbom:12} then the bilinear operator $\CalB$ ceases to satisfy the inf-sup condition. To overcome this degeneracy, we utilize the ideas in \cites{Golse:08, UkaiYangYu:03} by adding damping terms to \eqref{eq:half-space-1} and reconstructing solutions to \eqref{eq:half-space-1} from the damped equation.  In the case of linearized Boltzmann equation with a single species, we thus recover the results (in the $L^2$ spaces) in \cites{Golse:08, UkaiYangYu:03}. 

The main differences between our work and \cites{Golse:08, UkaiYangYu:03} are: first, we use a different variational formulation which is convenient for performing numerical analysis. Second, the reconstruction in \cites{Golse:08, UkaiYangYu:03} is restricted to a set of incoming data with a finite codimension such that the damping terms are identically zero. Here we use slightly different damping terms and we recover solutions to \eqref{eq:half-space-1} from the damped equation for any incoming data. 

On the other hand, our main concern is the convergence and accuracy of the numerical scheme and the basic $L^2$-spaces are sufficient for this purpose. Therefore, except for the hard sphere case, we do not try to achieve decay rates estimates of the half-space solution to its end-state at $x=\infty$, while in the literature there are  a lot of works that show subexponential or superpolynomial decay of the solution to its end-state for hard or soft potentials for the linearized Boltzmann equation(see for example \cites{CLY-04, WYY-06, WYY-07}).

Our analysis also applies to linearized Boltzmann equations with multiple species and linear neutron transport equations with critical or subcritical scatterings, thus providing an alternative proof to the well-posedness result (in the $L^2$-space) in \cite{BardosYang:12}.

In parallel with the analysis, numerically we first solve the damped half-space equation and then recover the solution to the original equation. We will use a spectral method and achieve quasi-optimal accuracy (for the damped equation) as in~\cite{EggerSchlottbom:12}. 
The spectral method dates back to Degond and Mas-Gallic~\cite{DegondMas-Gallic:87} for solving radiative transfer equations, and was later extended by Coron~\cite{Coron:90} to solving the linearized BGK equation as well. Compared with these works, our approach differs in three ways: First, as a result of using the even-odd formulation, we can derive explicit boundary conditions for the approximate equations. In particular, the number of these boundary conditions is shown to be
consistent with the number of the unknowns. Hence our discrete systems are always well-posed. This was not the case in
\cite{Coron:90} where a least square method was used  to solve a potentially overdetermined problem. Second, the method in \cite{Coron:90} used Hermite functions defined on the
whole velocity space as their basis functions. This leads to severe
Gibbs phenomenon, since in general the solution to the half-space
equation has a finite jump at $x = 0$ and $v = -u$. Here we choose to use basis functions with jumps at $v = -u$ which naturally fit into the even-odd formulation. This idea is inline with the double $P_N$ method. Third, we will treat the cases with arbitrary bulk velocities $u$ in a uniform way while in~\cites{Coron:90} different schemes are used for the cases $u = 0$ and $u \neq 0$. 

Since the main purpose of the current work is to establish the basic theoretical framework for solving the half-space equations, we only present two numerical examples in this paper. Both of them are for 1D velocity space and a single species. More extensive tests for multi-dimensional velocity space, multi-species, and multi-frequency cases will be done in a forthcoming paper \cite{LiLuSun2015} where general boundary conditions including various reflections at the boundary are considered.

There are also non-spectral methods developed for solving the
half-space equations. For example, 
the work by Golse and Klar~\cite{GolseKlar:95} uses Chapman-Enskog
approximation with diffusive closures. The accuracy of these approximations would be hard to analyze: the iterative approach couples the error from the systematic expansion truncation with the numerical error.  Moreover, this work (\cite{GolseKlar:95})
also treats the cases $u = 0$ and $ u \not = 0$ separately.  A positivity-preserving DG method was proposed in~\cite{CGP12} to treat the Vlasov-Boltzmann transport equation where algebraical convergence is proved. The
recent work by Besse \textit{et al.}~\cite{BesseBorghol:11} treats the half-space problem as
a boundary layer matching kinetics with the limiting fluid equation,
where a Marshak type approximation~\cite{Marshak:47} is applied for
boundary fluxes. Similar idea was also used
in~\cite{Dellacherie:03}. As shown already in \cite{Coron:90}, in
general the Marshak approximation does not yield accurate
approximations to the half-space problem.

The layout of this paper as follows: in Section 2, we gather the basic
information related to the linear operator $\CalL$ and the properties of the damped operator we will be using in the proof, together with the variational formulation we use. Section 3 is devoted to show the well-posedness of the damped equation and the recovery of the original equation. In Section 4 we show its numerical counterpart and present the result on the Galerkin approximation. Section 5 collects all numerical schemes and results for the linearized BGK and linear transport equations.

\section{Linear Operator and Basic Setting}\label{sec:setting}
In this section we will set the framework for our analysis and numerics. In particular, we will show the basic assumptions about the collision operator 
$\CalL$ and the structure of the damped operator and present the variational formulation of a damped version of  \eqref{eq:half-space-1}.

\subsection{Linear collision operator} In order to state the main assumptions imposed on $\CalL$, we first introduce some notations. Denote $\NullL$ as the null space of $\CalL$. Let $\CalP: (L^2(\dv))^m \to \NullL$ be the projection onto $\NullL$. Define the weight function
\begin{equation}\label{def-a}
     a(v) = (1 + |v|)^{\omega_0} \,,
\end{equation}
for some $0 \leq \omega_0 \leq 1$. Throughout the paper we use 
\begin{align} \label{notation}
   \langle f , g \rangle = \langle f , g \rangle_v = \int_{\VV} f \cdot g \dv \,,
\qquad
   \langle f , g \rangle_{x,v} = \int_{\R^d}\int_{\VV} f \cdot g \dv \dx \,.
\end{align}

\subsubsection{Assumptions on $\CalL$} The main assumptions on $\CalL$
are as follows:
\begin{enumerate}
\item[(A1)]  
   $\VecL: \CalD(\VecL) \to (L^2(\dv))^m$ is self-adjoint, nonnegative, and its domain is given by
  \begin{equation*} 
     \CalD(\VecL) = \{\Vecf \in (L^2(\dv))^m \big| \, 
                                a(v) \Vecf \in (L^2(\dv))^m\} \subseteq (L^2(\dv))^m \,,
  \end{equation*}
where $a(v)$ is defined in \eqref{def-a}.
\item[(A2)] $\VecL: (L^2(a\dv))^m \to (L^2(\tfrac{1}{a}\dv))^m$ is bounded, that is, there exists a constant $\sigma_0 > 0$ such that
\begin{equation*}
     \left\|\VecL \Vecf \right\|_{(L^2(\tfrac{1}{a}\dv))^m} 
     \leq  \sigma_0 \left\|\Vecf \right\|_{(L^2(a\dv))^m} \,.
\end{equation*}
\item[(A3)] 
  $\NullL$ is finite dimensional and $\NullL\subseteq (L^p(\dv))^m$ for all $p \in [1, \infty)$. 
\item[(A4)] $\VecL$ has a spectral gap: there exists $\sigma_0 > 0$ such that
\begin{equation*}
      \La f, \,\, \VecL f\Ra 
 \geq \sigma_0 \left\|\CalP^\perp f \right\|_{(L^2(a\dv))^m}^2  
\qquad \text{for any $f \in (L^2(a\dv))^m$} \,, 
\end{equation*}
where $\CalP^\perp = {\mathcal{I}} - \CalP$ is the projection (in $(L^2(\dv))^m$) onto the null orthogonal space $(\NullL)^\perp$.
\end{enumerate}
\smallskip

Note that Assumption (A4) guarantees that $\CalL$ has a bounded inverse on $(\NullL)^\perp$. Throughout this paper, we denote $\CalL^{-1}$ as its pseudo-inverse on $(L^2(\dv))^m$. 

One operator that is of particular importance is $\CalP_1: \NullL \to \NullL$ which is defined by
\begin{equation}
       \CalP_1 (f) = \CalP((v_1+u) f) \qquad \text{for any $f \in \NullL$} \,. \nn
\end{equation}
Note that $\CalP_1$ is a symmetric operator on the finite dimension space $\NullL$. Therefore, its eigenfunctions form a complete set of basis of $\NullL$. Denote $H^+, H^-, H^0$ as the eigenspaces of $\CalP_1$ corresponding to positive, negative, and zero eigenvalues respectively and denote their dimensions as
\begin{equation*}
  \dim H^+ = \nu_+, \qquad \dim H^- = \nu_-, \qquad \dim H^0 = \nu_0 \,.  
\end{equation*}
Let $X_{+,i}, X_{-,j}, X_{0,k}$ be the associated unit eigenfunctions with $1 \leq i \leq \nu_+$, $1 \leq j \leq \nu_-$, and $1 \leq k \leq \nu_0$ for $\nu_\pm, \nu_0 \neq 0$. Note that if any of $\nu_\pm, \nu_0$ is equal to zero, then we simply do not have any eigenfunction associated with the corresponding eigenspace. 
By their definitions, these eigenfunctions satisfy 
\begin{equation}  \label{cond:X-pm-0}
\begin{aligned} 
 \La X_{\alpha, \gamma}, X_{\alpha', \gamma'} \Ra_v = \delta_{\alpha\alpha'} \delta_{\gamma\gamma'} \,, \qquad
 \La (v_1+u) X_{\alpha, \gamma}, \,\, X_{\alpha', \gamma'} \Ra_v =  0 \,\, \text{if $\alpha \neq \alpha'$ or $\gamma \neq \gamma'$} \,, 
\\
 \La (v_1+u) X_{0, j}, \,\, X_{0, k} \Ra_v = 0 \,, \quad 
 \La (v_1+u) X_{+, j}, \,\, X_{+, i} \Ra_v > 0 \,, \qquad
 \La (v_1+u) X_{-, j},  \,\, X_{-, j} \Ra_v < 0 \,, 
\end{aligned}
\end{equation}
where $\alpha \in \{+, -, 0\}$, $\gamma \in \{i,j,k\}$, $1 \leq i \leq \nu_+$, $1 \leq j \leq \nu_-$, and $1 \leq k \leq \nu_0$. These relations in particular give that 
\begin{align*}
   (v_1 + u) X_{0, j} \in (\NullL)^\perp \,,
\qquad
   j = 1, \cdots, \nu_0 \,.
\end{align*}
Therefore $\CalL^{-1} \vpran{(v_1 + u) X_{0, j}} \in (\NullL)^\perp$ is well-defined.

\subsubsection{Examples of $\CalL$.} Many well-known linear or linearized kinetic models satisfy the 
assumptions (A1)-(A4) for the collision operators. These include the classical
linearized Boltzmann equations for either single-species system or
multi-species with hard-sphere collisions and the linear neutron
transport equations. The particular equations that we use as numerical examples are the isotropic neutron transport equation (NTE) with slab geometry and the linearized BGK equation. Similar analysis can be carried out to models
satisfying (A1)-(A4) without extra difficulties.
The main structure of these two equations are as follows. 
The linear operator of the isotropic NTE is the simplest scattering operator which has the form
\begin{equation}\label{operator:NTE}
      \CalL f = f - \frac{1}{2}\int_{-1}^1 f (v) \dv \,.
\end{equation}
In this case, $a(v) = 1 + |v| = \mathcal{O}(1)$ and $(L^2(a\dv))^m$ coincides with $(L^2(\dv))^m$.

The linearized BGK operator is the linearization of the nonlinear BGK operator, which is introduced as a simplified model that captures some
fundamental behavior of the nonlinear Boltzmann equation. The
collision operator of the nonlinear BGK is defined as
\begin{equation*}
     \mathcal{Q}[F] = F - \mathcal{M}[F] \,, 
\end{equation*}
where $\mathcal{M}[F]$ is the local Maxwellian associated with $F$ defined by
\begin{equation*}
      \mathcal{M}[F] = \frac{\rho}{\sqrt{2\pi \theta}} e^{-\frac{|v-u|^2}{2\theta}} \,, 
\end{equation*}
where
\begin{equation*}
     \rho = \int_\R F \dv \,,
\qquad
     \rho u = \int_\R v F \dv \,,
\qquad
    \rho u^2 + \rho \theta = \int_\R v^2 F \dv \,. 
\end{equation*}
For a given bulk velocity $u \in \R$, define the global Maxwellian with the 
steady state $(\rho, u, \theta) = (1, u, 1/2)$ as 
\begin{equation*}
    M_u =  \frac{1}{\sqrt{\pi}} e^{-|v-u|^2} \,.
\end{equation*}
Linearizing the operator $\mathcal{Q}$ around $M$ by setting 
\begin{equation*}
      F = M_u + \sqrt{M_u} f \,, 
\end{equation*}
we obtain the linearized BGK operator 
\begin{equation*}
       \mc{L}_u f = f - m_u, 
\end{equation*}
where $m_u(v)$ is $f$ projected onto the kernel space of $\mathcal{L}_u$. In the case of the 1D linearized BGK, one has:
\begin{equation*}
    \NullL_u = \Span\{\sqrt{M_u}, \,\, v \sqrt{M_u}, \,\, v^2 \sqrt{M_u}\} \,.
\end{equation*}
Therefore, $m_u(v)$ is a quadratic function associated with a Maxwellian to $1/2$ power:
\begin{equation*}
     m_u(v) =  \big(\tilde\rho + \tilde u (v-u) + \tfrac{\tilde\theta}{2} ((v-u)^2 - 1) \big)  \sqrt{M_u} \,,
\end{equation*}
where $(\tilde \rho, \tilde u, \tilde \theta)$ are defined in the way such that first three moments of $m(v)$ agree with those of $f$: 
\begin{equation*}
     \langle f-m_u,v^k\sqrt{M_u}\rangle = \int_\R (f-m_u)v^k \sqrt{M_u}\ud{v} = 0,\quad k = 0,1,2 \,.
\end{equation*}
The half-space equation with the linearized BGK operator that centered at bulk velocity $u$ is: 
\begin{equation}\label{eq:half-space-BGK}
\begin{aligned}
     v \del_x f + &{\mathcal{L}_u} f =  0, \\
    \Vecf|_{x=0} &=  \Vecphi(v) \,,  \qquad v > 0 \,.
\end{aligned}
\end{equation}
Following the classical treatment of the half-space equations, we shift the center of the Maxwellian $M_u$ to the origin by 
performing the change of variable $v - u\to v$. The half-space 
equation~\eqref{eq:half-space-BGK} then becomes
\begin{equation}\label{eq:half-space-2}
\begin{aligned}
     (v + u) \del_x &f + {\mathcal{L}} f =  0, \\
    \Vecf|_{x=0} &=  \Vecphi(v + u) \,,  \qquad v + u > 0 \,,
\end{aligned}
\end{equation}
where
\begin{equation} \label{eq:bgk_operator}
         \CalL f = f - m(v) \,,
\qquad
          m(v) =  m_0 \,,    
\end{equation}
and the null space of $\CalL$ becomes 
\begin{equation*}
     \NullL = \Span\{\sqrt{M}, \,\, v \sqrt{M}, \,\, v^2 \sqrt{M}\} \,,
\end{equation*}
where $M$ is the global Maxwellian centered at the origin such that
\begin{equation*}
    M = M_0 =  \frac{1}{\sqrt{\pi}} e^{-v^2} \,. 
\end{equation*}

As defined in~\eqref{cond:X-pm-0}, we look for $H^{\pm,0}$ decomposition of $\NullL$. For this particular case one could write down the basis functions explicitly. Following~\cite{CoronGolseSulem:88}, we define
\begin{equation}\label{eqn:def_chi}
  \begin{cases}
    \chi_0 =  \frac{1}{6^{1/2} \pi^{1/4}}\left(2v^2-3\right) \exp(-v^2/2)\\[2pt]
    \chi_\pm = \frac{1}{6^{1/2} \pi^{1/4}}\left(\sqrt{6}v\pm 2v^2\right) \exp(-v^2/2)
  \end{cases}\,.
\end{equation}
It is easy to show that 
\begin{equation}\label{eqn:chi_u}
\begin{cases}
\langle\chi_\alpha,\chi_\beta\rangle_v=\int_\R \chi_\alpha\chi_\beta\ud{v} = \delta_{\alpha\beta} \,, \\
\langle (v + u) \chi_\alpha,\chi_\beta\rangle_v=0 \,, \qquad \alpha \neq \beta \,, \\
\langle(v+u)\chi_0,\chi_0\rangle_v = u_0 = u \,,\\
\langle(v+u)\chi_+,\chi_+\rangle_v = u_+ = u+c \,, \\
\langle(v+u)\chi_-,\chi_-\rangle_v = u_- = u-c \,,\\
\end{cases}
\end{equation}
where  $\alpha, \beta \in \{+, -, 0\}$, $c = \sqrt{3/2}$,  and 
\begin{equation*}
     \La f, g\Ra_v = \int_\R f g \dv \,. 
\end{equation*}
%
%
Using these new basis functions, we can decompose $\NullL$ into subspaces: $\NullL = H^+ \oplus H^- \oplus H^0$ with: 
\begin{equation*}
H^+ = \Span\left\{\chi_\beta | \,\, u_\beta>0\right\}, 
\quad
H^- = \Span\left\{\chi_\beta | \,\, u_\beta < 0\right\}, 
\quad
H^0 = \Span\left\{\chi_\beta | \,\, u_\beta = 0\right\},  
\end{equation*}
where again $\beta \in \{+, -, 0\}$. 
For each fixed $u \in \R$, denote the dimensions of these subspaces as
\begin{equation*}
  \dim H^+ = \nu_+, \qquad \dim H^-= \nu_-, \qquad \dim H^0 = \nu_0 \,.  
\end{equation*}
Note that $\nu_\pm, \nu_0$ change with $u$. In particular, we have the following categories: 
\begin{align}\label{eqn:SevenCases}
\begin{cases}
u <-c: & (\dim H^+, \dim H^-, \dim H^0) = (0, 3, 0) \,, \\
u = -c: & (\dim H^+, \dim H^-, \dim H^0) = (0, 2, 1) \,, \\
-c < u < 0: & (\dim H^+, \dim H^-, \dim H^0) = (1, 2, 0)\,,\\ 
u = 0: & (\dim H^+, \dim H^-, \dim H^0) = (1, 1, 1) \,, \\
0 < u < c: & (\dim H^+, \dim H^-, \dim H^0) = (2, 1, 0) \,, \\
u = c: & (\dim H^+, \dim H^-, \dim H^0) = (2, 0, 1) \,, \\
u > c: & (\dim H^+, \dim H^-, \dim H^0) = (3, 0, 0)\,.
\end{cases}
\end{align}
This gives an explicit example that shows the structure of $\NullL$ changes with $u$.

\subsection{Damped Linear Operator $\CalL_d$} 
The main difficulty in both analysis and numerics is the
non-coercivity of $\CalL$. Although in some cases this degeneracy of
$\CalL$ can be handled by carefully choosing appropriate function
spaces for the variational formulation, we prefer to work with
strictly dissipative operators. To this end, we utilize the idea
developed in \cites{Golse:08, UkaiYangYu:03} to modify the original
equation \eqref{eq:half-space-2} by adding in damping terms. The
particular damping terms are chosen in the way such that we can easily
recover the undamped equation \eqref{eq:half-space-2} for any incoming
data and such that the damped operator is symmetric. The particular
damped operator we introduce is
\begin{equation}\label{eqn:def_damping}
  \begin{aligned}
    \VecL_d f = &\CalL f + \alpha \sum_{k=1}^{\nu_+} (v_1 + u) X_{+,k}
    \La (v_1 + u)X_{+,k}, f \Ra_{v}\\
    & + \alpha \sum_{k=1}^{\nu_-} (v_1 + u) X_{-,k} \La (v_1 +
    u)X_{-,k}, f \Ra_{v}
    + \alpha \sum_{k=1}^{\nu_0} (v_1  + u) X_{0,k} \La (v_1 + u) X_0, f \Ra_{v}\\
    & + \alpha \sum_{k=1}^{\nu_0} (v_1 + u) \CalL^{-1}((v_1 +
    u)X_{0,k})\La (v_1 + u) \CalL^{-1}((v_1 + u)X_{0,k}), f \Ra_{v}.
   \end{aligned}
\end{equation}
Here the constant $\alpha$ satisfies that $0 < \alpha \ll 1$. The size of $\alpha$ only depends on $\CalL$. The main property of $\CalL_d$ is its coercivity as stated in the following lemma:
\begin{lem} \label{lem:prop-L-d}
Let $\CalL$ be the linear operator that satisfies Assumptions (A1)-(A4). Then there exist two constants $\sigma_1, \alpha_0  > 0$ such that for any $0 < \alpha \leq \alpha_0$ we have
\begin{align*}
   \viint{f}{\CalL_d f} 
\geq 
   \sigma_1 \norm{f}^2_{(L^2(a\dv))^m}
\qquad
  \text{for any $f \in \CalD(\CalL)$} \,.
\end{align*}
\end{lem}
\begin{proof} 
By the definition of $\CalL_d$, we have
\begin{align*}
    \viint{f}{\CalL_d f}
    = &\La \Vf, \VecL \Vf \Ra 
        +  \alpha \sum_{k=1}^{\nu_+}  
               \La (v_1 + u) X_{+,k}, f \Ra^2
        + \alpha \sum_{k=1}^{\nu_-}
             \La (v_1 + u) X_{-,k},  f \Ra^2
       + \alpha  \sum_{k=1}^{\nu_0} 
             \La (v_1 + u) X_0, f \Ra^2      
\\
    & + \alpha  \sum_{k=1}^{\nu_0}
              \La (v_1 + u) \CalL^{-1}((v_1 + u)X_{0,k}), f\Ra^2  \,.\end{align*}
Write
\begin{equation*}
      f = f^\perp + \sum_{i=1}^{\nu_+} {f_{+,i}} X_{+,i}
           + \sum_{j=1}^{\nu_-} {f_{-,j}} X_{+,j}
           + \sum_{k=1}^{\nu_0} {f_{0,k}} X_{0,k} \,, 
\end{equation*}
where $f^\perp = \tilde\CalP f \in (\NullL)^\perp$. 
By Assumption (A4), if we chose $0 < \alpha \ll 1$, then
\begin{equation}
     \label{bound-B-1-1}
\begin{aligned}
  \viint{f}{\CalL_d f} 
\geq \, 
  & \sigma_0
  \|f^\perp\|_{(L^2(a\dv))^m}^2 + \frac{\alpha}{4}
  \sum_{k=1}^{\nu_+} \gamma_{+,k}^2 f_{+,k}^2 +
  \frac{\alpha}{4} \sum_{k=1}^{\nu_+} \gamma_{-,k}^2 f_{-,k}^2
\\
  & - \frac{\alpha}{4} \sum_{k=1}^{\nu_+} \La (v_1 + u) X_{+,k},
  f^\perp \Ra^2
  - \frac{\alpha}{4} \sum_{k=1}^{\nu_-} \La (v_1 + u) X_{-,k},
  f^\perp \Ra^2
  \\
  \geq \, & 
    \frac{\sigma_0}{2} \|f^\perp\|_{(L^2(a\dv))^m}^2 
    + \frac{\alpha}{4} \sum_{k=1}^{\nu_+} \gamma_{+,k}^2  f_{+,k}^2 
    + \frac{\alpha}{4} \sum_{k=1}^{\nu_-} \gamma_{-,k}^2 f_{-,k}^2 \,,
\end{aligned}
\end{equation}
where $\gamma_{\pm}$'s are defined as
\begin{equation}
   \label{notation:1}
\begin{aligned}   
  &\gamma_{+,i} := \La (v_1 + u) X_{+,i}, X_{+,i} \Ra_v > 0 \,, \quad
  && 0 \leq i \leq \nu_+ \,,
  \\
  & \gamma_{-,j} := - \La (v_1 + u) X_{-,j}, X_{-,j} \Ra_v > 0\,,
  \quad && 0 \leq j \leq \nu_- \,.
\end{aligned}
\end{equation}
In addition, if $\nu_0 \neq 0$, then
\begin{equation}
  \label{bound-B-1-2}
  \begin{aligned}
    \viint{f}{\CalL_d f} 
\geq \, 
   & \sigma_0 \|f^\perp\|_{(L^2(a\dv))^m}^2
        + \alpha  \sum_{k=1}^{\nu_0}
              \La (v_1 + u) \CalL^{-1}((v_1 + u)X_{0,k}), f\Ra^2 
\\
\geq \,
      & \frac{\sigma_0}{2} \|f^\perp\|_{(L^2(a\dv))^m}^2
         + \frac{\alpha}{4\nu_0} \left(\sum_{k,m=1}^{\nu_0} 
              \La (v_1 + u) \CalL^{-1}((v_1 + u)X_{0,k}), X_{0,m} 
               \Ra_v f_{0,m}\right)^2  
\\
     & - \frac{\alpha}{4}  \sum_{k=1}^{\nu_0}
              \La (v_1 + u) \CalL^{-1}((v_1 + u)X_{0,k}), \sum_{m=1}^{\nu_+} f_{+,k}X_{+,k}\Ra^2             
\\
     & - \frac{\alpha}{4}  \sum_{k=1}^{\nu_0}
              \La (v_1 + u) \CalL^{-1}((v_1 + u)X_{0,k}), \sum_{m=1}^{\nu_+} f_{-,k}X_{-,k}\Ra^2             
\\
     & - \frac{\alpha}{4}  \sum_{k=1}^{\nu_0}
              \La (v_1 + u) \CalL^{-1}((v_1 + u)X_{0,k}), f^\perp\Ra^2   \,.
\end{aligned}
\end{equation}
Since the matrix $\big(\La (v_1 + u) \CalL^{-1}((v_1 + u)X_{0,k}), X_{0,m} \Ra_v \big)$ is strictly positive, there exists a constant $c_0 > 0$ such that 
\begin{equation}
    \sum_{k=1}^{\nu_0}\biggl(\sum_{m=1}^{\nu_0} 
              \La (v_1 + u) \CalL^{-1}((v_1 + u)X_{0,k}), X_{0,m} 
               \Ra_v f_{0,m}\biggr)^2
\geq c_0 \sum_{k=1}^{\nu_0} f_{0,m}^2 \,.
\end{equation}
Hence by multiplying \eqref{bound-B-1-1} by a large enough number and adding it to \eqref{bound-B-1-2}, we have
\begin{equation}
    \label{est:1}
     \viint{f}{\CalL_d f} 
\geq 
    \sigma_1 \|\Vf\|_{(L^2(a\dv))^m}^2
\qquad \text{for some $\sigma_1 > 0$ } \,.
\end{equation}
provided $0 < \alpha \ll 1$.
\end{proof}

\subsection{Variational Formulation}
In this part we present the variational formulation for the half-space equation. First, we state the full equation that we want to study in this paper using the notation of $H^\pm, H^0$,. Suppose $\CalL$ is a linear operator in $v$ that satisfies (A1)-(A4). Our goal is to prove the well-posedness of the following equation and then construct efficient numerical schemes and obtain estimate of its accuracy:
\begin{equation}\label{eq:half-space-general}
  \begin{aligned}
     (v_1+u) \partial_x &f + \mc{L} f = 0, \quad && x \in [0, +\infty), \,\, v \in \VV  \,, \\
    f\big|_{x=0} &= \phi(v), \quad && v_1+u > 0 \,, \\
    f - f_\infty \in &(L^2(\dv\dx))^m \,,
  \end{aligned}
\end{equation}
for some $f_\infty \in H^+ \oplus H^0$. The particular formulation
about the end-state $f_\infty$ was given in \cite{CoronGolseSulem:88}
(for single species $m = 1$) where the authors proved the
well-posedness of the half-space linearized Boltzmann equation:
\begin{thm}[\cite{CoronGolseSulem:88}] \label{thm-CGS}
Let $\CalL$ be the linearized Boltzmann operator with a hard-sphere collision 
kernel and the incoming data $\phi \in L^2(a(v) \Id_{v_1+u>0}\dv)$. Then there exists a constant $\beta > 0$ and a unique 
$f_\infty \in H^+ \oplus H^0$ such that 
equation \eqref{eq:half-space-1} has a unique solution $f$ which satisfies
\begin{equation*}
     f -  f_\infty \in L^2(e^{2\beta x} \dx; L^2(a \dv)) \,, 
\end{equation*}
where $a(v) = 1 + |v|$. 
\end{thm}

\begin{rmk}
The main result in \cite{CoronGolseSulem:88} is actually stronger than Theorem \ref{thm-CGS} 
where $f-f_\infty$ is shown to be in $L^\infty(e^{2\beta x} \dx; L^2(\dv))$. 
Here we content ourselves with the $L^2$-weighted space (in $x$) since $L^2$ suffices our needs in proving the quasi-optimal convergence of our numerical scheme.
\end{rmk}

We will use $\VV = \R^3$ as the setting to explain the variational formulation. Other spaces for $v$ will work in a similar way. Let $u \in \R$ be given. 
We use the damped operator $\CalL_d$ and obtain the modified equation as
\begin{align}\label{eq:half-space-damping}
     (v_1 + u) \del_x f &+ \VecL_d f  = 0 \,,   \nn
\\
     f \big|_{x=0} &= \phi(v_1) \,,  
\hspace{1cm} 
    v_1+u > 0 \,.
\end{align}
We define the shifted ``even'' and ``odd'' parts of a function as 
\begin{equation}
    \label{even-odd}
\begin{aligned}    
    f^+(v) = \frac{f(v_1, v_2, v_3) + f(-2u - v_1, v_2, v_3)}{2} \,,
\qquad
    f^-(v) = \frac{f(v_1, v_2, v_3) - f(-2u - v_1, v_2, v_3)}{2}
\end{aligned}   
\end{equation}
such that $f = f^+ + f^-$ and 
\begin{equation*}
     f^\pm(-u + v_1, v_2, v_3) = \pm f^\pm(-u - v_1, v_2, v_3) \,.
\end{equation*}
Define the function space 
\begin{equation} \label{def-Gamma}
   \Gamma = \bigl\{f \in (L^2(a\dv\dx))^m \; \big| \; (v_1+u) \del_x f^+ \in (L^2(\tfrac{1}{a} \dv\dx))^m \bigr\}, 
\end{equation}
which  is a Hilbert space with the inner product 
\begin{equation*}
     \La f, g \Ra_\Gamma
     =\int_{\R} \int_{\R^3} f \cdot g  \, a\dv\dx 
       + \int_{\R} \int_{\R^3}  (v_1 + u) \del_x \Vf^+ \cdot (v_1+u) \del_x g^+ \,  \tfrac{1}{a} \dv\dx \,.
\end{equation*}
Thus the norm of $\Gamma$ is equivalent to
\begin{equation*} 
       \| \Vf \|_{(L^2(a \dv\dx))^m}
       + \|(v_1 + u) \del_x \Vf^+\|_{(L^2(\tfrac{1}{a} \dv\dx))^m}\,.
\end{equation*}
Moreover, every element $g \in \Gamma$ has a well-defined trace:
\begin{equation}
     \CalT: \Gamma \to (L^2(|v_1 + u| \dv))^m
\end{equation}
such that
\begin{equation}
     \CalT g =  g^+ \big|_{x=0} \,, \qquad \text{for all $g \in C([0, \infty); (L^2(a\dv))^m)$} \,,
\end{equation}
and
\begin{equation}
     \int_{\R^3} |v_1 + u| |g^+|^2 \dv < \infty \,.
\end{equation}

Now we define a bilinear operator $\CalB: \Gamma \times \Gamma \to \R$ such that
\begin{equation}
    \label{def-B}
\begin{aligned}
    \CalB(\Vf, \psi) 
    = &- \La \Vf^-, (v_1 + u) \del_x \psi^+\Ra_{x,v} 
      + \La (v_1 + u) \del_x \Vf^+, \psi^-\Ra_{x,v}
      + \La \psi, \CalL_d\Vf\Ra_{x,v} 
      + \La |v_1+u| \Vf^+, \psi^+ \Ra_{x=0}
\\
  = &- \La \Vf^-, (v_1 + u) \del_x \psi^+\Ra_{x,v} 
      + \La (v_1 + u) \del_x \Vf^+, \psi^-\Ra_{x,v}
      + \La \psi, \CalL\Vf\Ra_{x,v} 
\\
     &+  \alpha \sum_{k=1}^{\nu_+} 
              \La\La (v_1 + u) X_{+,k},  \psi \Ra_v , \,\,
               \La (v_1 + u)X_{+,k}, f \Ra_v \Ra_x
\\
     & + \alpha  \sum_{k=1}^{\nu_-}              
              \La \La (v_1 + u) X_{-,k}, \psi \Ra_v , \,\,
             \La (v_1 + u)X_{-,k},  f \Ra_v \Ra_x
\\
    & +  \alpha \sum_{k=1}^{\nu_0} 
             \La\La (v_1 + u) \CalL^{-1}((v_1 + u)X_{0,k}), \psi \Ra_v , \,\,
              \La (v_1 + u) \CalL^{-1}((v_1 + u)X_{0,k}),  f\Ra_v \Ra_x    
\\
    &+ \alpha \sum_{k=1}^{\nu_0} 
          \La\La (v_1 + u) X_{0,k}, \psi \Ra_v, \,\, 
          \La (v_1 + u) X_{0,k}, f \Ra_v \Ra_x      
    + \La |v_1+u| \Vf^+, \psi^+ \Ra_{x=0}   \,.                       
\end{aligned}
\end{equation}
Recall that the inner product $\La \cdot, \cdot \Ra_{x,v}$ is defined in~\eqref{notation}. 
It is straightforward to check by using integration by parts and
symmetry that the variational formulation of
\eqref{eq:half-space-damping} has the form
\begin{equation}
      \CalB(f, \psi) = l(\psi) \,, \qquad \text{for every $\psi \in \Gamma$} \,.
      \label{variational}
\end{equation}
Here the linear operator $l(\cdot)$ is given by
\begin{equation}
     l(\psi) = 2 \int_{v_1+u>0} (v_1+u) \, \phi \, \psi^+\dv \,, \label{def-l}
\end{equation}
where $\phi$ is the given incoming data and $\psi^+$ is the even (with
respect to $-u$) part of $\psi$ as defined in \eqref{even-odd}.

\section{Well-posedness}\label{sec:main}
In this section we show the well-posedness of the half-space equation~\eqref{eq:half-space-general}. The proof will be done in two steps: first, we use the variational form~\eqref{variational} to show the well-posedness of the damped equation~\eqref{eq:half-space-damping}. Then we construct recovering procedures to find the solution to the original half-space equation. 

\subsection{Solution of the damped equation}
The main tool we use to show the well-posedness of the weak formulation \eqref{variational} is to use the Babu\v{s}ka-Aziz lemma \cite{BabuskaAziz:72}. There are two parts in this lemma and we recall its statement below.

\begin{thm}[Babu\v{s}ka-Aziz]\label{thm-BA}
Suppose $\Gamma$ is a Hilbert space and $\CalB: \Gamma \times \Gamma \to \R$ is a bilinear operator on $\Gamma$.  Let $l: \Gamma \to \R$ be a bounded linear functional on $\Gamma$. 
\smallskip

\noindent (a) If $\CalB$ satisfies the boundedness and inf-sup conditions on $\Gamma$ such that
\begin{itemize}
\item  there exists a constant $c_0 > 0$ such that $|\CalB (f, g)| \leq c_0 \|f\|_\Gamma \|g\|_\Gamma$ for all $f, g \in \Gamma$; 

\item there exists a constant $\kappa_0 > 0$ such that 
\begin{equation}
    \label{cond:inf-sup}
\begin{aligned}
     \sup_{\|\Vf\|_{\Gamma}=1} \CalB(\Vf, \psi) \geq \kappa_0 \|\psi\|_\Gamma \,, \qquad &\text{for any $\psi \in \Gamma$} \,, 
\\
     \sup_{\|\psi\|_{\Gamma}=1} \CalB(\Vf, \psi) \geq \kappa_0 \|\Vf\|_\Gamma \,,  \qquad &\text{for any $f \in \Gamma$}
\end{aligned}       
\end{equation}
for some constant $\kappa_0 > 0$. 

\end{itemize}
then there exists a unique $f \in \Gamma$ which satisfies 
\begin{equation*}
     \CalB (f, \psi) = l(\psi) \,,
\qquad \text{for any $\psi \in \Gamma$} \,.      
\end{equation*}

\noindent (b) Suppose $\Gamma_N$ is a finite-dimensional subspace of $\Gamma$. If in addition $\CalB: \Gamma_N \times \Gamma_N \to \R$ satisfies the inf-sup condition on $\Gamma_N$, then there exists a unique solution $f_N$ such that
\begin{equation*}
     \CalB (f_N, \psi_N) = l(\psi_N) \,,
\qquad \text{for any $\psi_N \in \Gamma_N$} \,. 
\end{equation*} 
Moreover, $f_N$ gives a quasi-optimal approximation to the solution $f$ in (a), that is, there exists a constant $\kappa_1$ such that
\begin{equation*}
    \|f - f_N\|_\Gamma \leq \kappa_1 \inf_{w \in \Gamma_N} \|f - w\|_{\Gamma} \,. 
\end{equation*}
\end{thm}

It is clear that the inf-sup condition of $\CalB$ is essential to the solvability of~\eqref{variational}. We thus first show that $\CalB$ satisfies this condition. 
\begin{prop}[Inf-sup]\label{inf-sup}
Let $\Gamma$ and $\CalB$ be the function space and the bilinear operator defined in \eqref{def-Gamma} and \eqref{def-B} respectively. Then $\CalB: \Gamma \times \Gamma \to \R$ satisfies the inf-sup condition~\eqref{cond:inf-sup}.
\end{prop}
\begin{proof}
Note that $\CalB$ is symmetric in its variables. Hence it suffices to show 
that the second condition in \eqref{cond:inf-sup} holds. To this end, let 
$\Vf \in \Gamma$ be arbitrary. We only need to find an appropriate $\psi$ 
such that 
\begin{equation}
     \label{ineq: inf-sup}
      \CalB(\Vf, \psi) \geq \kappa_0 \|\Vf\|_\Gamma^2 \,,
\qquad
      \|\psi\|_\Gamma \leq \kappa_1 \|\Vf\|_\Gamma \,.      
\end{equation}
Indeed, if $\psi$ satisfies \eqref{ineq: inf-sup}, then one can simply
let $\Psi = \frac{\psi}{\|\psi\|_\Gamma}$ and obtain the second
inequality in \eqref{cond:inf-sup} (with a different constant). The
construction of such $\psi$ will be carried out in two steps. First,
let $\psi_1 = \Vf$. Then by Lemma~\ref{lem:prop-L-d},
\begin{equation*}
\begin{aligned}   
    \CalB(\Vf, \psi_1) 
 = \viint{f}{\CalL_d f}_{x,v}   
          + \La |v_1+u| \Vf^+, \Vf^+ \Ra_{x=0} 
 \geq 
    \sigma_1 \|\Vf\|_{(L^2(a\dv\dx))^m}^2 \,.    
\end{aligned}
\end{equation*}
 Next, let 
\begin{equation*}
      \psi_2 = \frac{1}{(1 + |v_1 + u| + |v_2| + |v_3|)^{\omega_0}} (v_1+u) \del_x \Vf^+ \,. 
\end{equation*}
We claim that $\psi_2 \in \Gamma$. Indeed, by the definition of $a(v)$, one can find two constants $c_1, c_2 > 0$ such that
\begin{equation*}
     \frac{c_1}{a(v)} \leq \frac{1}{(1 + |v_1 + u| + |v_2| + |v_3|)^{\omega_0}} \leq \frac{c_2}{a(v)} \,. 
\end{equation*}
Here the constants $c_1, c_2$ depend on $u$. Thus $\psi_2 \in (L^2(a \dv\dx))^m$ because
\begin{equation*}
    \|\psi_2\|_{(L^2(a \dv\dx))^m} \leq \|(v_1 + u) \del_x \Vf^+\|_{\left(L^2(\tfrac{1}{a} \dv\dx)\right)^m} \leq \|\Vf\|_\Gamma \,. 
\end{equation*}
Moreover the definition of $\psi_2$ implies that 
\begin{equation*}
  \psi_2^+ = 0 \in (L^2(\tfrac{1}{a} \dv\dx))^m \,.  
\end{equation*}
Hence $\psi_2 \in \Gamma$ and it satisfies
\begin{equation}
      \|\psi_2\|_\Gamma \leq \|\Vf\|_\Gamma \,. \label{bound:2}
\end{equation}
Using $\psi_2$ in $\CalB$, we have
\begin{equation*}
\begin{aligned}
   \CalB(\Vf, \psi_2) 
   = &\La (v_1 + u) \del_x \Vf^+, \psi_2\Ra
       + \La \psi_2, \CalL\Vf\Ra 
       + \alpha \sum_{k=1}^{\nu_+} \La \La (v_1 + u) X_{+,k}, \psi_2
       \Ra_v \La (v_1 + u)X_{+,k}, f \Ra_v \Ra_x
       \\
       & + \alpha \sum_{k=1}^{\nu_-} \La \La (v_1 + u) X_{-,k}, \psi_2
       \Ra_v \La (v_1 + u)X_{-,k}, f \Ra_v \Ra_x
       \\
       &+ \alpha \sum_{k=1}^{\nu_0} \La \La (v_1 + u) X_{0,k}, \psi_2
       \Ra_v \La (v_1 + u) X_0, f \Ra_v \Ra_x
       \\
       & + \alpha \sum_{k=1}^{\nu_0} \La \La (v_1 + u) \CalL^{-1}((v_1
       + u)X_{0,k}), \psi_2 \Ra_v \La (v_1 + u) \CalL^{-1}((v_1 +
       u)X_{0,k}), f\Ra_v \Ra_x
       \\
       \geq & \|(v_1 + u) \del_x
       \Vf^+\|_{(L^2(\tfrac{1}{a}\dv\dx))^m}^2 - \kappa_2
       \|\Vf\|_{(L^2(a\dv\dx))^m}^2 \,,
     \end{aligned}
   \end{equation*}
for some constant $\kappa_2 > 0$. Hence by taking $\kappa_3 > 0$ large enough, we have that 
\begin{equation}
       \label{est:2}
       \CalB(\Vf, \kappa_3 \psi_1 + \psi_2) \geq \kappa_0 \|\Vf\|_\Gamma^2 \,,
\end{equation}
for some $\kappa_0 > 0$. Recall that by the definition of $\psi_1$ and \eqref{bound:2}, we also have
\begin{equation*}
      \|\kappa_3 \psi_1 + \psi_2\|_\Gamma 
\leq \sqrt{1 + \kappa_3}  \, \|\Vf\|_\Gamma \,, 
\end{equation*}
which, together with \eqref{est:2}, shows the inf-sup property of $\CalB$ on $\Gamma \times \Gamma$.
\end{proof}

Using the inf-sup property of $\CalB$ and the Babu\v{s}ka-Aziz Lemma, we can now show the solvability of the variational form~\eqref{variational}.

\begin{prop}[Well-posedness of the damped equation]\label{prop:wellposed}
Suppose $\CalL$ satisfies Assumption (A1)-(A4) and $\CalL_d$ is defined as in~\eqref{eqn:def_damping} with $\alpha$ small enough such that the coercivity in Lemma~\ref{lem:prop-L-d} holds. Let $\Vphi \in (L^2(a(v) \Id_{v_1+u>0}\dv))^m$ and $\Gamma$ be the function space defined in~\eqref{def-Gamma}. Then


\Ni (a) There exists a unique $\Vf \in \Gamma$ such that \eqref{variational} holds.


\Ni (b) Moreover, $f$ satisfies that
\begin{equation*}
     (v_1+u) \del_x f \in (L^2(\tfrac{1}{a} \dv\dx))^m
\end{equation*}
and it solves the damped half-space equation in the sense of distributions
\begin{equation}
  \label{eq:half-space-damping-1}
\begin{aligned} 
  (v_1 + u) \del_x f + \VecL_d f 
  = & (v_1 + u) \del_x f + \VecL f  
     +  \alpha  \sum_{k=1}^{\nu_+} 
          (v_1 + u) X_{+,k} \La (v_1 + u)X_{+,k}, f \Ra_v
\\
    & + \alpha \sum_{k=1}^{\nu_-} 
           (v_1 + u) X_{-,k} \La (v_1 + u)X_{-,k},  f \Ra_v
    + \alpha \sum_{k=1}^{\nu_0} 
           (v_1 + u) X_{0,k} \La (v_1 + u) X_{0,k}, f \Ra_v                            
\\
     & + \alpha \sum_{k=1}^{\nu_0}
              (v_1 + u) \CalL^{-1}((v_1 + u)X_{0,k})
             \La (v_1 + u) \CalL^{-1}((v_1 + u)X_{0,k}),  f\Ra_v
          = 0  
\end{aligned}
\end{equation}
with the boundary conditions (defined in the trace sense at $x=0$)
\begin{equation}
     \label{cond:bdry-damped}
\begin{aligned}
     \Vf |_{x=0} = \Vphi(v) \,, \qquad v_1+u > 0 \,. 
\end{aligned} 
\end{equation}

\smallskip

\Ni (c) If $a(v) = 1 + |v|$, then there exists $\beta > 0$ such that $(L^2(e^{2\beta x} \dx; L^2(a\dv)))^m$. 
\end{prop}

\begin{proof}
(a) It is straightforward to verify the boundedness of $\CalB$ and $l$ as defined in \eqref{def-B} and \eqref{def-l}. The well-posednes of the variational form is then an immediate consequence of Proposition \ref{inf-sup} and part (a) of the Babu\v{s}ka-Aziz lemma. 

\Ni (b) In order to show that $(v_1+u) \del_x f \in (L^2(\tfrac{1}{a}\dv\dx))^m$, we note that the damped equation \eqref{eq:half-space-damping-1} holds in the sense of distributions by choosing the test function $\psi \in C_c^\infty((0, \infty) \times \R)$. Thus 
\begin{equation*}
     (v_1+u) \del_x f = \beta (v_1 + u) f - \CalL_d(f) \in (L^2(\tfrac{1}{a} \dv\dx))^m \,. 
\end{equation*}
By the density argument this implies that 
\begin{equation*}
\begin{aligned}
      \La (v_1+u) \del_x f^-, \,\, \psi^+\Ra_{x,v} 
      + \La (v_1+u) \del_x f^+, \,\, \psi^-\Ra_{x,v} 
      - \beta \La (v_1+u) \psi, f \Ra_{x,v}
      + \La \psi, \CalL_d f\Ra_{x,v} = 0 \,,
\end{aligned}
\end{equation*}
for all $\psi \in C^\infty(0, \infty)$.
Therefore, if we choose $\phi \in C^\infty[0, \infty)$ and integrate by parts in the variational form \eqref{variational}, then boundary terms satisfy
\begin{equation*}
   \La (v_1+u)f^-, \,\, \psi^+\Ra_v + \La |v_1+u| f^+, \,\, \psi^+ \Ra_v
   =  2 \int_{v_1+u>0} (v_1+u) \, \phi \, \psi^+ \dv
\qquad \text{at $x=0$} \,,    
\end{equation*}
which implies,
\begin{equation*}
   \int_{v_1+u>0} (v_1+u) f \psi^+ \dv 
   =  \int_{v_1+u>0} (v_1+u) \, \phi \, \psi^+ \dv
\qquad \text{at $x=0$}\,. 
\end{equation*}
Since $\psi^+ \in C^\infty(0, \infty)$ is arbitrary, we have $f = \phi$ at $x=0$ when $v_1 + u > 0$. 

\Ni (c) If $a(v)= 1+ |v|$, then there exists $\beta > 0$ such that $f \in (L^2(e^{2\beta x} \dx; L^2(a\dv)))^m$. The proof will be along the same line for the general case of $a(v)$. We use the standard way to incorporate the exponential into the bilinear form by changing $f$ by $g = e^{\beta x} f$. The new bilinear form $\CalB_\beta$ is
\begin{equation}\nn
      \CalB_\beta(g, \psi) = \CalB (g, \psi) - \beta \La (v_1+u) g, \psi\Ra_{x,v} \,,
\end{equation}
 where $\CalB(g, \psi)$ is defined in~\eqref{def-B}. Note that by Cauchy-Schwartz, if we choose $0 < \beta \ll \alpha$, then by the spectral gap assumption (A4), we have
\begin{equation} \nn
\begin{aligned}
     \big | \beta \La (v_1 + u) g, \,\, \psi_1\Ra_{x,v}\big |
&\leq \frac 12 \CalB(g, \psi) \,,
\\
     \big | \beta \La (v_1 + u) g \,\, \psi_2\Ra_{x,v}\big |
&\leq \frac 12 \CalB(g, \psi) + \frac 12 \La (v_1+u) \del_x g^+, \,\, \psi_2\Ra_{x,v}\,.
\end{aligned}
\end{equation}
Hence, this extra $\beta$-term will not affect the inf-sup estimate. 
Since $g \in (L^2(\dv\dx))^m$, we have that $f \in (L^2(e^{2\beta x}\dx; L^2(\dv)))^m$.
\end{proof}

\begin{rmk}
Note that $f - f_\infty$ for the neutron transport equations satisfy the exponential decay as $x \to \infty$ since $a(v) \sim 1$ in this case. 
\end{rmk}

\subsection{Recovery of the undamped solution} \label{sec:recovery}
Using the solution of the damped equation~\eqref{eq:half-space-damping-1}, we now explicitly construct solutions to the original undamped equation~\eqref{eq:half-space-general}. First we introduce the following notations: for any solution $f$ to the damped equation~\eqref{eq:half-space-damping-1}, denote
\begin{equation}
   \label{def:U-pm-0}
\begin{aligned}
  \VU_+ (f) & = \left(\La (v_1 + u)X_{+,1}, f \Ra_v, \,\, \cdots \,, \,\,
    \La (v_1 + u)X_{+,\nu_+}, f \Ra_v \right)^{\TT} \,, \\
  \VU_- (f) & = \left(\La (v_1 + u)X_{-,1}, f \Ra_v, \,\, \cdots \,, \,\,
    \La (v_1 + u)X_{-,\nu_-}, f \Ra_v \right)^{\TT} \,, \\
  \VU_0 (f) & = \left(\La (v_1 + u)X_{0,1}, f \Ra_v, \,\, \cdots \,, \,\, \La
    (v_1
    + u)X_{0,\nu_0}, f \Ra_v\right)^{\TT} \,, \\
  \VU_{\VecL, 0} (f)& = \left(\La (v_1 + u) \VecL^{-1}( (v_1 + u )X_{0,1}), f
    \Ra_v, \,\, \cdots \,, \,\, \La (v_1 + u) \VecL^{-1}( (v_1 + u
    )X_{0,1})X_{0,\nu_0}, f \Ra_v\right)^{\TT} \,, 
\end{aligned}
\end{equation}
and
\begin{equation}
  \VU (f)
  = \left(\VU_+^{\TT} (f), \,\, \VU_-^{\TT} (f), \,\, \VU_0^{\TT} (f), \,\, \VU_{\VecL, 0}^{\TT} (f)
  \right)^{\TT} \,. \label{def-U}
\end{equation}

Next we define some auxiliary functions. For each $1 \leq i \leq \nu_+$, let $g_{+, i}$  be the solution to~\eqref{eq:half-space-damping} with boundary conditions given by $X_{+, i}$: 
\begin{equation*}
  g_{+, i} \vert_{x=0} = X_{+, i}, \quad v_1 + u > 0. 
\end{equation*}
Similarly, for each $1 \leq j \leq \nu_0$, denote $g_{0, j}$ as the solution to
\eqref{eq:half-space-damping} where 
\begin{equation*}
  g_{0, j} \vert_{x=0} = X_{0, j}, \quad v_1 + u > 0 \,.
\end{equation*}
Let $C$ be the block matrix defined by
\begin{equation}
  C = 
  \begin{pmatrix} 
    C_{++} & C_{+0} \\
    C_{0+} & C_{00} 
  \end{pmatrix} \,, \label{def:C}
\end{equation}
where 
\begin{align*}
  & C_{++, ii'} = \langle(v_1+u)X_{+, i},g_{+, i'} \rangle \big|_{x=0},  
  && C_{+0, ij'} = \langle(v_1+u)X_{+, i},g_{0, j'} \rangle \big|_{x=0}, \\
  & C_{0+, ji'} = \langle(v_1+u)X_{0, j},g_{+, i'} \rangle \big|_{x=0}, 
  && C_{00, jj'} = \langle(v_1+u)X_{0, j},g_{0, j'} \rangle \big|_{x=0}
\end{align*}
for $1\leq i, i' \leq \nu_+$ and $1 \leq j, j' \leq \nu_0$.  In the case where $\dim H^0 = 0$, we have
\begin{align} \label{def:C-no-H-0}
   C = C_{++} \,.
\end{align}
The main property we will show about $C$ is that $C$ is non-singular. This will be an easy consequence of the following lemma:
\begin{lem} \label{lem:for-C}
Let $f$ be a solution to the damped equation~\eqref{eq:half-space-damping-1} and $\vec U(f)$ be defined as in~\eqref{def-U}. Suppose 
\begin{align} \label{cond:U-plus-0}
   \vec U_+(f) = \vec U_0(f) = 0 \,,
\qquad
  \text{at $x=0$} \,.
\end{align} 
Then $\vec U(f) = 0$ for all $x$.
\end{lem}
\begin{proof}
We separate the proof in two parts according to $\dim H^0$.
\smallskip

\noindent {\underline{\em Case 1: $\dim(H^0) = 0$}.} In this case condition~\eqref{cond:U-plus-0} reduces to 
\begin{align} \label{cond:U-plus}
    \vec U_+(f) = 0 \,,
\qquad
  \text{at $x=0$} \,.
\end{align}
Moreover, the damped equation~\eqref{eq:half-space-damping-1} reduces to
\begin{equation}
     \label{eq:half-space-damping-2}
\begin{aligned}
     (v_1 + u) \del_x f &+ \VecL  f  
     + \alpha  \sum_{k=1}^{\nu_+} 
            (v_1 + u) X_{+,k} \La (v_1 + u)X_{+,k}, f \Ra_v
\\
    &+ \alpha \sum_{k=1}^{\nu_-} 
          (v_1 + u) X_{-,k} \La (v_1 + u)X_{-,k},  f \Ra_v
          = 0 \,, 
\end{aligned}
\end{equation}
and $\VU(f)$ becomes 
\begin{equation*}
     \VU(f)
     = \left(\VU_+^{\TT}(f), \,\, \VU_-^{\TT}(f) \right)^{\TT} \,. 
\end{equation*}
Multiplying \eqref{eq:half-space-damping-2} by $X_{+,k}, X_{-,j}$ 
and integrating over $v \in \VV$, we obtain a linear system for $\VU$: 
\begin{equation}
    \del_x \VU + A_1 \VU = 0 \,, \label{ode-U-1}
\end{equation}
where the coefficient matrix is diagonal: 
\begin{equation} \label{def:A-1}
    A_1 
    =\left(
       \begin{array}{c|c}
        \begin{matrix}
               \alpha D_+\\
       \end{matrix} & \mbox{0}\\ \hline
        0 & \begin{matrix}
            - \alpha D_- \\
        \end{matrix}
\end{array}\right) \,,
\end{equation}
where $D_+, D_-$ are positive definite and  
\begin{equation*}
      D_+ = \diag(\gamma_{+,1}, \cdots, \gamma_{+,\nu_+}) \,,
\qquad
      D_- = \diag(\gamma_{-, 1}, \cdots, \gamma_{-, \nu_-}) \,,       
\end{equation*}
where $\gamma_{\pm, k} > 0$ are defined as in \eqref{notation:1}.
Since solutions to \eqref{eq:half-space-damping-2} are in $(L^2(\dv\dx))^m$, 
it is clear that 
\begin{equation*}
     \La (v_1 + u)X_{-,j}, f(0, \cdot) \Ra_v
     = \La (v_1 + u)X_{-,j}, f(x, \cdot) \Ra_v = 0 \,,
\qquad
    \text{for all $1 \leq j \leq \nu_-$ and $x \geq 0$} \,. 
\end{equation*}
Hence $\VU_-(f) = 0$ holds for all $x$. Moreover, by the structure of $A_1$ in~\eqref{def:A-1} and the initial condition~\eqref{cond:U-plus}, we have $\VU_+(f) = 0$ for all $x$. Thus $\VU(f) = 0$ for all $x$.

\noindent {\underline{\em Case 2: $\dim(H^0) \neq 0$}.}
In this case, we multiply $X_{+,j}, X_{-,i}, X_{0, k}, \CalL^{-1}(v_1 X_{0,m})$ to \eqref{eq:half-space-damping-1} and integrate over $v \in \VV$. This gives
\begin{equation}
     \del_x \VU + A_2 \VU  = 0 \,, \label{ode-U-2}
\end{equation}     
where the coefficient matrix $A_2$ is 
\begin{equation}
    A_2
    =\left(
       \begin{array}{c|c|c}
        \begin{matrix}
              \alpha D_{+} \\       
          && -\alpha D_{-} \\       \end{matrix}        
       & \mbox{0}
        & \begin{matrix}\alpha A_{21} \\  \alpha A_{22}\end{matrix}
\\ \hline
        0 & 0 
     & \alpha B
\\ \hline
       \begin{matrix}\alpha  A_{21}^{\TT} && \alpha A_{22}^{\TT} \end{matrix}& I + \alpha B & \alpha D
\end{array}\right) \,,
\end{equation}
where again $D_\pm$ are positive diagonal matrices such that
\begin{equation*}
\begin{aligned}
    D_+ = \diag(\gamma_{+,1}, \cdots, 
                          \gamma_{+, \nu_+})_{\nu_+ \times \nu_+} \,,
\qquad
    D_- = \diag(\gamma_{-,1}, \cdots, 
                      \gamma_{-, \nu_-})_{\nu_- \times \nu_-} \,.
\end{aligned}
\end{equation*}
The other matrices are
\begin{equation*}
\begin{aligned}
     A_{21, ik} 
     &= \left(\La (v_1+u) X_{+,i}, \,\, \VecL^{-1}((v_1+u) X_{0,k})\Ra_v\right)_{\nu_+ \times \nu_0} \,,
\\
     A_{22, jk} 
     &= \left(\La (v_1+u) X_{-,j}, \,\, \VecL^{-1}((v_1+u) X_{0,k})\Ra_v \right)_{\nu_- \times \nu_0} \,,
\\
     B_{ij}
     & = \La (v_1 + u) X_{0,i}, \,\, \VecL^{-1}((v_1+u) X_{0,j})\Ra_{v,\nu_0 \times \nu_0}  \,,
\\
    D_{ij}
    & =  \La (v_1 + u) \VecL^{-1}((v_1+u) X_{0,i}), \,\,
                  \VecL^{-1}((v_1+u) X_{0,j})\Ra_{v,\nu_0 \times \nu_0}  \,, 
\end{aligned}
\end{equation*}
where $B$ is symmetric positive definite and $D$ is symmetric. 
Note that if we define 
\begin{equation*}
    Q 
    =\left(
       \begin{array}{c|c|c}
        \begin{matrix}
               I \\       
          && I \\
       \end{matrix} 
       & 0
       & 0
       \\ \hline
        0 & (\alpha B)^{1/2} (I + \alpha B)^{-1/2}
     & 0
\\ \hline
       0 & 0
       & I
\end{array}\right) \,, 
\end{equation*}
and
\begin{equation*}
    \tilde A_2 
    =\left(
       \begin{array}{c|c|c}
        \begin{matrix}
               \alpha D_{+} \\       
          && - \alpha D_{-} \\
       \end{matrix} 
       & 0
       & \begin{matrix}
            \alpha A_{21} \\  \alpha A_{22}\
          \end{matrix}
       \\ \hline
        0 & 0 
     & (I + \alpha B)^{1/2} (\alpha B)^{1/2}
\\ \hline
       \begin{matrix}
            \alpha A_{21}^{\TT} &&  \alpha A_{22}^{\TT}
          \end{matrix} & (I + \alpha B)^{1/2} (\alpha B)^{1/2} 
       & \alpha D
\end{array}\right) \,. 
\end{equation*}
Then 
\begin{equation}
     A_2 = Q^{-1} \tilde A_2 Q \,.  \label{A-2-Q}
\end{equation}
Thus $A_2$ and $\tilde A_2$ have the same signature. In particular, they have the same number of negative eigenvalues. Now we count the number of negative eigenvalues of $\tilde A_2$. Let
\begin{equation*}
    P
    =\left(
       \begin{array}{c|c|c}
           \begin{matrix} I \\ &&&&& I \end{matrix}
       & \mbox{0}
       & 0
       \\ \hline
        0 & I 
     & 0
\\ \hline
       \begin{matrix} -A_{21}^{\TT} D_+^{-1} & A_{22}^{\TT} D_-^{-1}\end{matrix}& 0 & I
\end{array}\right) \,. 
\end{equation*}
Then $P$ is non-singular and
\begin{equation*}
   A_3 = P \tilde A_2 P^{\TT} 
          = \left(
       \begin{array}{c|c|c}
        \begin{matrix}
               \alpha D_{+} \\       
          && - \alpha D_{-} \\
       \end{matrix} 
       & 0
       & 0
       \\ \hline
        0 & 0 
     &  (I + \alpha B)^{1/2} (\alpha B)^{1/2}
\\ \hline
      0 &  (I + \alpha B)^{1/2} (\alpha B)^{1/2} 
       & \alpha D_1
\end{array}\right),  
\end{equation*}
where $D_1$ is symmetric and 
\begin{equation*}
      D_1 = D - A_{21}^{\TT} D_+^{-1} A_{21}
                 +  A_{22}^{\TT} D_-^{-1} A_{22} \,. 
\end{equation*}
By Sylvester's law of inertia, the matrices $\tilde A_2$ and $A_3$, thus $A_2$ 
and $A_3$, have the same number of negative eigenvalues. The total number of 
negative eigenvalues of $A_3$ is determined by that of the submatrix 
$\begin{pmatrix} 0 &  (I + \alpha B)^{1/2} (\alpha B)^{1/2} \\  (I + \alpha B)^{1/2} (\alpha B)^{1/2} & \alpha D_1\end{pmatrix}$
. Define  
\begin{equation*}
     P_1 
     = \begin{pmatrix}
           (I + \alpha B)^{-1/4}(\alpha B)^{-1/4} & 0 \\
           0 & (I + \alpha B)^{-1/4}(\alpha B)^{-1/4}
        \end{pmatrix} \,. 
\end{equation*}
Then
\begin{equation*}
     A_4 = P_1 
     \begin{pmatrix} 
     0 & (I + \alpha B)^{1/2} (\alpha B)^{1/2} \\ 
     (I + \alpha B)^{1/2} (\alpha B)^{1/2} & \alpha D_1
     \end{pmatrix} P_1^{\TT}
     = \begin{pmatrix}
           0 & I \\
           I & \alpha D_2
     \end{pmatrix} \,, 
\end{equation*}
where 
\[ D_2 = (I + \alpha B)^{-1/4}(\alpha B)^{-1/4} D_1 
              (I + \alpha B)^{-1/4}(\alpha B)^{-1/4} \,. \] 
Note that $D_2$ is symmetric.
Hence, $D_2$ has a complete set of eigenvectors. Let 
$(\lambda, \bd{E}) = (\lambda, (\bd{e}_1, \bd{e}_2)^{\TT})$ be an  eigenpair of 
$A_4$ such that
\begin{equation}
    \begin{pmatrix}
         0 & I \\
         I & \alpha D_2 
    \end{pmatrix}
    \begin{pmatrix}
           \bd{e}_1^{\TT} \\
           \bd{e}_2^{\TT}
    \end{pmatrix}
    = \lambda  
    \begin{pmatrix}
           \bd{e}_1^{\TT} \\
           \bd{e}_2^{\TT}
    \end{pmatrix} \,.
\end{equation}
This is equivalent to 
\begin{equation*}
    {\bd{e}_2} = \lambda {\bd{e}_1} \,,
\qquad
     {\bd{e}_1}^{\TT} + {\alpha D_2} \bd{e}_2^{\TT} = \lambda \bd{e}_2^{\TT} \,. 
\end{equation*}
Note that $\lambda \neq 0$. Since $D_2$ is symmetric, it has a complete set of orthogonal eigenvectors. Let $\bd{e}$ be an arbitrary eigenvector of $D$ with eigenvalue $\lambda_{\bd e}$ and take $\bd{e}_2 = \bd{e}$. Then
\begin{equation}
    {\bd{e}_1} = \frac{1}{\lambda} {\bd{e}} \,,
\qquad
     \frac{1}{\lambda} {\bd{e}}^{\TT} + \alpha D_2 {\bd e}^{\TT} =  \lambda {\bd e}^{\TT} \,. 
\end{equation}
Thus 
\begin{equation}
     \frac{1}{\lambda} + \alpha \lambda_{\bd e} - \lambda = 0 \,,
\end{equation}
which has exactly one negative solution for $\lambda$. Since the set of 
eigenvectors of $D_2$ is complete, the matrix $A_4$ has exactly $\nu_0$ 
negative eigenvalues. Together with $D_-$, we have that $A_3$, thus $A_2$, has 
exactly $\nu_- + \nu_0$ negative eigenvalues, which prescribes $\nu_- + \nu_0$ 
conditions on $\VU$ such that  
\begin{equation}
      {\bd E}_k \cdot \VU(x) = 0 \,, \qquad 1 \leq k \leq \nu_- + \nu_0 \,, \quad x \geq 0 \,,
      \label{eq:9}
\end{equation}
where ${\bd E}_k$ are the eigenvectors associated with negative eigenvalues. 
Write each ${\bd E}_k$ as
\begin{equation*}
     {\bd E}_k = ({\bd e}_{k,+}, \,\, {\bd e}_{k,-}, \,\, {\bd e}_{k,0}, \,\, {\bd e}_{k, \CalL, 0})^{\TT}
\end{equation*}
and define the matrix $\bd E$ by
\begin{equation*}
      {\bd E} 
      = \begin{pmatrix}
             {\bd e}_{1, -} && {\bd e}_{1, \CalL, 0} \\
                & \cdots \\
             {\bd e}_{\nu_- + \nu_0, -} && {\bd e}_{\nu_- + \nu_0, \CalL, 0}  
         \end{pmatrix}_{(\nu_-+\nu_0) \times (\nu_- + \nu_0)} 
\end{equation*}
By~\eqref{cond:U-plus-0} we have
\begin{equation} \label{eq:U-minus-L}
      {\bd E}
      \begin{pmatrix}
           \vec U_- \\[10pt]
            \vec U_{\CalL, 0} 
      \end{pmatrix}
    = 0 \,,  \qquad \text{at $x=0$}.
\end{equation}
Now we show that $\bd E$ is nonsingular. Suppose not. Let $\mathcal{N}_2$ be the space spanned by the eigenvectors of $A_2$ with negative eigenvalues. Then there exists a nontrivial vector in $\mathcal{N}_2$ which takes the form 
\[ \hat {\bd E} = (\hat{\bd e}_+, 0, \hat{\bd e}_0, 0)^{\TT} \,. \]
By \eqref{A-2-Q}, if ${\bd E} = ({\bd e}_+, {\bd e}_-, {\bd e}_0, {\bd e}_{\CalL, 0})^{\TT}$ is an eigenvector of $A_2$ with eigenvalue $\lambda$, then ${\bd F} = Q({\bd e}_+, {\bd e}_-, {\bd e}_0, {\bd e}_{\CalL, 0})^{\TT}$ is an eigenvector of $\tilde A_2$ with the same eigenvalue. By the definition of $Q$, if we denote 
\[ {\bd F} = ({\bd f}_+, {\bd f}_-, {\bd f}_0, {\bd f}_{\CalL, 0})^{\TT} \,, \]
then 
\[ {\bd e}_- = {\bd f}_-, \qquad {\bd e}_{\CalL, 0} = {\bd f} _{\CalL, 0} \,.\]
Let $\mathcal{QN}_2$ as the space spanned by the eigenvectors of $\tilde A_2$ with negative eigenvalues. Then there exists a nontrivial $\hat {\bd F} \in \mathcal{QN}_2$ such that
\[ \hat {\bd F} = (\hat{\bd f}_+, 0, \hat{\bd f}_0, 0)^{\TT} \,. \]
Since $\mathcal{QN}_2$ is an invariant subspace of $\tilde A_2$, we have that
\begin{equation*}
      \tilde A_2 \hat {\bd F} 
      = (\alpha D_+ \hat{\bd f}_+^{\TT}, 0, 0, \hat{\bd f}_2)
           \in \mathcal{QN}_2 \,
\end{equation*}
where $\hat{\bd f}_2 = \alpha A_{21}^{\TT} \hat{\bd f}_+^{\TT} + 
          (I + \alpha B)^{1/2} (\alpha B)^{1/2} )^{\TT} \hat{\bd f}_0^{\TT}$.
By the symmetry and non-degeneracy of $\tilde A_2$, the quadratic form given by $\tilde A_2$ on $\mathcal{QN}_2$ is strictly negative. Therefore, 
\begin{equation*}
      \hat{\bd F}^{\TT} \tilde A_2 \hat{\bd F} = \alpha \hat{\bd f}_+^{\TT} D_+ \hat{\bd f}_+ \leq 0 \,. 
\end{equation*}
Since $D_+$ is strictly positive definite, we have that $\hat{\bd f}_+ = 0$ and
\[ \hat{\bd F}^{\TT} \tilde A_2 \hat{\bd F} = 0 \,,\]
which implies that $\hat{\bd F} = 0$. This contradicts the assumption that $\hat{\bd F}$ is non-trivial. Hence the matrix $\bd E$ is non-singular. By~\eqref{eq:U-minus-L} we derive that
\begin{align*}
   \VU_-(f) = 0 \,,
\qquad
   \VU_{\CalL, 0}(f) = 0 \,,
\qquad
   \text{at $x=0$}.
\end{align*}
Together with~\eqref{cond:U-plus-0}, we have the initial data for the ODE~\eqref{ode-U-2} as $\VU(f) = 0$ at $x=0$. Thus the only solution to this ODE is $\VU(f) = 0$ for all $x$. 
\end{proof}

Using Lemma~\ref{lem:for-C} we can now show
\begin{lem}
The matrix $C$ defined in~\eqref{def:C} is non-singular. 
\end{lem}
\begin{proof}
First we recall \cite{CoronGolseSulem:88} the uniqueness property of the solution to \eqref{eq:half-space-full}: if $f$ is a solution to \eqref{eq:half-space-full} which satisfies $f \in (L^2(a\dv\dx))^m$ and $(v_1+u) \del_x f \in (L^2(\tfrac{1}{a} \dv\dx))^m$, then $f$ must be unique.  For the convenience of the reader, we brief explain its proof: Suppose $h$ is a solution to the half-space equation \eqref{eq:half-space-full} with incoming data $\phi=0$. Then $\int_{\R^3} (v_1+u) h^2 \dv$ is decreasing in $x$. Since there exists $h_\infty \in H^+ \oplus H^0$ such that $h - h_\infty \in (L^2(\dv\dx))^m$, we can find a sequence $x_k$ such that 
\begin{equation}\nn
    \int_{\R^3} (v_1+u) h^2(x_k, v) \dv \to \int_{\R^3} (v_1+u) h_\infty^2(v) \dv \geq 0 \,.
\end{equation}
Hence $\int_{\R^3} (v_1+u) h^2(x, v) \dv \geq 0$ for all $x \geq 0$. This holds in particular at $x=0$. Since the incoming data is zero at $x=0$,  the outgoing data at $x=0$ must also be zero and $\int_{\R^3} (v_1+u) h^2(x_k, v) \dv = 0$ for all $x \geq 0$. The conservation property of the half-space equation then implies that $h(x, \cdot) \in (\NullL)^\perp$. By multiplying  the equation by $h$ and integrate over $v$, we have $\La h, Lh \Ra = 0$ for all $x \geq 0$. Hence $\tilde\CalP h =0$ for all $x \geq 0$ by the spectral gap of $\CalL$ in (A4). Therefore $h \equiv 0$ and the solution to the half-space equation is unique. 

Now suppose $C$ is singular. Then there exist constants
\[ (\eta_{+,1}, \cdots, \eta_{+,\nu_+}, \eta_{0,1}, \cdots, \eta_{0,\nu_0}) \neq 0 \] such that we can find incoming data 
\begin{equation*}
     \Vphi_{g} = \sum_{j=1}^{\nu_+}\eta_{+,j} X_{+,j} + \sum_{k=1}^{\nu_0} \eta_{0,k} X_{0,k}
\end{equation*}
that gives rise to a solution $\Vg$ satisfying that
\begin{equation}
   \label{eq:11}
\begin{aligned}  
    &\La (v_1+u) X_{+,1}, \,\, \Vg \Ra_v  
    = \cdots
    = \La (v_1+u) X_{+,\nu_+}, \,\, \Vg \Ra_v  
    = 0 \,, 
\\    
    &\La (v_1+u) X_{0,1}, \,\,  \Vg \Ra_v
    = \cdots
    =  \La (v_1+u) X_{0,\nu_0}, \,\,  \Vg \Ra_v
    = 0 \,, 
\end{aligned}
\qquad \text{at $x=0$}.
\end{equation}
By Lemma~\ref{lem:for-C}, we have
\begin{equation}
   \VU(g) = 0
\qquad
   \text{for all $x$} \,.
\end{equation}
Thus the solution $\Vg$ satisfies both the damped and the original half-space equation \eqref{eq:half-space-1} with the end-state $g_\infty = 0$. By the uniqueness of solutions to \eqref{eq:half-space-1}, we have $\eta_{+,1} = \cdots = \eta_{+, \nu_+} = \eta_{0,1} = \cdots = \eta_{0,\nu_0} = 0$ which is a contraction. Thus $C$ must be non-singular. 
\end{proof}

%

Now we state and prove the main recovery theorem. 
\begin{prop}[Recovery]\label{prop:recovery}
  Let $\Vphi \in (L^2(a(v) \Id_{v_1+u>0}\dv))^m$ and $f$ be the solution to the damped equation~\eqref{eq:half-space-damping-1} with incoming data $\phi$. Let $C, g_{+, i}, g_{0, j}$ be the matrix and the family of auxiliary functions defined in~\eqref{def:C-no-H-0} and~\eqref{def:C}. Define the coefficient vector $\eta = \left(\eta_{+,1},\cdots \eta_{+, \nu_+},
  \eta_{0,1}, \cdots,\eta_{0,\nu_0}\right)^{\TT}$ such that 
\begin{equation}\label{eq:linearsysC}
  \eta = C^{-1 }(\VU_+(f), \VU_0(f))^{\TT} \Big|_{x=0} 
\end{equation}
and
\begin{align} \label{def:g-Phi}
   g = \sum_{i=1}^{\nu_+} \eta_{+,i} g_{+,i} + \sum_{j=1}^{\nu_0} \eta_{0,j} g_{0,j} \,,
\qquad
   \Phi = \sum_{i=1}^{\nu_+} \eta_{+,i} X_{+,i} + \sum_{j=1}^{\nu_0} \eta_{0,j} X_{0,j} \,.
\end{align}
Define
\begin{equation} \label{soln:approx}
   f_{\phi} 
   = f - g + \Phi 
   = f- \sum_{i=1}^{\nu_+} \eta_{+,i} (g_{+, i} - X_{+,i}) - \sum_{j=1}^{\nu_0} 
    \eta_{0,j} ( g_{0, j} - X_{0,j}) \,.
  \end{equation}
Then $f_\phi$ is the unique solution to the half-space
  equation 
\begin{equation} \label{eq:half-space-full}
  \begin{aligned}
     (v_1 + u) \del_x &f_\phi + {\mathcal{L}} f_\phi =  0, \\
    \Vecf_\phi|_{x=0} &=  \Vecphi(v) \,,  \qquad &&v_1 + u > 0 \,,
\\
      f_\phi - f_{\phi, \infty} \in L^2 &(\dx; L^2(\dv))^m \,,
\end{aligned}
\end{equation}
  where $f_{\phi, \infty} \in H^+ \oplus H^0$ is the end-state given by
  \begin{equation*}
  f_{\phi, \infty} = \sum_{j=1}^{\nu_+} \eta_{+,j} X_{+,j} + \sum_{k=1}^{\nu_0} \eta_{0,k} X_{0,k} \,. 
  \end{equation*}
\end{prop}
\begin{proof}
We directly show that $f_\phi$ satisfies~\eqref{eq:half-space-full}. First, by the definitions of $g_{+, i}, g_{0, j}$, we have $\Vecf_\phi|_{x=0} =  \Vecphi(v)$ for $v_1 + u > 0$. Second, 
it follows from the definition in~\eqref{def:g-Phi} that $g$, thus $f - g$, are both solutions to the damped equation~\eqref{eq:half-space-damping-1}.
By the definition of $\eta$ we have
\begin{align*}
    \VU_+(f - g) = 0 \,,
\qquad
   \VU_0(f - g) = 0 \,.
\end{align*}
Hence by Lemma~\ref{lem:for-C}, we have $\VU(f - g) = 0$. This shows $f - g$ is in fact a solution to the undamped equation~\eqref{eq:half-space-full}. Since every $X_{+, i}$ and $X_{0, j}$ are solutions to~\eqref{eq:half-space-full}, we have $f_\phi$ as a solution to~\eqref{eq:half-space-full}.  
\end{proof}

\section{Galerkin Approximation and Numerical Scheme}
Let us now use the variational formulation \eqref{variational} to
design a Galerkin method to approximate the solution to the damped
equation \eqref{eq:half-space-damping-1}. There are two parts in this
section: first we show the construction of the finite-dimensional
approximation and its error estimate. Then we transform the
finite-dimensional variational form into an ODE system which will set
base for our numerical scheme.

\subsection{Galerkin approximation.} First we use both parts of the Babu\v{s}ka-Aziz lemma to show the validity of the Galerkin approximation and its quasi-optimality. 

\begin{prop}[Approximations in $\R^3$]\label{prop:approximation-NK}
Suppose $\{\psi_n^{(1)}\}_{n=1}^\infty$ is an orthonormal basis of $L^2(\!\dv_1)$ such that
\begin{itemize}
\item $\psi_{2n-1}^{(1)}(v_1)$ is odd and $\psi_{2n}^{(1)}(v_1)$ is even in $v_1$ with respect to $-u$ for any $n \geq 1$;


\item $(v_1+u) \psi_{2n}^{(1)}(v_1) \in \Span\{\psi_1^{(1)}, \cdots, \psi_{2n+1}^{(1)}\}$ for each $n \geq 1$.
\end{itemize}
Suppose $\{\psi_n^{(2)}\}, \{\psi_n^{(3)}\}_{n=1}^\infty$ are orthonormal bases for $L^2(\!\dv_2)$ and $L^2(\!\dv_3)$ respectively. 
Define the closed subspace $\Gamma_{NK}$ as
\begin{equation*}
   \Gamma_{NK} = \left\{g(x, v) \in \Gamma \Big| \, g(x, v) = \sum_{i=1}^m\sum_{l,n =1}^K\sum_{k=1}^{2N+1}  g_{kln}^{(i)}(x) \psi_k^{(1)}
(v_1)\psi_l^{(2)}(v_2)\psi_n^{(3)}(v_3) \, \Be_i, \,\, g_{kln}^{(i)} \in H^1(\!\dx) \right\}
   \,,
\end{equation*} 
where $\Be_i = (0, \cdots, 0, 1, 0, \cdots, 0)^T$ is the standard $i^{th}$ basis vector of $\mathbb{R}^m$ with $1 \leq i \leq m$. 
Then 

\Ni (a) there exists a unique $f_{NK} \in \Gamma_{NK}$ such that  
\begin{equation}\label{def:f-NK}
     f_{NK}(x, v) 
     =  \sum_{i=1}^m \sum_{l,n =1}^K\sum_{k=1}^{2N+1}  a_{kln}^{(i)}(x) \psi_k^{(1)}(v_1)\psi_l^{(2)}(v_2)\psi_n^{(3)}(v_3) \, \Be_i \,, 
\end{equation}
which satisfies 
\begin{equation}\label{eq:variational-finitedim-NK}
      \CalB(f_{NK}, g) = l(g) \quad \text{for every $g \in \Gamma_{NK}$} \,, 
\end{equation}
where $\CalB$ and $l$ for the damped equation and are defined in \eqref{def-B} and \eqref{def-l} respectively. The coefficients $\{a_{kln}^{(i)}(x)\}$ satisfy
\begin{equation*}
     a_{kln}^{(i)}(\cdot) \in C^1[0, \infty) \cap H^1(0, \infty), \qquad 
1 \leq k \leq 2N+1 \,,  \,\, 1 \leq l,n \leq K \,,
\,\, 1 \leq i \leq m \,. 
\end{equation*}

\Ni (b) There exists a constant $C_0$ such that
\begin{equation*}
    \|f - f_{NK}\|_\Gamma \leq C_0 \inf_{w \in \Gamma_{NK}} \|f - w\|_{\Gamma} \,, 
\end{equation*}
where $\norm{\cdot}_\Gamma$ is the norm defined in \eqref{def-Gamma}.
\end{prop}
\begin{proof}
Both (a) and (b) directly follow from the Babu\v{s}ka-Aziz lemma as long as we verify the inf-sup condition of $\CalB$ on the finite-dimensional subspace $\Gamma_{NK}$. Since it is similar as the continuum case in Proposition~\ref{inf-sup}, we only explain the modification in choosing the test functions $\psi_1$ and $\psi_2$. For any $\Vf \in \Gamma_N$, we choose 
\begin{equation*}    
    \psi_1 = \Vf \,, 
\qquad 
    \psi_2 = \CalP_N \left(\frac{1}{(1 + |v_1 + u| + |v_2| + |v_3|)^{\omega_0}} (v_1+u) \del_x \Vf^+ \right) \,,
\end{equation*}    
where $\CalP_N: (L^2(\dv))^m \to \Gamma_N$ is the projection onto $\Gamma_N$. The rest of the estimates are similar to the proof in Section 3, and thus omitted. 
\end{proof}

Since our numerical examples are both in one-dimension for a single species, we apply Proposition \ref{prop:approximation-NK} to $\VV \subseteq \R^1$ and $m=1$ to obtain the following corollary for two special cases:

\begin{cor}[Approximations in $\R^1$]\label{prop:approximation}
Let $\VV = \R^1$ or $\VV = [-1, 1]$. Let $u \in \R$ be arbitrary if $\VV = \R^1$ and $u = 0$ if $\VV= [-1, 1]$. Suppose $\{\psi_n\}_{n=1}^\infty$ is an orthonormal basis of $L^2(\dv)$ such that
\begin{itemize}
\item $\psi_{2n-1}$ is odd and $\psi_{2n}$ is even in $v$ with respect to $-u$ for any $n \geq 1$;


\item $(v+u) \psi_{2n}(v) \in \Span\{\psi_1, \cdots, \psi_{2n+1}\}$ for each $n \geq 1$.
\end{itemize}
 
Define the closed subspace $\Gamma_N$ as
\begin{equation*}
   \Gamma_N = \left\{g(x, v) \in \Gamma \Big| \, g(x, v) = \sum_{k=1}^{2N+1}  g_k(x) \psi_k(v), \,\, g_k \in H^1(\dx) \right\}
   \,. 
\end{equation*} 
Then there exists a unique $f_N \in \Gamma_N$ such that  
\begin{equation}\label{def:f-N}
     f_N(x, v) =  \sum_{k=1}^{2N+1} a_k(x) \psi_k (v) \,, 
\qquad a_k(x) \in C^1[0, \infty), \,\, 1 \leq k \leq 2N+1 \,,
\end{equation}
which satisfies 
\begin{equation}\label{eq:variational-finitedim}
      \CalB(f_N, g) = l(g) \,, \qquad \text{for every $g \in \Gamma_N$} \,, 
\end{equation}
where $\CalB$ and $l$ are defined in \eqref{def-B} and \eqref{def-l}
respectively.  
\end{cor}

The approximate solution to the undamped solution is constructed similarly as for the continuous case: let $C, g_{+, i}, g_{0, j}$ be the same matrix and auxiliary functions as in~\eqref{def:C-no-H-0} and~\eqref{def:C}. Let $g_{+, NK}^{(i)}, g_{0, NK}^{(j)}$ be the Galerkin approximate solutions to $g_{+, i}$ and $g_{0, j}$ respectively. Let
\begin{equation}\label{eq:linearsysC-NK}
\begin{aligned}
  \eta_{NK} 
  &= \left(\eta_{+, NK}^{(1)},\cdots \eta_{+, NK}^{( \nu_+)},
  \eta_{0,NK}^{(1)}, \cdots,\eta_{0,NK}^{(\nu_0)}\right)^{\TT}
  = C^{-1 }(\VU_+(f_{NK}), \VU_0(f_{NK}))^{\TT} \Big|_{x=0} \,,
\\
  g_{NK} 
  = &\sum_{i=1}^{\nu_+} \eta^{(i)}_{+, NK} g_{+, NK}^{(i)}
    + \sum_{i=1}^{\nu_0} \eta^{(j)}_{0, NK} g_{0, NK}^{(j)}  \,,
\qquad
  \Phi_{NK} 
  = \sum_{i=1}^{\nu_+} \eta^{(i)}_{+, NK} X_{+, i}
    + \sum_{i=1}^{\nu_0} \eta^{(j)}_{0, NK} X_{0, j}  \,,
\end{aligned}
\end{equation}
Let $f_\phi$ be the solution to the undamped half-space equation~\eqref{eq:half-space-full}. Define its approximation $f_{\phi, NK}$ as
\begin{align} \label{soln:approx-NK}
   f_{\phi, NK} = f_{\phi} - g_{NK} + \Phi_{NK} \,,
\end{align}
which is an analog of the continuous version in~\eqref{soln:approx}. The following proposition shows the above approximation is almost quasi-optimal with a correction term.


\begin{prop} \label{prop:quasi-optimality-undamped} Let $f_\phi$ be the solution to the undamped half-space equation~\eqref{eq:half-space-full}. Suppose $f_{\phi,NK}$ is constructed as in~\eqref{soln:approx-NK}. Suppose $f_\phi$ is the unique solution to the equation
  \eqref{eq:half-space-2}.  Then there exists a constant $C_0$ such
  that
  \begin{equation*}
    \|f_\phi - f_{\phi, NK}\|_\Gamma 
    \leq 
    C_0 \left(\inf_{w \in \Gamma_N} \|f_\phi - w\|_{\Gamma} 
      + \inf_{w \in \Gamma_N} \| f - w\|_\Gamma
      + \delta_N \|f\|_{(L^2(a\dv\dx))^m} \right)\,, 
  \end{equation*}
  where $\norm{\cdot}_\Gamma$ is the norm defined in \eqref{def-Gamma}
  and
  \begin{equation*}
    \delta_N := \sum_{i=1}^{\nu_+} \inf_{w \in \Gamma_N} \|g_{+, i} - w\|_\Gamma
    + \sum_{j=1}^{\nu_0} \inf_{w \in \Gamma_N} \|g_{0, j} - w\|_\Gamma \,. 
  \end{equation*}
\end{prop}
\begin{proof}
Let $f$ be the solution the damped equation~\eqref{eq:half-space-damping-1} with incoming data $\phi$. Let $g, \Phi$ be defined as in~\eqref{def:g-Phi}. 
Then there exist constants $\kappa_4, \tilde\kappa_4>0$ such that\begin{equation*}
\begin{aligned}
      \|\Phi - \Phi_{NK}\|_\Gamma 
      &= \|\Phi - \Phi_{NK}\|_{(L^2(\dv))^m} 
\\
      & \leq \kappa_4 
                \| f - f_N\|_\Gamma
                       + \tilde\kappa_4 \|f\|_{(L^2(\dv\dx))^m} 
                       \left(\sum_{i=1}^{\nu_+} \|g_{+, i} - g_{+, NK}^{(i)}\|_\Gamma
                       + \sum_{j=1}^{\nu_0} \|g_{0, j} - g_{0,NK}^{(j)}|_\Gamma \right) \,,
\end{aligned}
\end{equation*}
Second, since $f - g$ is a solution to the damped equation, we have
\begin{equation*}
    \|(f - g) - (f_N - g_{NK})\|_\Gamma \leq \kappa_5 \inf_{w \in \Gamma_{NK}} \|w - (f -  g)\|_{\Gamma} \,, 
\end{equation*}
since $f_{NK}, g_{NK} \in \Gamma_{NK}$. Therefore,
\begin{equation*}
\begin{aligned}
  \|(f - g + \Phi) - (f_{NK} - g_{NK} +
  \Phi_{NK})\|_\Gamma 
& \leq \kappa_5 \inf_{w \in \Gamma_{NK}} \|w - (f -
  g)\|_{\Gamma} + \|\Phi - \Phi_{NK}\|_\Gamma
\\
& \hspace{-1cm}
  \leq \kappa_5 \inf_{w \in \Gamma_{NK}} \|w - (f - g + \tilde
  \Phi)\|_{\Gamma} + \kappa_4 \|f - f_{NK}\|_\Gamma
\\
& \hspace{-1cm}
   \leq \kappa_6 \left(\inf_{w \in \Gamma_{NK}} \|f_\phi - w\|_{\Gamma} + \inf_{w \in \Gamma_{NK}} \|f - w\|_\Gamma 
     + \delta_{NK} \|f\|_{(L^2(\dv\dx))^m}\right)
  \,,
\end{aligned}
\end{equation*}
where
\begin{equation} \nn
       \delta_{NK} = 
             \sum_{i=1}^{\nu_+} \inf_{w \in \Gamma_{NK}} \|g_{+, i} - w\|_\Gamma
        + \sum_{j=1}^{\nu_0} \inf_{w \in \Gamma_{NK}} \|g_{0, j} - w\|_\Gamma \,.
\end{equation}
Note that the second inequality holds because $\Phi \in H^+ \oplus H^0 \subseteq \Gamma_{NK}$.
\end{proof}

\begin{rmk}
Note that in the above reconstruction scheme, the solutions
$g_{+, NK}^{(i)}$ for $1 \leq i \leq \nu_+$ and $g_{0, NK}^{(j)}$ for $1 \leq j \leq
\nu_0$ can be precomputed, as they do not depend on the prescribed
incoming data $\phi$. In particular, we can use a higher order
approximation (larger $N, K$) for these functions.
\end{rmk}

\subsection{ODE formulation}
In this part we reformulate the variational form \eqref{eq:variational-finitedim-NK} into an ODE with explicit boundary conditions. This ODE will be the system that we solve in numerics; since this is a linear ODE, its solution can be directly obtained by solving the associated generalized eigenvalue problems. 
To illustrate the idea, we first treat the special case where there is a single species in 1D, that is, $m=K=1$.
\begin{prop} \label{prop:1D-ODE}
The variational form
\eqref{eq:variational-finitedim} is equivalent to the following ODE
for the coefficients $a_k(x)$ together with the boundary conditions at
$x=0$:
\begin{align}
  \label{ODE-1D}
  \sum_{k=1}^{2N+1} \mathsf{A}_{kl} \del_x a_k(x) &=
  \sum_{k=1}^{2N+1}\mathsf{B}_{kl} a_k(x) \,,
  \\
  \label{cond:boundary}
   \sum_{k=1}^{N+1} \La (v+u) \psi_{2k-1}, \, \psi_{2j} \Ra_v a_{2k-1}(0)
   + &\sum_{k=1}^{N} \La |v+u| \psi_{2k}, \psi_{2j}\Ra_v a_{2k}(0)
   = 2 \int_{v+u > 0} (v_1+u) \, \phi \, \psi_{2j} \dv  \,,
\end{align}   
where $1 \leq j \leq N$ and
\begin{equation} \label{def:A-B}
  \mathsf{A}_{kl}  = \La (v+u) \psi_k, \,\, \psi_l \Ra_v \,,
  \qquad
  \mathsf{B}_{kl} = -\La \psi_k, \,\, \mc{L}_d \psi_l \Ra_v \,, 
  \qquad
  1 \leq i,j \leq 2N+1 \,.
\end{equation}
\end{prop}
\begin{proof}
In order to show that the boundary conditions for the solution to \eqref{eq:variational-finitedim} are given by \eqref{cond:boundary}, we first choose test functions $G_{2j}(x, v) = g(x)\psi_{2j}(v)$ where $g(x) \in C_c^\infty([0, \infty))$ and $1 \leq j \leq N$. Applying $G_j$ in \eqref{eq:variational-finitedim}, we get 
\begin{equation}
  \label{eq:f-N-1}
\begin{aligned}  
       - \La f_N^{-}, \,\, (v_1+u) \psi_{2j}(v) \del_x g(x)  \Ra_{x,v} + \La (\CalL\psi_{2j}) g(x),  \,\, f_N\Ra_{x,v}
    = 0 \,,
\end{aligned}
\end{equation}
where $f_N$ is defined in \eqref{def:f-N} and $f_N = f_N^{-} + f_N^{+}$ with
\begin{equation*} 
      f_N^{-} = \sum_{k=1}^{N+1} a_{2k-1} (x) \psi_{2k-1} \,,
\qquad
      f_N^{+} = \sum_{k=1}^N a_{2k} (x) \psi_{2k} \,. 
\end{equation*}
By integration by parts in \eqref{eq:f-N-1} we obtain
\begin{equation*}
\begin{aligned}  
       \La \sum_{k=1}^{N+1} \psi_{2k-1}\del_x a_{2k-1} (x) , \,\, (v_1+u) \psi_{2j}(v) g(x)  \Ra_{x,v} + \La (\CalL\psi_{2j}) g(x),  \,\, f_N\Ra_{x,v}
    = 0 \,.
\end{aligned}
\end{equation*}
Since $g \in C_c^\infty([0, \infty))$ is arbitrary, we have
\begin{equation}
  \label{eq:f-N-2}
\begin{aligned}  
       \sum_{k=1}^{N+1} \La \psi_{2k-1} , \,\, (v_1+u) \psi_{2j}(v)  \Ra_{v} \del_x a_{2k-1} (x) + \La (\CalL\psi_{2j}),  \,\, f_N\Ra_{v}
    = 0 \,,
\end{aligned}
\end{equation}
for each $1 \leq j \leq N$ and $x \in [0, \infty)$. Note we choose $\tilde G_{2j} = \tilde g (x) \psi_{2j}(x)$ where $\tilde g \in C^\infty([0, \infty))$. Then equation \eqref{eq:variational-finitedim} becomes 
\begin{equation}
  \label{eq:f-N-3}
\begin{aligned}  
       - \La f_N^{-}, \,\, (v+u) \psi_{2j}(v) \del_x g(x)  \Ra_{x,v} + \La (\CalL\psi_{2j}) g(x),  \,\, f_N\Ra_{x,v}
     + \La (v_1+u) f_N^{+}, \,\, \psi_{2j} \tilde g(0)\Ra_{x=0} 
\\
     = 2 \int_{v_1+u > 0} (v_1+u) \phi(v) \psi_{2j}(v) g(0) \dv  \,,
\end{aligned}
\end{equation}
for each $1 \leq j \leq N$. 
The set of $N$ boundary conditions \eqref{cond:boundary} then follows from integrating by parts in \eqref{eq:f-N-3} and applying \eqref{eq:f-N-2}. 
\end{proof}

The general case follows from the similar idea and we only sketch its proof.
\begin{prop}
Let 
\begin{equation*}
      \mathsf{A} = \left(\La (v_1 + u) \psi_k^{(1)}, \,\, \psi_j^{(1)}\Ra_{v_1} \right)_{(2N+1)\times (2N+1)} \,.
\end{equation*}
Define two 8-tensors $\mathfrak{A}$ and $\mathfrak{B}$ as
\begin{equation}\label{def:D-B}
\begin{aligned}
      \mathfrak{A} &= \mathsf{A} \otimes \ii \otimes \ii \otimes \ii = \big(\mathsf{A}_{ik} \delta_{lj} \delta_{ns} \delta_{pq}\big)_{(2N+1)^2 \times K^2 \times K^2 \times m^2}\,,
\\
      \mathfrak{B}_{klnp}^{ijsq}
      &= - \La \psi_k^{(1)}(v_1)\psi_l^{(2)}(v_2)\psi_n^{(3)}(v_3) \, \Be_p,  \,\,
         \CalL_d \left(\psi_i^{(1)}(v_1)\psi_j^{(2)}(v_2)\psi_s^{(3)}(v_3) \, \Be_q \right)\Ra_v 
\end{aligned}
\end{equation}
for $1 \leq i,k \leq 2N+1$, $1 \leq j, l \leq K$, $1 \leq s,n \leq K$, and $1 \leq p, q \leq m$.
Then the variational form \eqref{eq:variational-finitedim-NK} is equivalent to the  following ODE for the coefficients $a_{kln}^{(p)}(x)$:
\begin{equation} \label{ODE-general}
        \sum_{p=1}^m\sum_{l,n=1}^K \sum_{k=1}^{2N+1}\mathfrak{A}_{klnp}^{ijsq} \del_x a_{kln}^{(p)}(x) 
        = \sum_{p=1}^m \sum_{l,n=1}^K \sum_{k=1}^{2N+1}\mathfrak{B}_{klnp}^{ijsq} a_{kln}^{(p)}(x),
\end{equation}
together with the boundary conditions at $x=0$:
\begin{equation}\label{Boundary-general}
\begin{aligned}
	\sum_{k=1}^{N+1} \La (v_1+u) \psi_{2k-1}^{(1)}, \, \psi_{2i}^{(1)} \Ra_{v_1} a_{2k-1,jl}^{(q)}(0)
   &+ \sum_{k=1}^{N} \La |v_1+u| \psi_{2k}^{(1)}, \psi_{2i}^{(1)}\Ra_{v_1} a_{2k,jl}^{(q)}(0)
\\
  &= 2 \int_{v_1+u > 0} (v_1+u) \, \phi \cdot \, \psi_{2i}^{(1)}(v_1) \psi_{j}^{(2)}(v_2)\psi_{k}^{(3)}(v_3) \Be_q\dv  
\end{aligned}
\end{equation}
for $i = 1, \cdots, N$, $j,l = 1, 2, \cdots, K$, and $q = 1, \cdots, m$. 

\end{prop}

\begin{proof}
Equation~\eqref{ODE-general} is obtained by choosing the test function $g$ in~\eqref{eq:variational-finitedim-NK} as the basis functions such that the velocity part is $\psi_i^{(1)}(v_1)\psi_j^{(2)}(v_2)\psi_l^{(3)}(v_3) \, \Be_q$. The boundary condition~\eqref{Boundary-general} is derived by choosing the test functions as $\psi_{2i}^{(1)}(v_1)\psi_j^{(2)}(v_2)\psi_l^{(3)}(v_3) \, \Be_q$. 
\end{proof}


Numerically, the approximate solutions $f_{NK}$ in~\eqref{ODE-general} (or $f_N$~\eqref{ODE-1D} in 1D) will be solved using the method of generalized eigenvalues. In particular, we define the generalized eigenvalues
  and its associated eigen-tensor for $(\mathfrak{A}, \mathfrak{B})$ as $\lambda\in\R$ and $\eta =
  (\eta_{kln}^{(p)})_{(2N+1)\times K\times K \times m}$ such that
\begin{equation}\label{BD-eigenvalue}
      \mathfrak{A}\eta 
      = \sum_{p=1}^m \sum_{l,n=1}^K \sum_{k=1}^{2N+1} \mathfrak{A}_{klnp}^{ijsq}\,\eta_{kln}^{(p)} 
      = \lambda  \sum_{p=1}^m \sum_{l,n=1}^K \sum_{k=1}^{2N+1} \mathfrak{B}_{klnp}^{ijsq}\,\eta_{kln}^{(p)}
\end{equation}
for all $1 \leq i \leq 2N+1$, $1\leq j, s \leq K$, and $1 \leq q \leq m$. When reduced to 1D system, the generalized eigenvalue problem for $(\mathsf{A}, \mathsf{B})$ becomes
\begin{equation}\label{ref:geneigen}
\mathsf{A}\,\eta = \lambda \mathsf{B}\,\eta \,.
\end{equation}
To solve for the coefficient $a(x)$, we take \eqref{ODE-1D} as an example. Define $\gamma(x) = \eta^{\TT} \mathsf{B} \,a(x)$ and multiply \eqref{ODE-1D} by $\eta^T$ from the left. We then obtain the equation for $\gamma$ as
\begin{equation*} 
  \eta^{\TT}\mathsf{A}\,\partial_x a(x) = \eta^{\TT}\mathsf{B}\,a(x) \quad\Rightarrow\quad
  \lambda \partial_x\gamma(x) = 
  \gamma(x)  \,.
\end{equation*}
If $\lambda= 0$, then we immediately get the constraint 
\begin{equation}\label{constraint:lambda-0}
  \gamma(x) =
\eta^{\TT}\mathsf{B}\, a = 0 \,.
\end{equation}  
If $\lambda \neq 0$, then we
have
\begin{equation}
 \gamma(x) = e^{x / \lambda}\gamma(0) \,. \nn
\end{equation}
Depending on the signs of the eigenvalues, $\gamma$ either 
grows exponentially to infinity or decays exponentially to zero; as we look for bounded decaying solutions, this gives us 
constraints to $\gamma(0)$ for the growing modes: If $\lambda > 0$, then we have the constraints 
\begin{equation}\label{eqn:constraint_PosEigen}
  \gamma = \eta^{\TT} \mathsf{B}\, a  = 0.
\end{equation}
Note that we do not need constraints for modes with negative eigenvalues. The total number of constraints in the form of \eqref{eqn:constraint_PosEigen} is determined by the number of positive generalized eigenvalues. The following Proposition gives the signature of $(\mathsf{A}, \mathsf{B})$:

\begin{prop}\label{prop:eigenvalue}
Let $\mathfrak{A}, \mathfrak{B}$ be the 8-tensors defined in \eqref{def:D-B} with any arbitrary $u \in \R$ and $N, K \geq 1$. Then 

\Ni (a) there are $mNK^2$ positive generalized eigenvalues, $mNK^2$ negative eigenvalues, and $mK^2$ zero eigenvalue for the pair $(\mathfrak{A}, \mathfrak{B})$. 

\Ni (b) In the special case where $m = K = 1$ and $\mathsf{A}, \mathsf{B}$ be the matrices defined in \eqref{def:A-B} with any arbitrary $u \in \R$ and $N \geq 1$, there are $N$ positive generalized eigenvalues, $N$ negative eigenvalues, and one zero eigenvalue for the pair $(\mathsf{A}, \mathsf{B})$
\end{prop}
\begin{proof}
We first verify that in the 1D case, $(\mathsf{A}, \mathsf{B})$ has $N$ positive, $N$ negative, and one zero generalized eigenvalues. By the definition of $\mathsf{B}$ and the strict coercivity of $\CalL_d$, the matrix $\mathsf{B}$ is symmetric and strictly positive definite. Hence the numbers of positive, negative, and zero generalized eigenvalues are the same with the signature of the matrix $\mathsf{B}^{-1} \mathsf{A}$. Furthermore, by the Sylvestre's Law of Inertia, $\mathsf{B}^{-1} \mathsf{A}$ and $\mathsf{A}$ have the same signature. Hence, we only need to count the numbers of positive, negative, and zero eigenvalues of $\mathsf{A}$. Note that by the definition of the basis functions $\psi_k$ in \eqref{def:basis-function}, $\mathsf{A}$ is independent of $u$ since one can perform a change of variable $v+u \to v$ in each entry in $\mathsf{A}$. Thus we only need to study the matrix $\mathsf{A}_0$ with $u=0$. Change the order of the basis functions such that 
\begin{equation*} 
  (\tilde\psi_1, \tilde\psi_2, \cdots, \tilde\psi_{N+1}, \tilde\psi_{N+2}, \cdots \tilde\psi_{2N+1})
  = (\psi_1, \psi_3, \cdots, \psi_{2N+1}, \psi_2, \cdots, \psi_{2N})
  = P (\psi_1, \psi_2, \cdots, \psi_{2N+1}) \,,
\end{equation*}
where $P$ is the similarity matrix. Defined $\tilde{\mathsf{A}}_0 = P \mathsf{A}_0 P^{-1}$. Then $\tilde{\mathsf{A}}_0$ and $\mathsf{A}_0$ have the same signature. By the even/odd properties of $\tilde \psi_i$, the matrix $\tilde{\mathsf{A}}_0$
has the form
\begin{equation*} 
      \tilde{\mathsf{A}}_0
      = \begin{pmatrix}
             0 & A_1 \\[2pt]
             A_1^{\TT} & 0 
         \end{pmatrix} \,,
\end{equation*}
where $A_1 = \left(\int_\R v \psi_{2i} \psi_{2j+1}\right)_{N \times (N+1)}$. Suppose $\eta = (\eta_{1,1}, \cdots, \eta_{1, N}, \eta_{2,1}, \cdots, \eta_{2, N+1})^T= (\eta_1^{\TT}, \eta_2^{\TT})^{\TT}$ is an eigenvector of $\tilde{\mathsf{A}}_0$ with eigenvalue $\lambda$. Then one has
\begin{equation*}
      A_1 \eta_2 = \lambda \eta_1 \,, \qquad A_1^{\TT} \eta_2 = \lambda \eta_1 \,. 
\end{equation*}
It is clear that $(\eta_1, -\eta_2)$ is also an eigenvector of $\tilde{\mathsf{A}}_0$ and the associated eigenvalue is $-\lambda$. This shows the eigenvalues of $\tilde{\mathsf{A}}_0$ appear in pairs. Since $A_1$ has a full rank $N$, we have that $\rank \tilde{\mathsf{A}}_0 = 2N$. Therefore $\tilde{\mathsf{A}}_0$, thus $\mathsf{A}_0$  and $\mathsf{A}$,  has $N$ positive eigenvalues, $N$ negative eigenvalues, and one zero eigenvalue. 

Now we claim that each generalized eigenpair $(\lambda, v)$ of $\mathsf{A}$ gives rise to $mK^2$ eigenpairs of $\mathfrak{A}$. Indeed, let $\{w^{(l)}\}_{l=1}^K$ be a set of basis vectors of $\R^K$. Choose the 4-tensor $\eta^{(ln)}_i = v \otimes w^{(l)} \otimes w^{(n)} \otimes \Be_i$. Then
\begin{equation*}
     \mathfrak{A} \eta^{(ln)}_i 
     = (\mathsf{A} \otimes \ii \otimes \ii \otimes \ii) (v \otimes w^{(l)} \otimes w^{(n)} \otimes \Be_i) 
     = (\mathsf{A} v) \otimes w^{(l)} \otimes w^{(n)} \otimes \Be_i
     = \lambda \eta^{(ln)}_i
     \,,
\end{equation*}
for any $1 \leq l, n \leq K$. Thus each $(\lambda, v \otimes w^{l} \otimes w^{(n)} \otimes \Be_i)$ is an eigenpair of $\mathfrak{A}$.

Note that we can also view $\mathfrak{A}$ and $\mathfrak{B}$ as two matrices of size $(m(2N+1)K^2) \times (m(2N+1)K^2)$ by defining a bijection between the indices 
\begin{equation*}
    \Upsilon: \{(i,j,l,p) | \, i = 1, \cdots, 2N+1,\, j, l = 1, \cdots, K, \, p = 1, \cdots, m\}\to \{1, \cdots, m(2N+1)K^2\} \,.
\end{equation*}
Then $\mathfrak{B}$ is symmetric and positive definite and $\mathfrak{A}$ is symmetric. Therefore, by a similar argument as for $(\mathsf{A}, \mathsf{B})$ using Sylvestre's Law of Inertia, the number of positive, negative, and zero generalized eigenvalues agree with those of $\mathfrak{A}$. This shows there are $mNK^2$ positive, $mNK^2$ negative, and $mK^2$ zero generalized eigenvalues for $(\mathfrak{A}, \mathfrak{B})$. \end{proof}



By Proposition~\ref{prop:eigenvalue}, we outline the specific steps that we take in our numerical computation: in total we have $N+1$ equations for $a(0)$ given by the constraints~\eqref{constraint:lambda-0} and~\eqref{eqn:constraint_PosEigen}. Combining them with the $N$ equations given by the boundary conditions~\eqref{cond:boundary} for 
$a(0)$, we get $2N+1$ equations for $2N+1$ unknowns $\{a_k(0)\}$.  
The linear system~\eqref{ODE-1D} for $a$ is then uniquely solvable, which further uniquely determines the approximate solution $f_N(x,v)$ by~\eqref{def:f-N}.

\subsection{Numerical scheme}\label{sec:scheme}

Let us now summarize the numerical algorithm for the half space
equation.  For simplicity, we present the algorithm for the 1D case
and the extension to the higher dimensional cases is similar.

The whole procedure consists of two parts: Compute the damped
equation, as shown in Algorithm~\ref{alg:damped} and recover the
solution to the original equation, as presented in
Algorithm~\ref{alg:recover}. Computing the damped equation itself has
discretization set-up step and computation step.

\RestyleAlgo{boxruled}
\begin{algorithm}
\KwData{Boundary condition: $\phi(v)$ for $v>0$ and the discretization $N$.}
\KwResult{$f$ that solves~\eqref{variational}, the variational formulation of~\eqref{eq:half-space-damping}.}
\textbf{Step I} Set up discretization:
\begin{itemize}
\item[1.] Construct $2N+1$ basis functions.
\item[2.] Compute two matrices defined in~\eqref{def:A-B}.
\item[3.] Solve the generalized eigenvalue problem~\eqref{ref:geneigen}.
\item[4.] Store the $N+1$ eigenvectors associated with non-negative eigenvalues.
\end{itemize}
\textbf{Step II} Compute the damped equation, seek for $a(0)$.
\begin{itemize}
\item[1.] Find $2N+1$ equations satisfied by $a(0)$:
\begin{itemize}
\item Use~\eqref{eqn:constraint_PosEigen} to find $N+1$ equations that projects out positive eigenvectors provided in II.4.
\item Impose the boundary condition~\eqref{cond:boundary}, which provides $N$ equations.
\end{itemize}
\item[2.] Compute $a(0)$.
\end{itemize}

\textbf{Step III} Assemble $f$ using equation\eqref{def:f-N}.
\caption{Compute the damped equation~\eqref{eq:half-space-damping}}\label{alg:damped}
\end{algorithm}

The first substep in Step I requires constructing $2N+1$ basis functions. Since it depends on the collision operator, we leave the details to numerical example section where we show basis preparation for the linearized BGK and the transport equation. The forth step in Step I requires the number of non-negative eigenvalue being exactly $N+1$ and this is guaranteed by Proposition 4.6, which is also used in substep 1 in Step II. 

The main cost of the numerical scheme lies in solving the eigenvalue
problem \eqref{ref:geneigen}, which scales cubicly as $N$
increases. Note that this is a common step for different boundary
conditions for the damped equation, and thus only needs to be done
once. As we employ a spectral discretization, as shown further in the
numerical results, accurate results are obtained even with a small
number of basis functions $2N + 1$. Therefore, the computational cost
is quite low.

\RestyleAlgo{boxruled}
\begin{algorithm}
\KwData{Boundary condition: $\phi(v)$ for $v>0$ and the positive modes $X_{+,0}$.}
\KwResult{$f_\phi$ that solves~\eqref{eq:half-space-full}.}
\begin{itemize}
\item[1.] Use Algorithm~\ref{alg:damped} to
  compute~\eqref{eq:half-space-damping} using $\phi$ as the boundary
  condition. 

  Denote the solution by $f$.
\item[2.] Use Algorithm~\ref{alg:damped} to
  compute~\eqref{eq:half-space-damping} using $X_{+,0}$ as the
  boundary conditions. 

  Denote the solution by $g_{+,0}$.
\item[3.] Compute $C$ in~\eqref{def:C} and $U$ in~\eqref{def:U-pm-0}.
\item[4.] Invert $C$ for $\eta$ as shown in~\eqref{eq:linearsysC}.
\item[5.] $f_\phi$ given by~\eqref{soln:approx} and $f_\infty$ given
  by the equation below~\eqref{eq:half-space-full}.
\end{itemize}
\caption{Recover the solution to the original
  equation\eqref{eq:half-space-full}}\label{alg:recover}
\end{algorithm}

%
%
%
%
%
%
%
%
%
%
%

\section{Numerical Examples}

As explained in Section~\ref{sec:scheme}, the overall strategy to solve
the half-space equation consists of two steps: First, we solve for the numerical solution to the half-space damped equation \eqref{eq:half-space-damping-1}
using the Galerkin approximation; Second, we recover the undamped
solution by Proposition~\ref{prop:recovery}, which involves the
solutions of the damped equation with various boundary conditions in order to
obtain the matrix $C$ in the linear system \eqref{eq:linearsysC}.

Below we consider the linearized BGK equation and a linear transport
equation, both restricted to one dimension and single species (more
general cases are studied and presented in \cite{LiLuSun2015}).  As in
Proposition~\ref{prop:approximation}, for the Galerkin approximation,
we specify a set of even and odd functions to form the approximation
space $\Gamma_N$. The choice of these functions depends on the
particular equation under study. By Proposition~\ref{prop:1D-ODE}, the
solution of the approximate system 
\eqref{ODE-1D}--\eqref{cond:boundary} is reduced to solving the
generalized eigenvalue problem \eqref{ref:geneigen}, where we assemble
the matrices $\mathsf{A}$ and $\mathsf{B}$ using Gaussian
quadrature. This will be discussed in more details below.

Our algorithm is implemented in \textsf{MATLAB}. The Gaussian
quadrature abscissas and weights are obtained using symbolic
calculations in order to guarantee the precision.

\subsection{Linearized BGK equation} 
We first consider the case of one-dimension linearized BGK
equation. In this case, the basis functions is constructed using
the half-space Hermite polynomials.  Those are orthogonal
polynomials defined on the positive half $v$-axis with the weight
functions $\exp(-v^2)$: $\{B_n(v), v > 0\}$ such that
each $B_n(v)$ is a polynomial of order $n$ and 
\begin{equation}
  \int_0^\infty B_m(v) B_n(v) e^{-v^2} \ud v = \delta_{nm} \,.
\end{equation}
The orthogonal polynomials can be constructed using three term
recursion formula (see for example \cite{Shizgal:81}).  For completeness we recall some details in Appendix~\ref{sec:ortho}.

The basis functions $\psi_k$'s we need are either odd or even with respect
to $v = -u$. Hence we shift the functions $B_n$'s by
$-u$ and make even and odd extensions:
\begin{align}
  B^E_n(v) = 
  \begin{cases}
    B_n(v + u) / \sqrt{2}, & v > -u \,, \\
    B_n(-v - u) / \sqrt{2}, & v < -u \,.  
  \end{cases} \\
  B^O_n(v) = 
  \begin{cases}
    B_n(v + u) / \sqrt{2}, & v > -u \,, \\
    - B_n(-v - u) / \sqrt{2}, & v < -u \,. 
  \end{cases}
\end{align}
Finally, $\psi_k$'s are obtained by multiplying these
functions by the square root of the Maxwellian: for $n \geq 1$
\begin{equation}
   \label{def:basis-function}
\begin{aligned} 
  & \psi_{2n-1} = B_{n-1}^O e^{-(v+u)^2/2}, \\
  & \psi_{2n} = B_{n-1}^E e^{-(v+u)^2/2}.
\end{aligned}
\end{equation}
By definition, $\psi_{2n-1}$ is odd, $\psi_{2n}$ is even, and they form a 
orthonormal basis of $L^2(\rd v)$. For a fixed $n$, $(v+u) \psi_{2n}(v)$ is a odd 
function with respect to $v = -u$. For $v > -u$,
\begin{equation*}
  (v+u) \psi_{2n}(v) = (v + u) B_{n-1}(v+u) e^{-(v +u)^2/2} / \sqrt{2}.  
\end{equation*}
Since $(v+u) B_n(v+u)$ is a $n$-th order polynomial in $v+u$, there
exists an expansion 
\begin{equation}
  (v+u) B_{n-1}(v+u) = \sum_{i = 0}^n \alpha_i B_i(v+u). 
\end{equation}
This yields that 
\begin{equation}
  (v+u) \psi_{2n}(v) = \sum_{i=0}^n \alpha_i \psi_{2i+1} \in \Span\{ \psi_1, \cdots, \psi_{2n+1} \}. 
\end{equation}
Therefore, $\Gamma_N = \Span\{\psi_1, \cdots, \psi_{2N+1}\}$ satisfies
the condition of Proposition~\ref{prop:approximation} and the
variational formulation
\eqref{eq:variational-finitedim}--\eqref{cond:boundary} is well-posed.
The $(2N+1)\times(2N+1)$ matrices $\mathsf{A}$ and $\mathsf{B}$ are
then given by
\begin{equation} \nn
  \mathsf{A}_{ij}  = \int_\R (v+u) \psi_i \psi_j \ud v
  \quad \text{and} \quad 
  \mathsf{B}_{ij} = -\int_\R \psi_i \mc{L}_d \psi_j \ud v. 
\end{equation}
Note that both matrices are symmetric. The matrix $\mathsf{A}$ can be
obtained by using the recurrence relation of the orthogonal
polynomials. For the matrix $\mathsf{B}$, recall that
\begin{align*}
  \mc{L} \psi_i &= \psi_i - m_i= \psi_i - \chi_0 \int_\R
  \psi_i\chi_0\ud{v}- \chi_+ \int_\R
  \psi_i \chi_+\ud{v} - \chi_- \int_\R \psi_i\chi_-\ud{v}.\\
  \mc{L}_d \psi_i & = \mc{L} \psi_i + \alpha \sum_{k=1}^{\nu_+} (v +
  u) X_{+,k}
  \int_\R (v + u)X_{+,k} \psi_i \ud{v} \\
  & + \alpha \sum_{k=1}^{\nu_-} (v + u) X_{-,k} \int_\R (v +
  u)X_{-,k} \psi_i \ud{v}
  + \alpha \sum_{k=1}^{\nu_0} (v  + u) X_{0,k} \int_\R (v + u) X_0 \psi_i \ud{v}\nonumber\\
  & + \alpha \sum_{k=1}^{\nu_0} (v + u) \CalL^{-1}((v + u)X_{0,k})\int_\R
  (v + u) \CalL^{-1}((v + u)X_{0,k}) \psi_i \ud{v}.
\end{align*}
All the integrals involved in calculating $\mathsf{B}$ can be easily
made exact up to machine precision by using Gaussian quadrature. For simplicity, let us just focus on 
\begin{equation*}
  \int_\R \psi_{2j} \chi_{0} \ud{v}
\end{equation*}
and note that the other integrals share the same structure: the
integrand is a product of two polynomials and two Gaussians
$e^{-v^2/2}$ and $e^{-(v+u)^2/2}$. To evaluate this type of integral using Gaussian quadrature, we first
split the integral into two parts:
\begin{equation*} 
  \int_\R \psi_{2j}\chi_0 \ud{v} = \int_{-u}^\infty \psi_{2j}\chi_0 \ud{v} + 
\int_{-\infty}^{-u} \psi_{2j}\chi_0 \ud{v}.
\end{equation*}
Note that $\psi_{2j}$, on either side of $-u$, is a $(j-1)$-th order
polynomial multiplied by $\exp(-(v+u)^2/2)$, while $\chi_0$ is a
quadratic function multiplied with a different weight function
$\exp(-v^2/2)$. The product of two Gaussians centered at different
locations could be combined into a single Gaussian:
\begin{align}\label{eqn:int_pos}
\int_{-u}^\infty \psi_{2j}\chi_0\ud{v} &= \frac{\sqrt{2}}{2}\int_{-u}^\infty 
B_{j-1}(v+u)\frac{\chi_0(v)}{e^{-v^2/2}}e^{-\frac{(v+u)^2+v^2}{2}}\ud{v}
\\
&=\frac{\sqrt{2}}{2} e^{-u^2/4}\int_{0}^\infty B_{j-1}(v)\frac{\chi_0(v-u)}{e^{-(v-u)^2/2}}e^{-(v-u/2)^2}\ud{v} \,. \nn 
\end{align}
Similarly, for $v < -u$ we have
\begin{align}\label{eqn:int_neg}
\int_{-\infty}^{-u} \psi_{2j}\chi_0\ud{v} 
&= \frac{\sqrt{2}}{2}\int_{-\infty}^{-u} 
B_{j-1}(-v-u)\frac{\chi_0(v)}{e^{-v^2/2}}e^{-\frac{(v+u)^2+v^2}{2}}\ud{v}
\\
&=\frac{\sqrt{2}}{2} e^{-u^2/4}\int_{0}^{\infty} B_{j-1}(v)\frac{\chi_0(-v-u)}{e^{-(v+u)^2/2}}e^{-(v+u/2)^2}\ud{v} \,. \nn
\end{align}
The integrals \eqref{eqn:int_pos} and~\eqref{eqn:int_neg} can be
evaluated up to machine precision by Gaussian quadrature based on
weight $e^{-(v-u/2)^2}$ and $e^{-(v+u/2)^2}$ respectively, as
$B_{j-1}\chi_0 e^{v^2/2}$ is a polynomial with its degree up to $N+3$.
The boundary condition~\eqref{cond:boundary} requires the numerical
evaluation of the integral
\begin{equation*}
  \int_{v+u>0} (v+u)\phi \psi_{2j} \ud{v}.
\end{equation*}
We calculate this using Gaussian quadrature with the weight $e^{-(v+u)^2}$.
The error of the quadrature depends on the number of quadrature points
and the regularity of the incoming data $\phi$.

\medskip

We now present some numerical results for the linearized BGK equation.  In
the first set of examples, we compare our numerical results with
analytical solutions, when the specified boundary data $\phi$ is given
by the restriction of some $f \in H^0\oplus H^+$ on $v>-u$. In this
case, the solution to the undamped equation~\eqref{eq:half-space-1} is
simply $f$ on the whole velocity space. As discussed in
\eqref{eqn:SevenCases}, the dimension of the space $H^0\oplus H^+$
depends on the bulk velocity $u$ and the sound speed, which is $c=\sqrt{3/2}$ in our case as $T
= 1/2$. We will choose $\chi_{+/-/0}$
defined in~\eqref{eqn:def_chi} as the incoming data. 
By the uniqueness of the half-space equation, the solution will simply be $\chi_{+/0}$ when the incoming data is chosen as $\chi_{+/0}$. 
We take six choices of $u$ corresponding to the six cases listed in
\eqref{eqn:SevenCases} (the case $u < -c$ gives an empty $H^0\oplus
H^+$ hence not included). The results are shown in
Figures~\ref{fig:u-2}--\ref{fig:u3} below. In all these figures, the
blue squared line is the incoming data, given by $\chi_-$,
$\chi_0$ and $\chi_+$ respectively. The green triangle line is the
solution at $x=\infty$, and the red dotted line is the solution at $x
= 0$.

Several remarks are in order: First, when the $\chi$ modes lie in
$H^0\oplus H^+$ for the given bulk background velocity $u$, we observe
in Figure 1-6 that the solution at $x=0$ gives a perfect match. We thus recover the exact solution from the numerical scheme. Second, we note that in general, the solution exhibits a jump at $v =
-u$, as clearly seen for example in Figure~\ref{fig:u-2}(left). This
justifies our choice of the even-odd formulation and basis functions
from the half-space Hermite polynomials. Finally, we remark that we have used a filtering (with $2$nd order
cosine filter) to reduce the Gibbs oscillations caused by the large
derivatives in some cases (for instance Figure~\ref{fig:u-1/2}(left)).

\begin{figure}[htp]
\begin{center}
   {\includegraphics[width=2.0in]{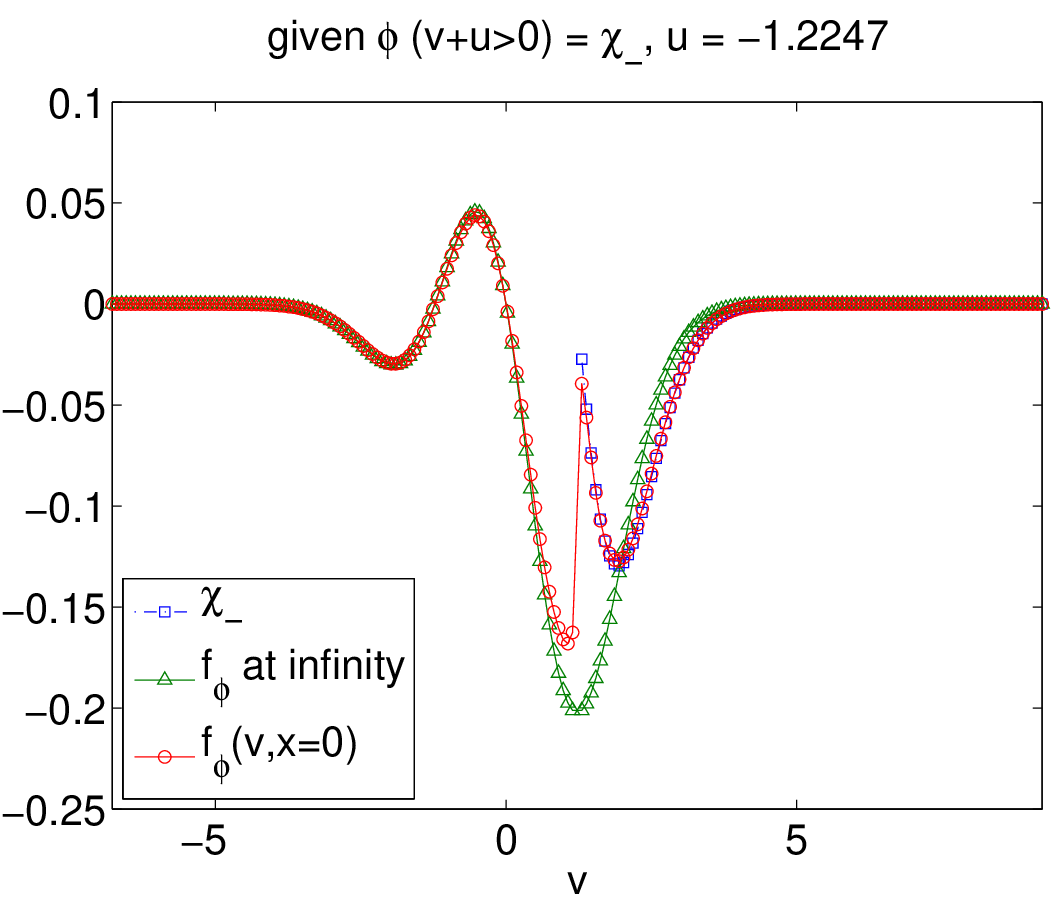}
   \includegraphics[width=2.0in]{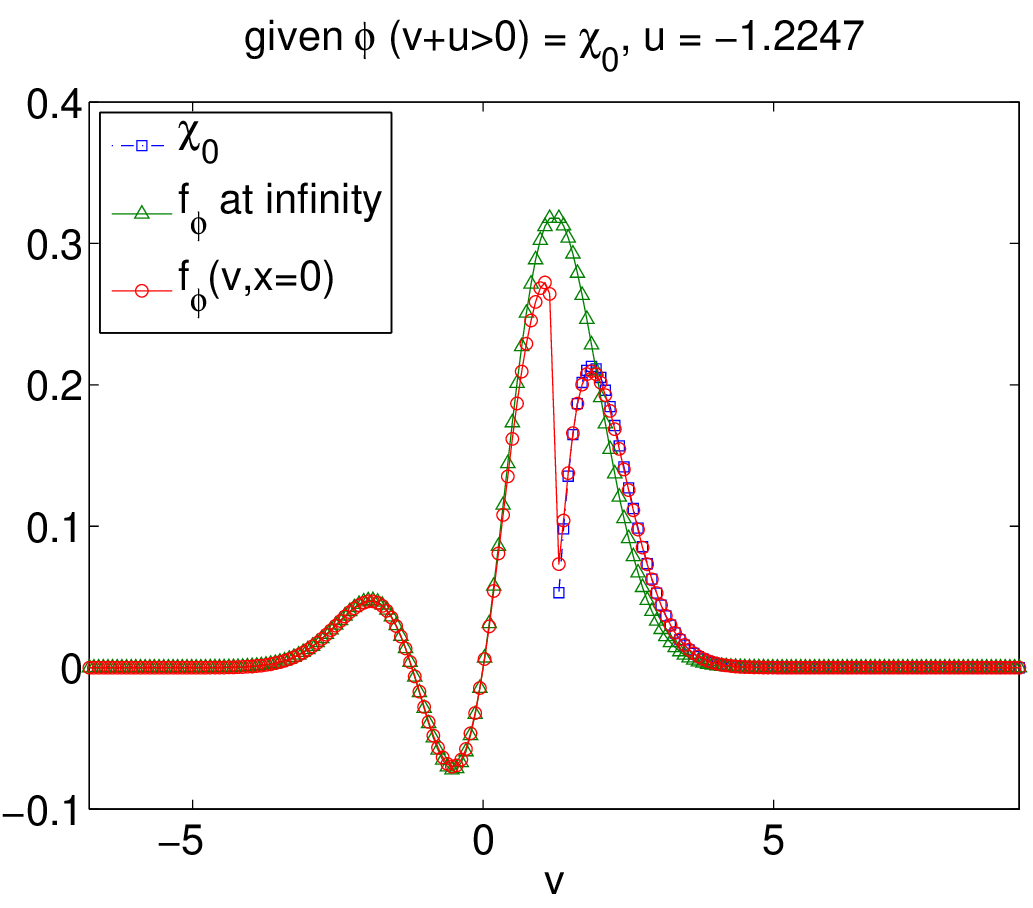}
   \includegraphics[width=2.0in]{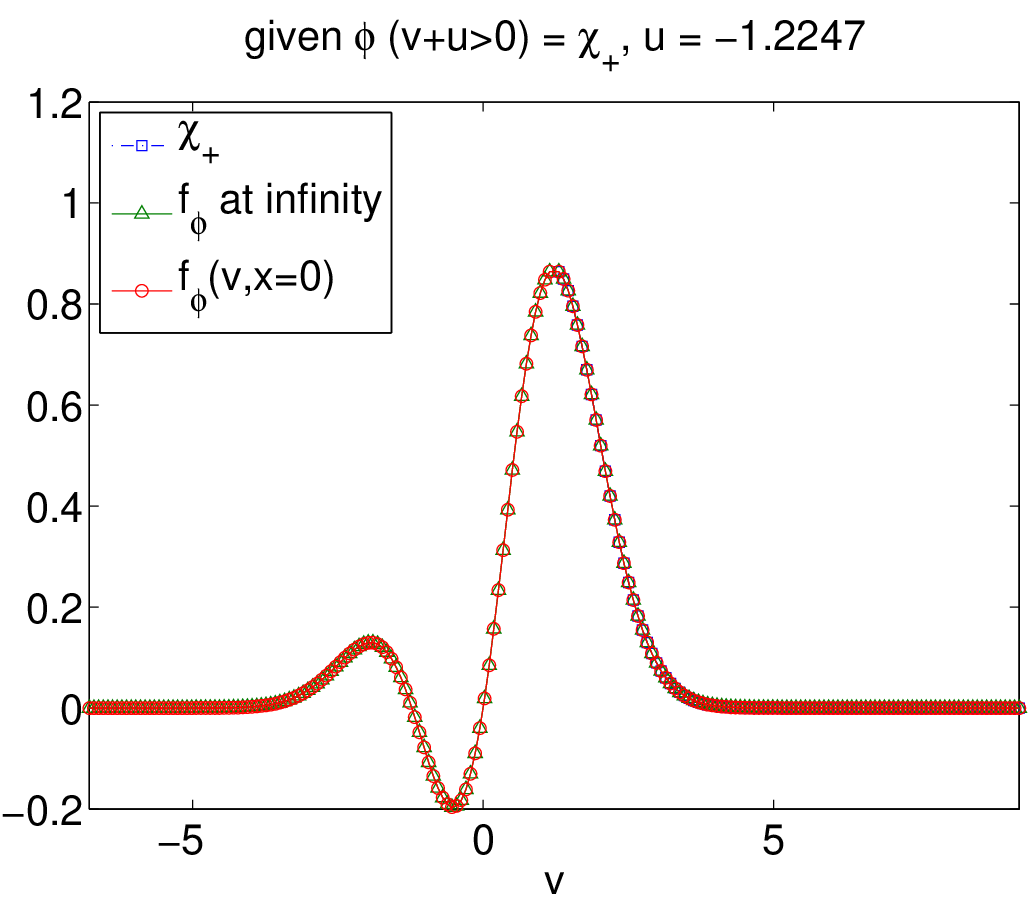}}
\end{center}\caption{$u = -\sqrt{1.5} = -c$. In this case $\chi_+\in
  H^0$, and $\chi_-$ and $\chi_0$ are in $H^-$. }\label{fig:u-2}
\end{figure}
\begin{figure}[htp]
\begin{center}
   {\includegraphics[width=2.0in]{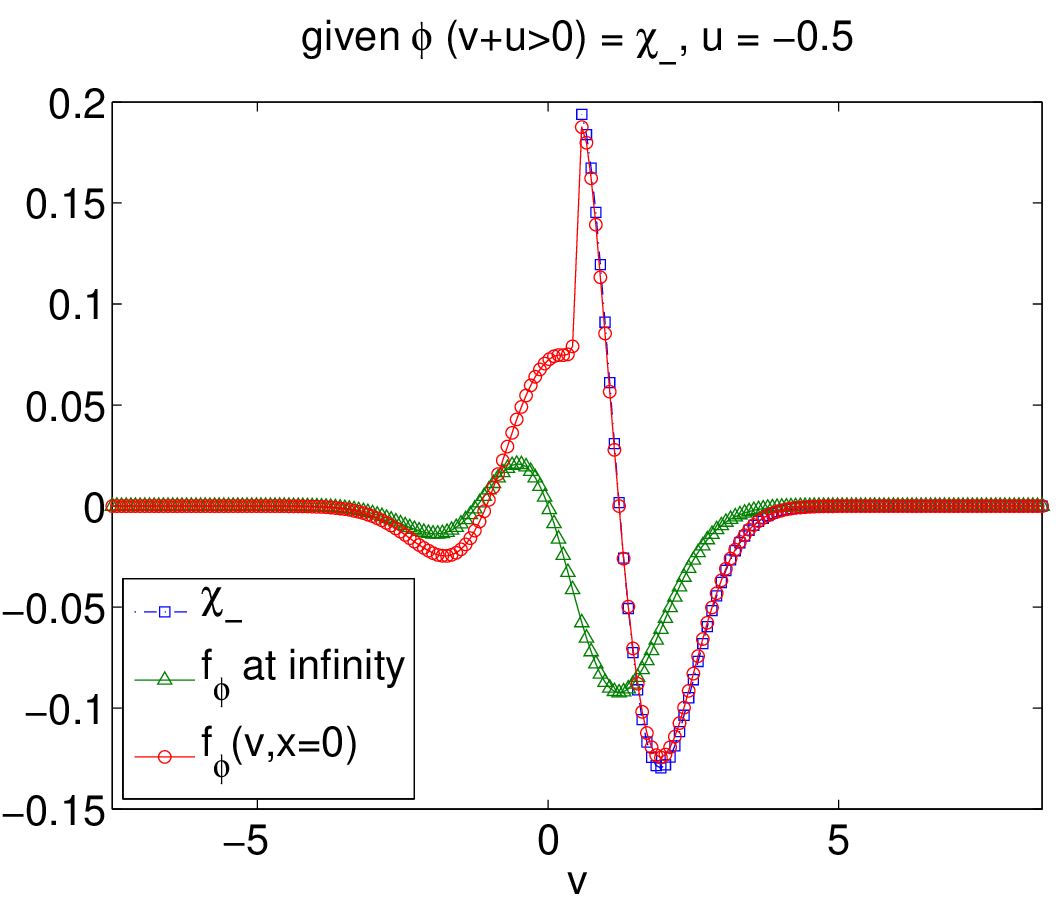}
   \includegraphics[width=2.0in]{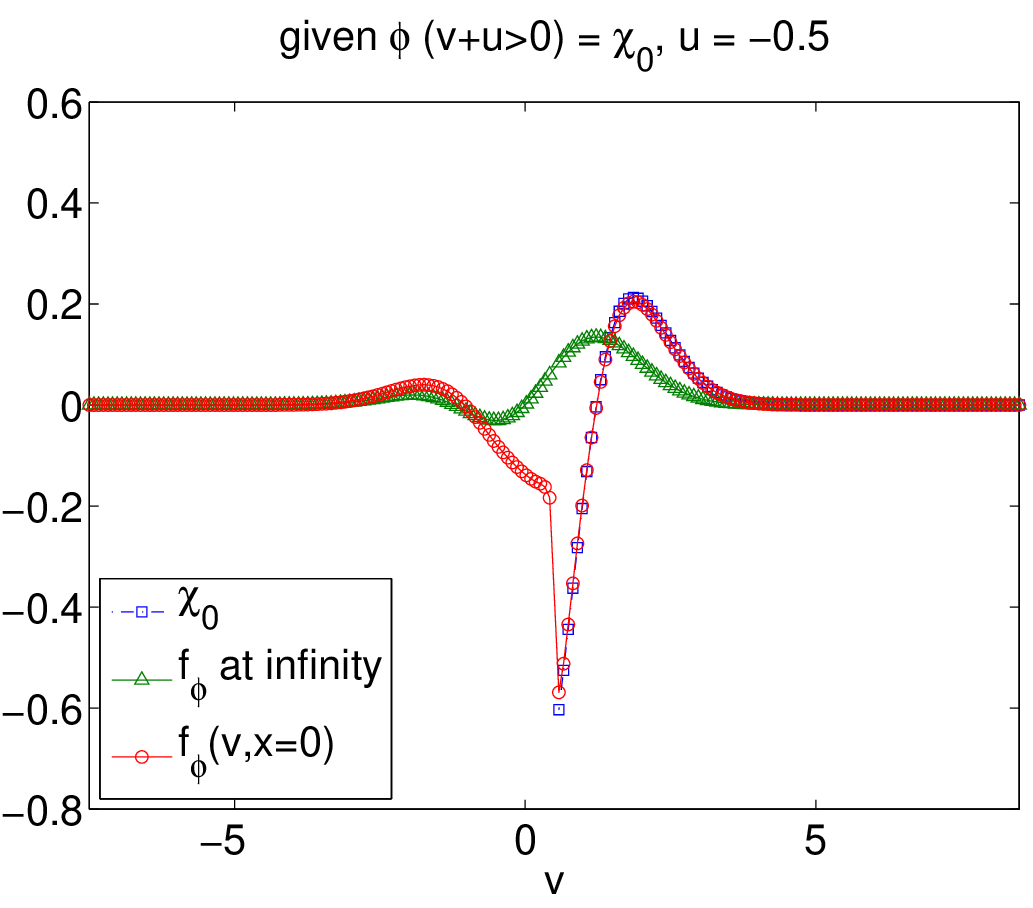}
   \includegraphics[width=2.0in]{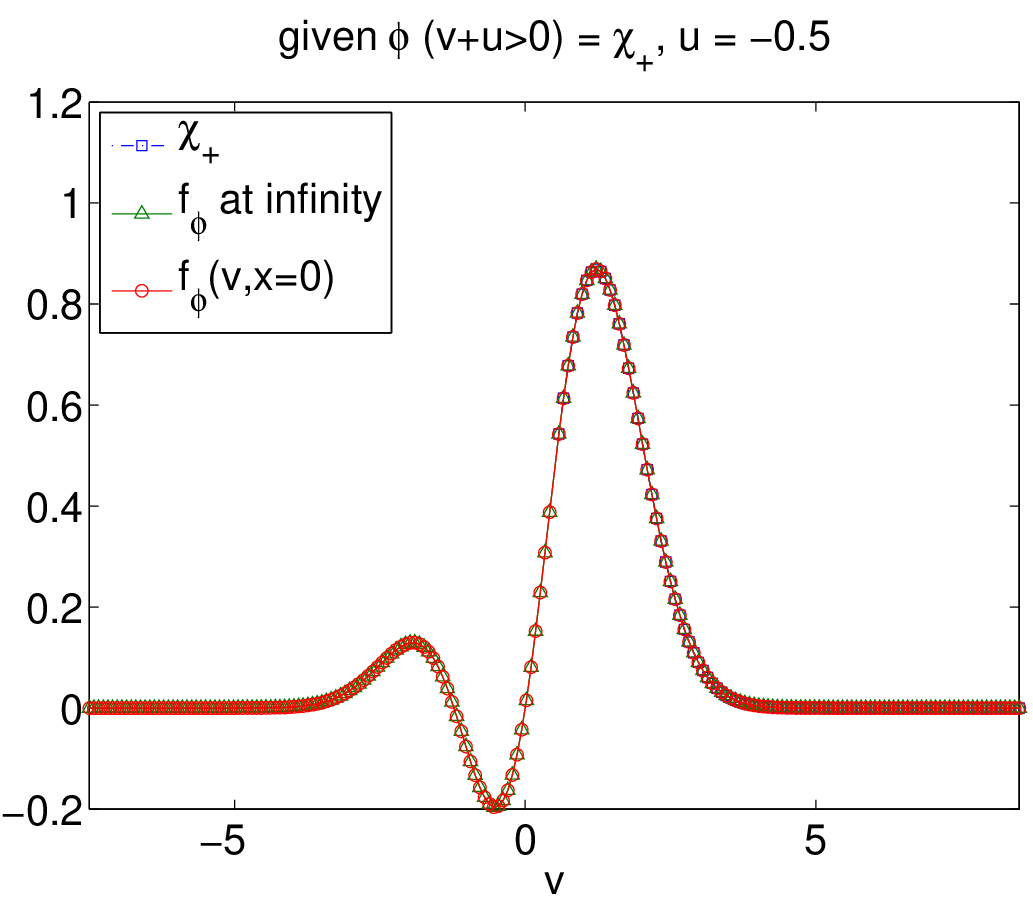}}
\end{center}\caption{$-c < u = -0.5 < 0$. In this case $\chi_+\in
  H^+$, and $\chi_-$ and $\chi_0$ are in $H^-$. }\label{fig:u-1/2}
\end{figure}
\begin{figure}[htp]
\begin{center}
   {\includegraphics[width=2.0in]{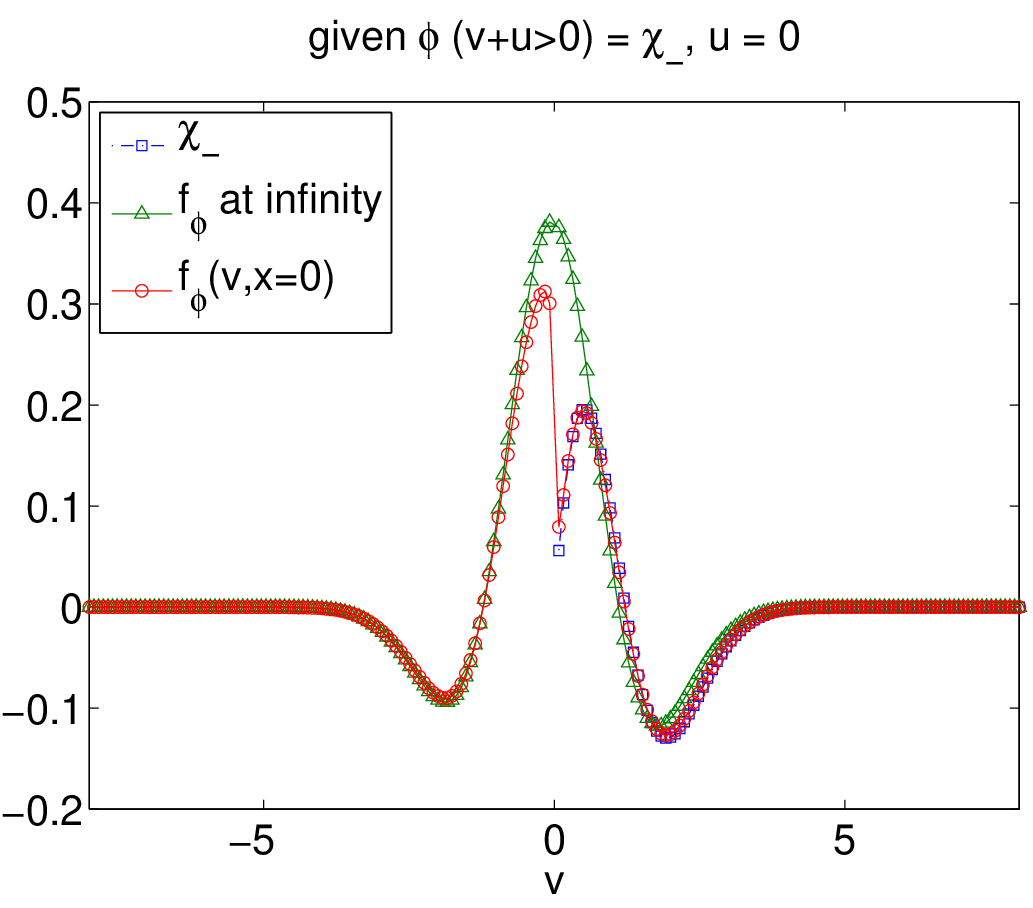}
   \includegraphics[width=2.0in]{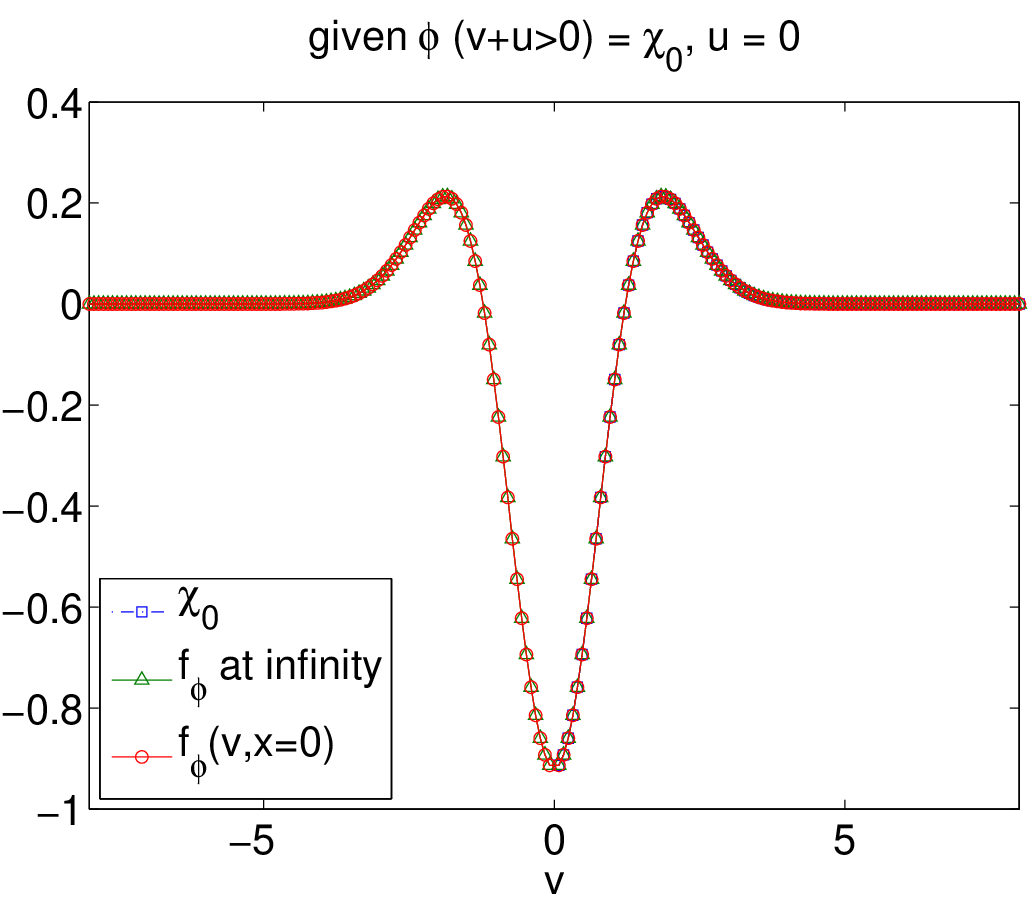}
   \includegraphics[width=2.0in]{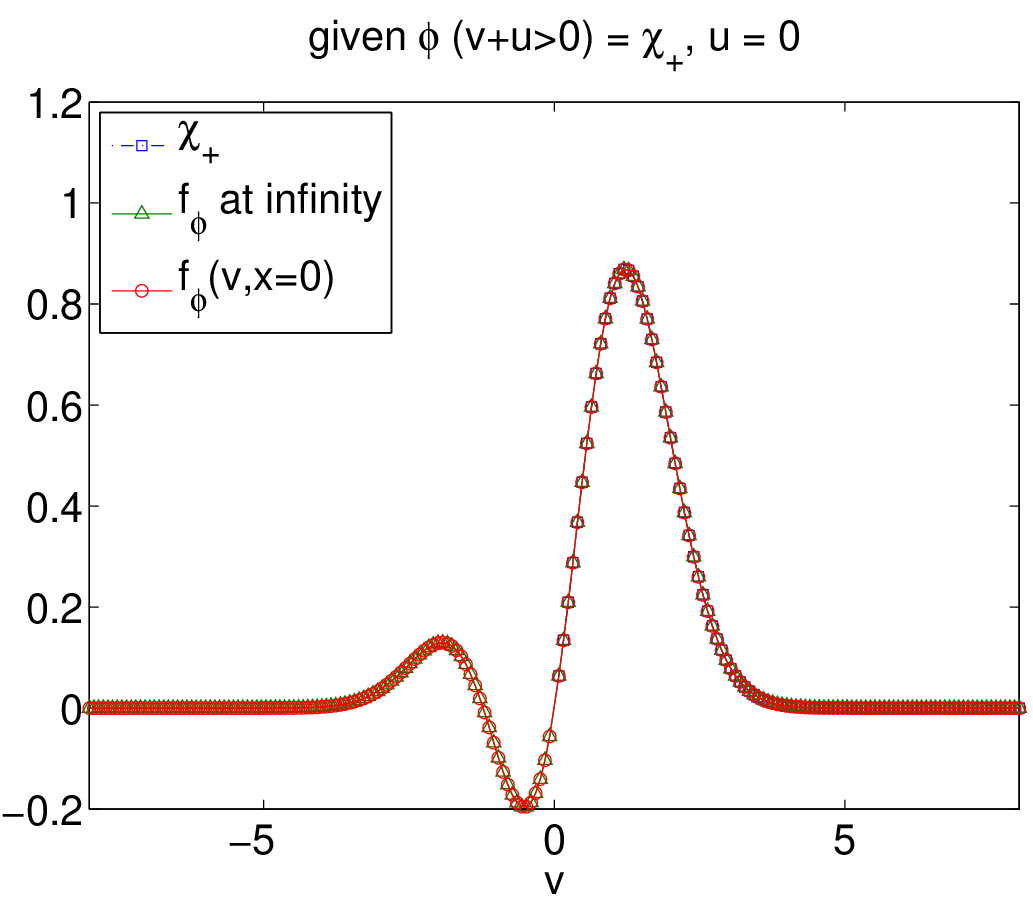}}
\end{center}\caption{$u = 0$. In this case $\chi_+\in H^+$, $\chi_0\in
  H^0$ and $\chi_-\in H^-$.}
\end{figure}
\begin{figure}[htp]
\begin{center}
   {\includegraphics[width=2.0in]{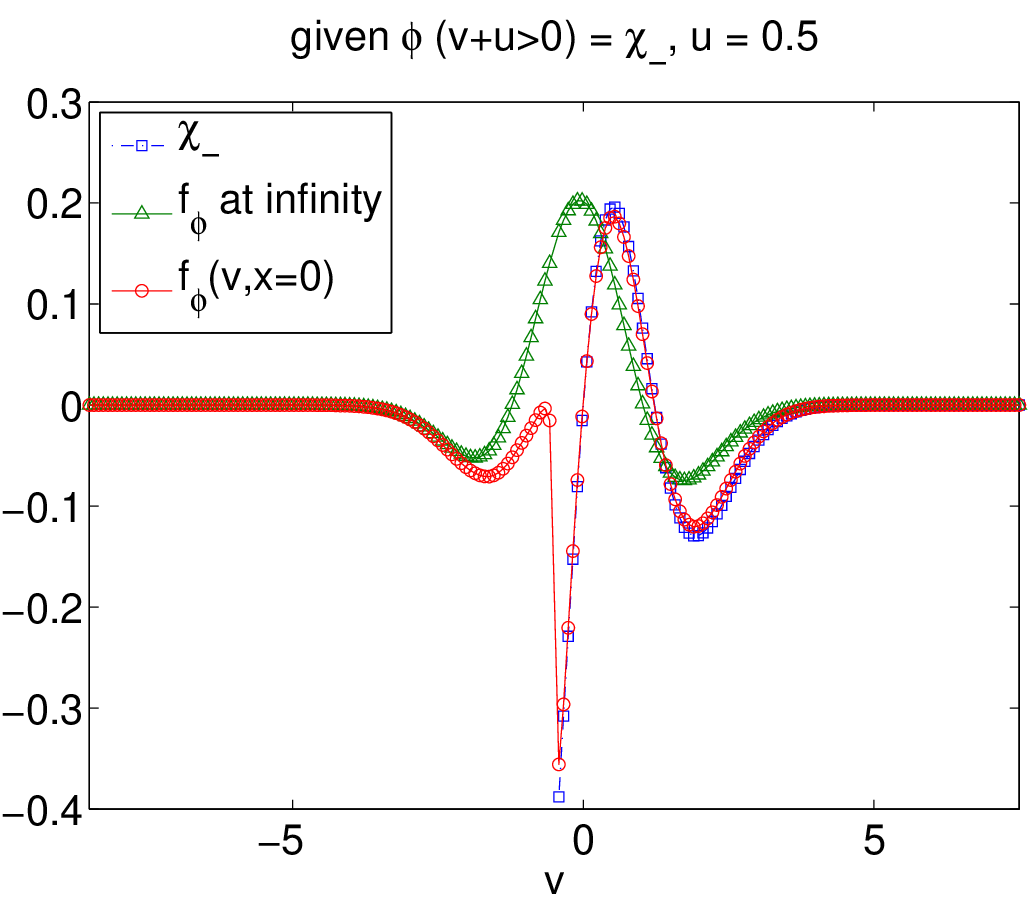}
   \includegraphics[width=2.0in]{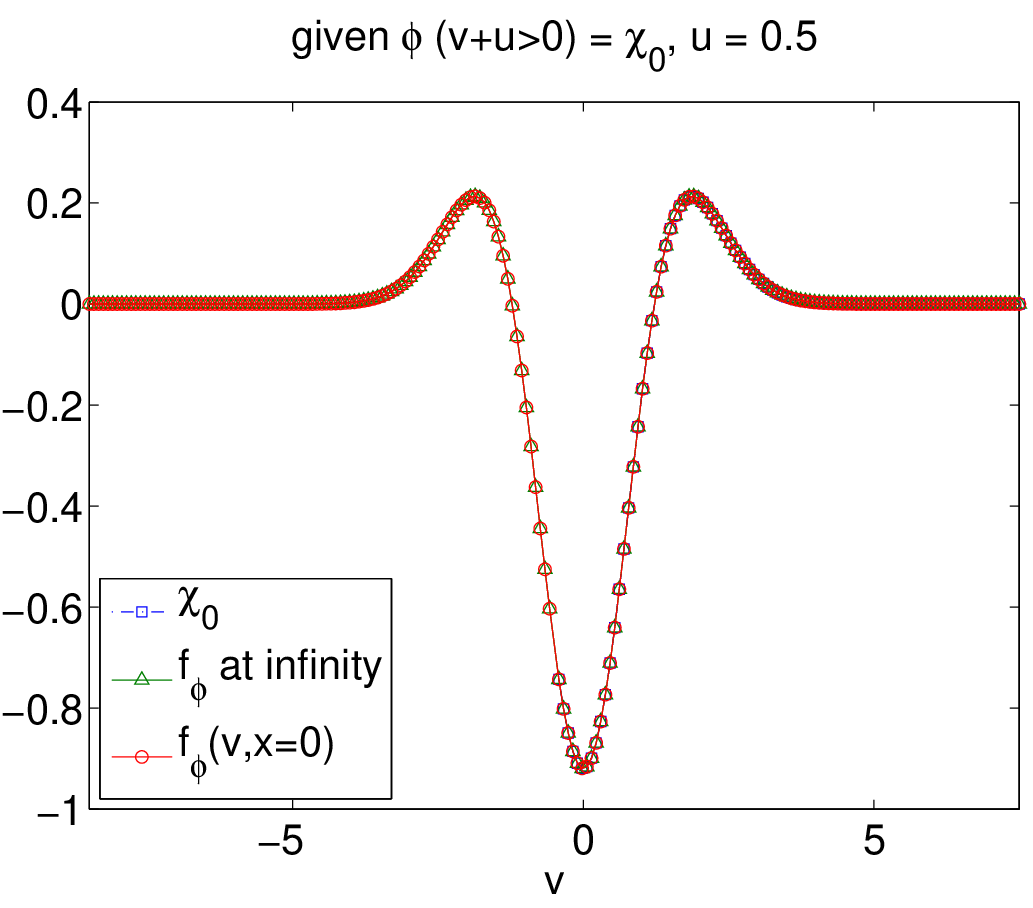}
   \includegraphics[width=2.0in]{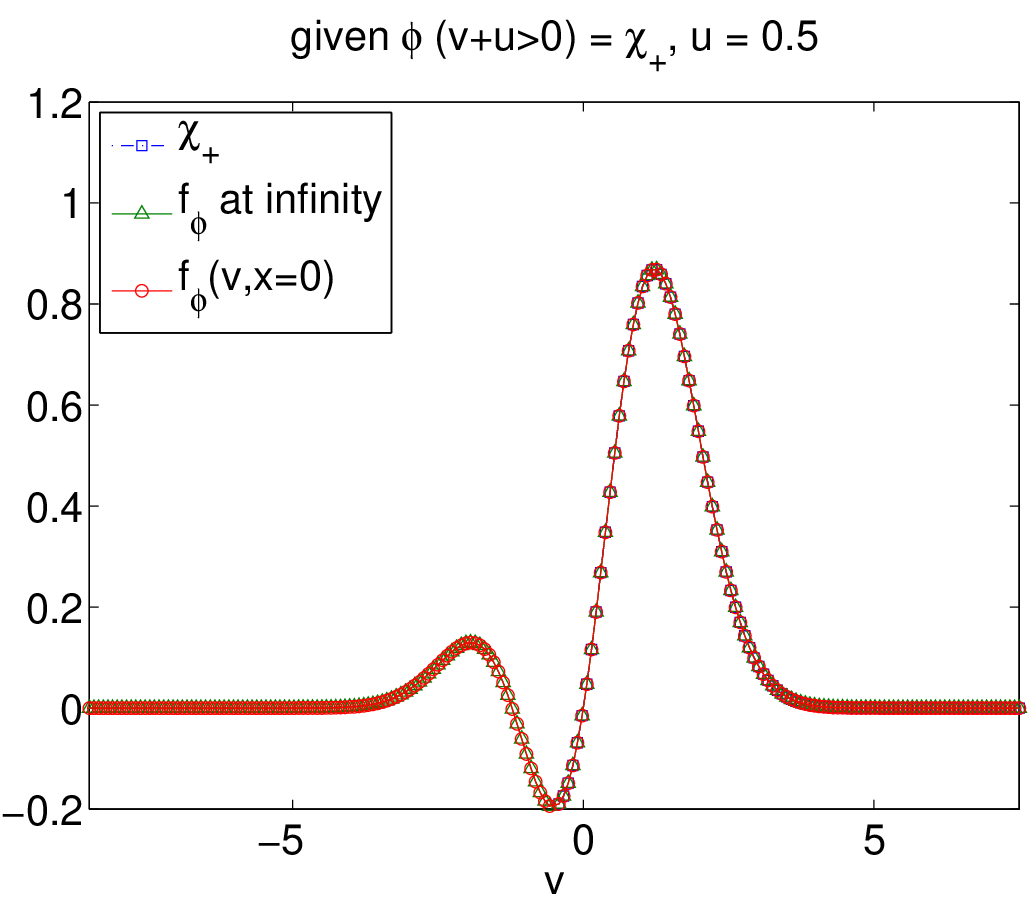}}
\end{center}\caption{$0 < u = 0.5 < c$. In this case $\chi_+$ and
  $\chi_0$ are in $H^+$, and $\chi_-\in H^-$. }
\end{figure}
\begin{figure}[htp]
\begin{center}
   {\includegraphics[width=2.0in]{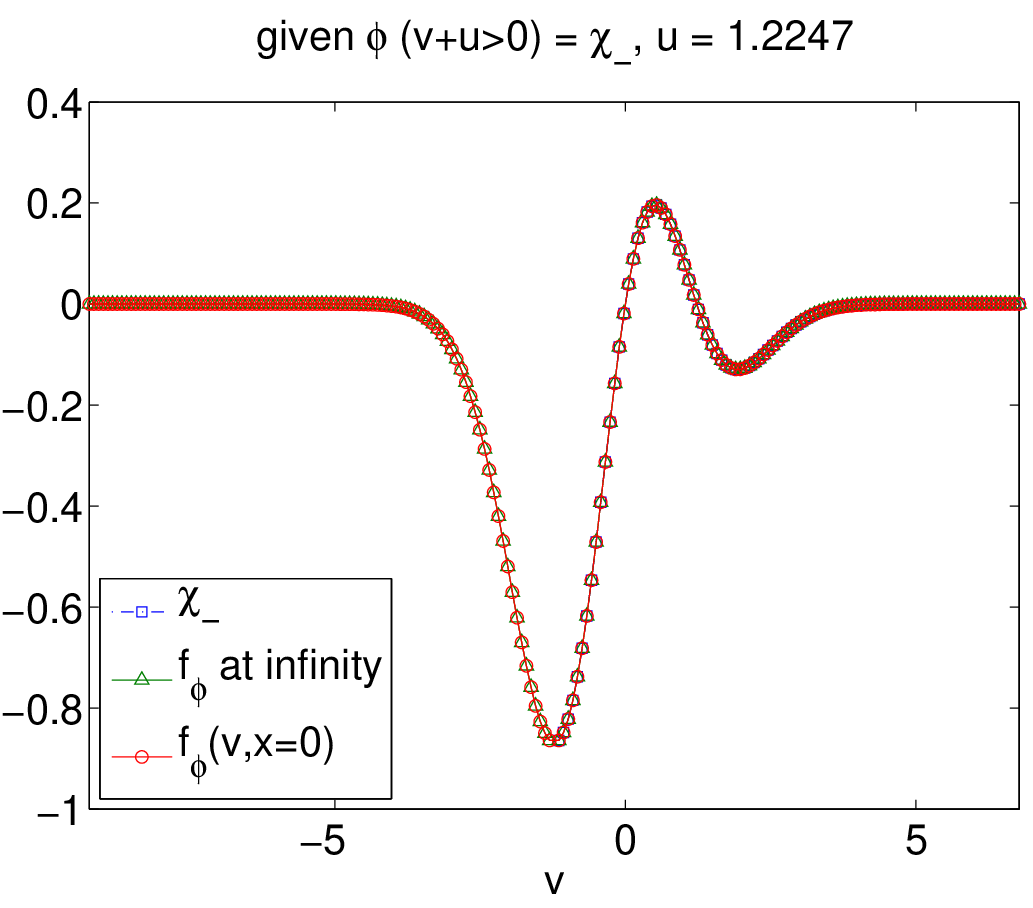}
   \includegraphics[width=2.0in]{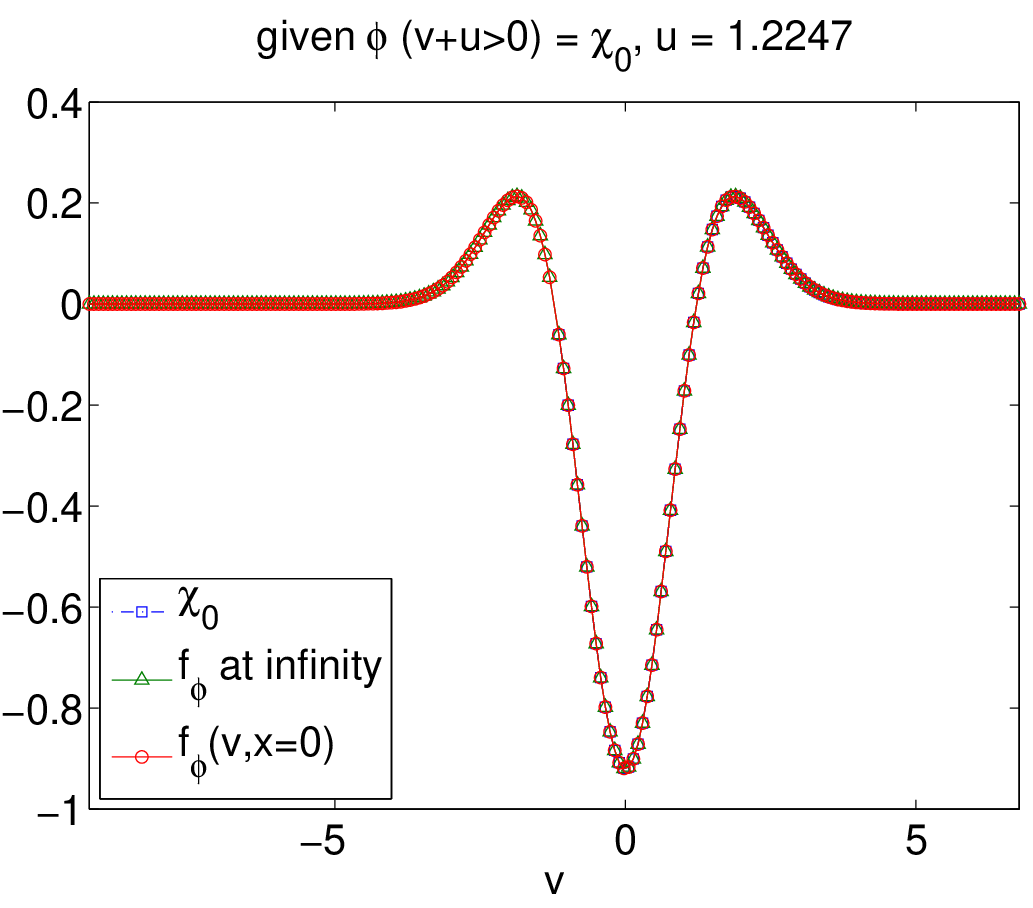}
   \includegraphics[width=2.0in]{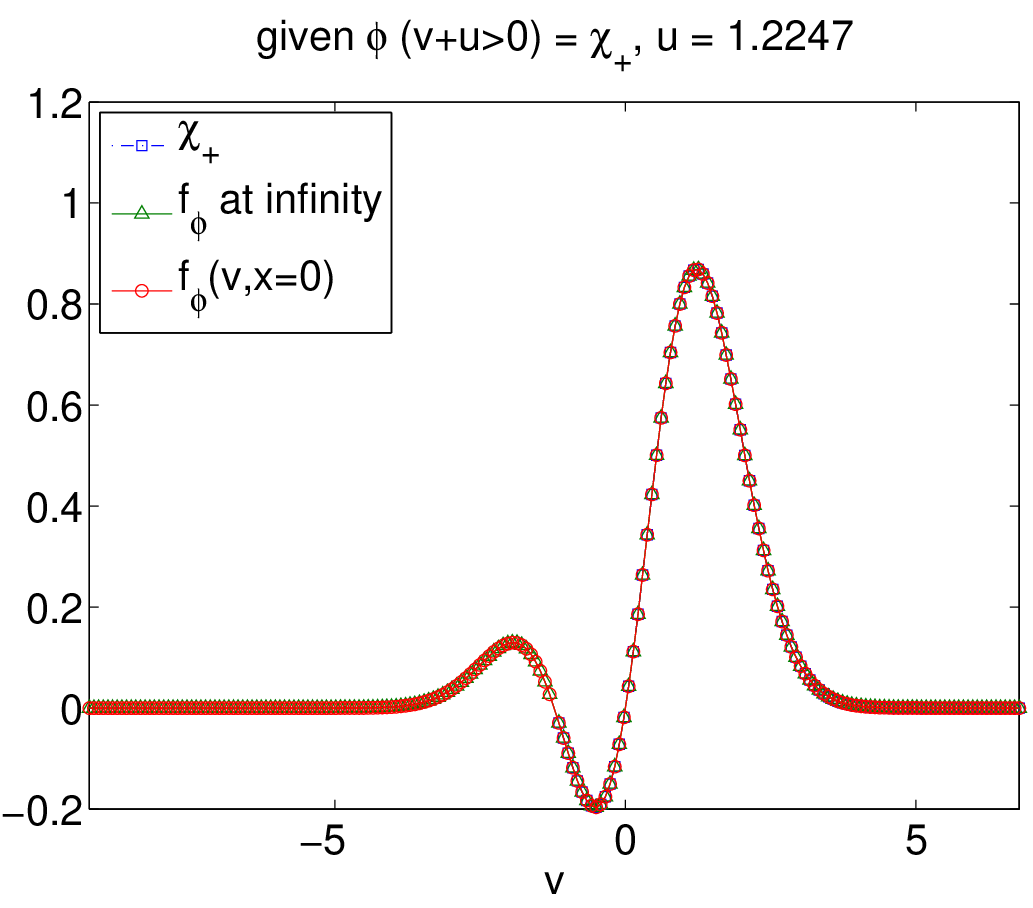}}
\end{center}\caption{$u = \sqrt{1.5} = c$. In this case $\chi_+$ and
  $\chi_0$ are in $H^+$, and $\chi_-\in H^0$. }
\end{figure}
\begin{figure}[htp]
\begin{center}
   {\includegraphics[width=2.0in]{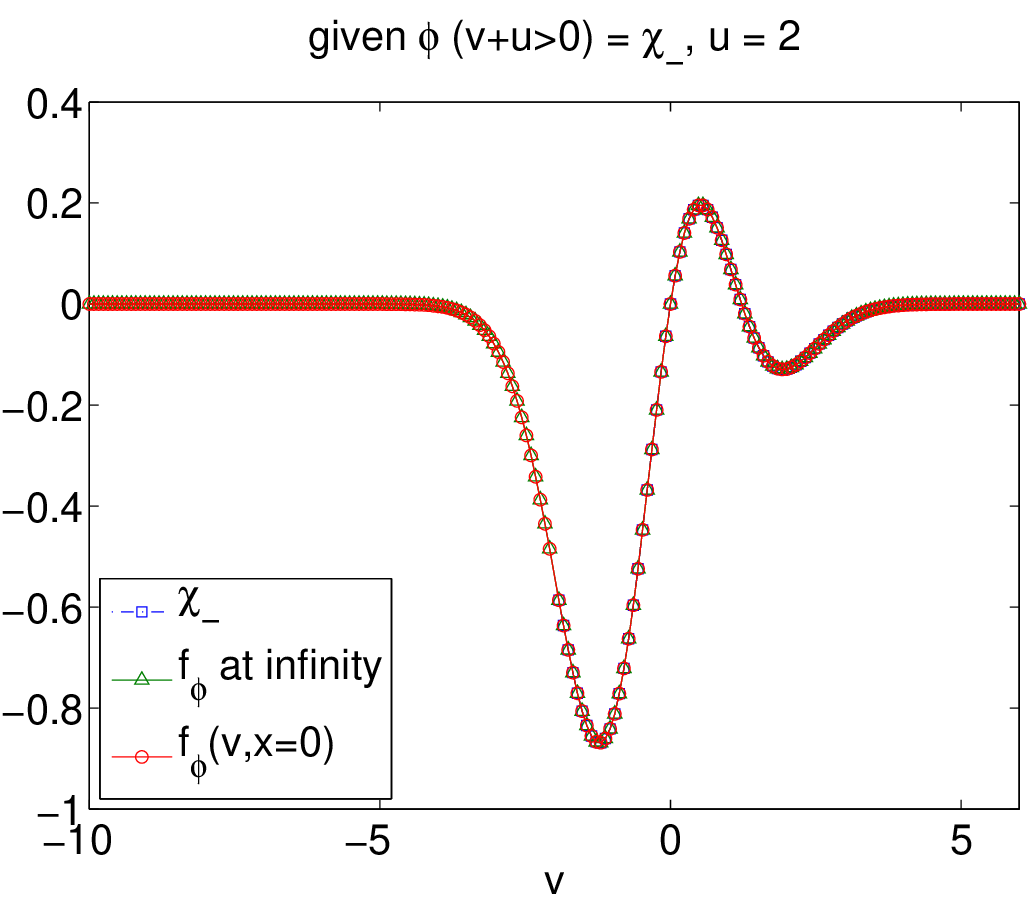}
   \includegraphics[width=2.0in]{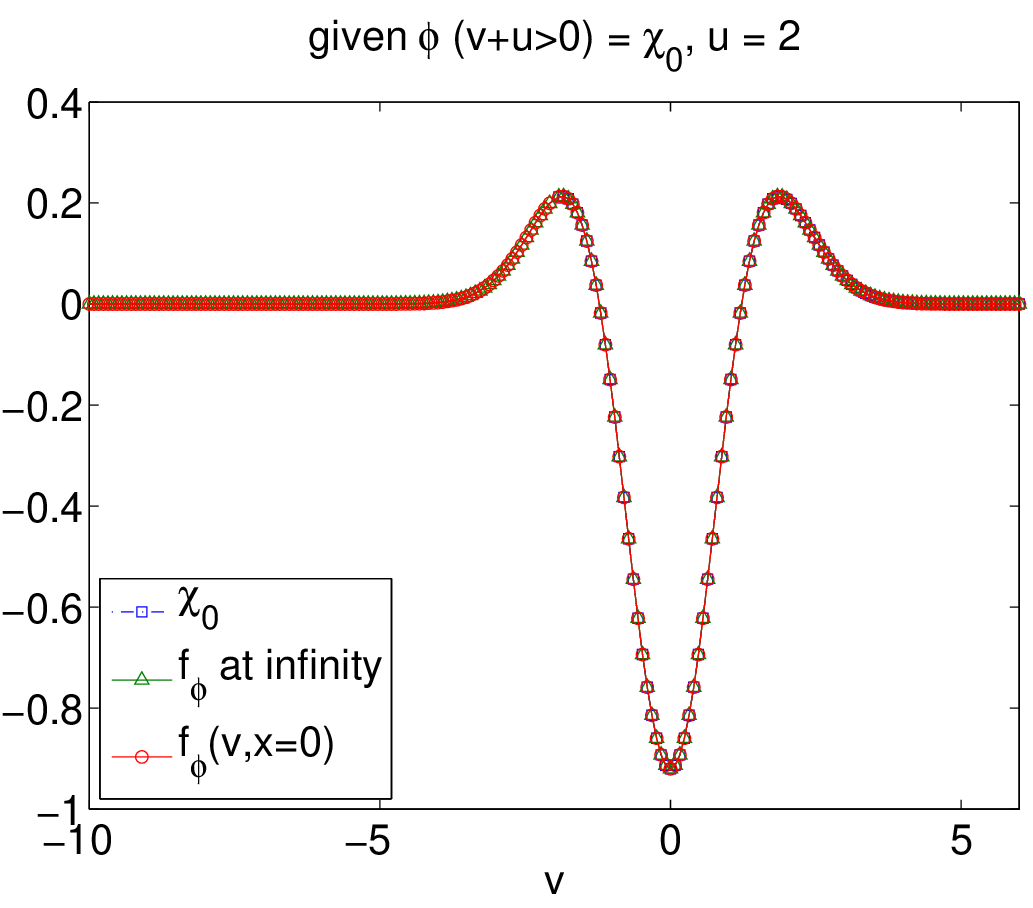}
   \includegraphics[width=2.0in]{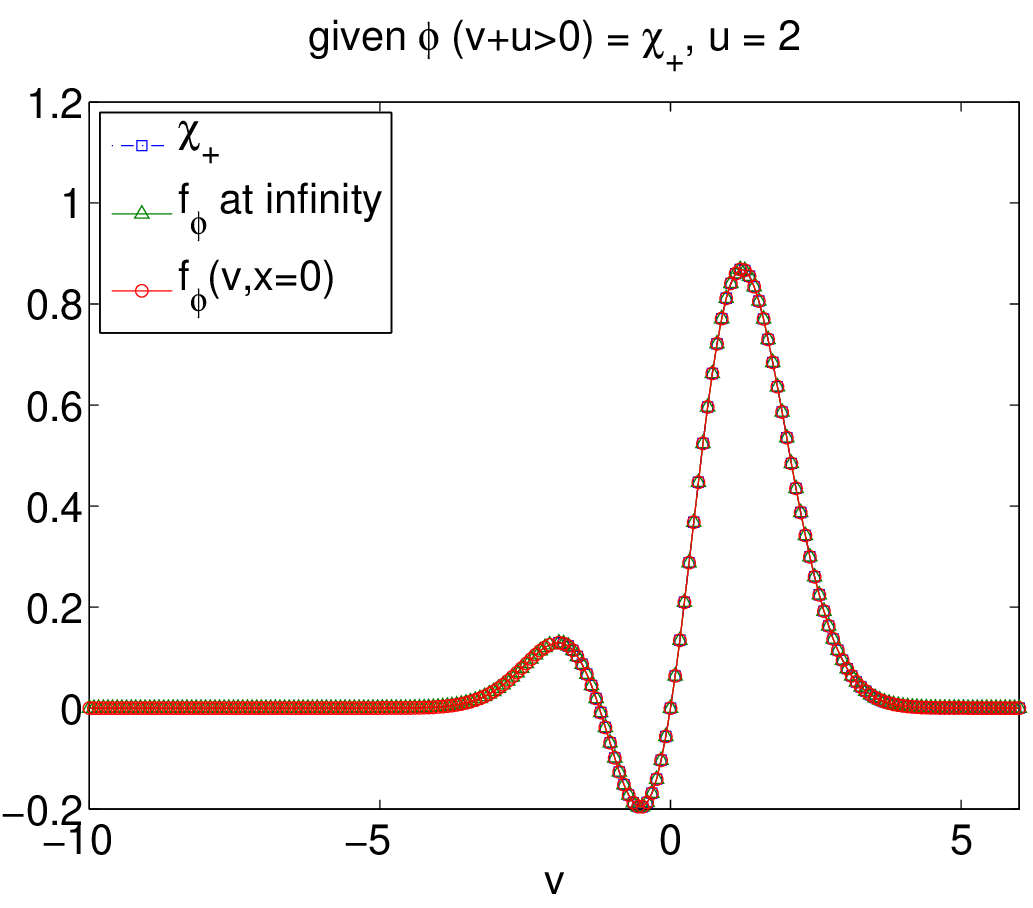}}
\end{center}\caption{$u = 2 > c$. In this case all $\chi$ are in
  $H^+$. }\label{fig:u3}
\end{figure}

Next, we consider an example where the exact solution is not known.
We solve the equation~\eqref{eq:half-space-1} for $u = 0$ with
boundary data $\phi = v^3, v>0$.  The numerical solution is
shown in Figure~\ref{fig:v_cube}.
\begin{figure}[htp]
\begin{center}
\includegraphics[width = 3.5in]{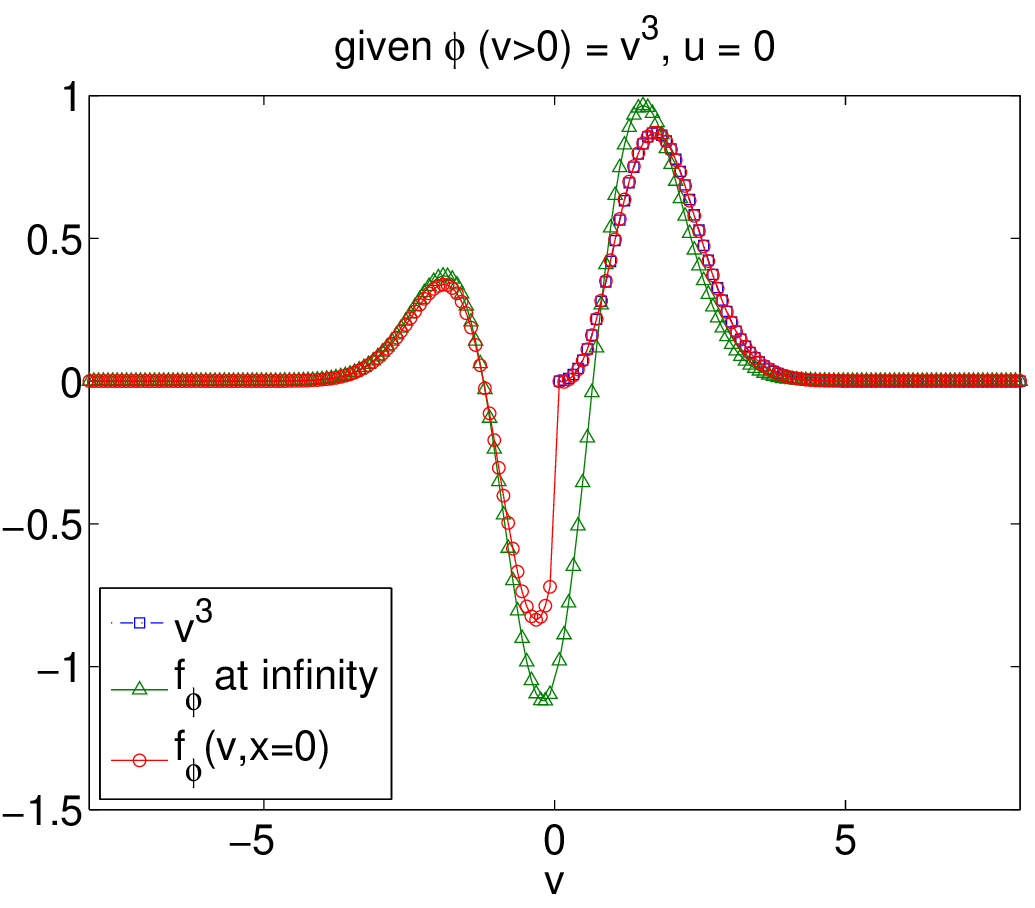}
\end{center}
\caption{Blue boxed line is the input data $\phi = v^3 (v>0)$. Green triangle line is the solution at infinity and the red circled line is the solution at the boundary. N = 36 here.}\label{fig:v_cube}
\end{figure}

\subsection{Isotropic neutron transport equation}

We further consider the isotropic neutron transport equation. The
construction of the basis functions is similar to the linearized BGK
case. However, instead of using half-space Hermite polynomials, we
start with Legendre polynomials on the interval $[0, 1]$ and carry out
the even-odd extensions. The Legendre polynomials, which are
orthogonal polynomials for constant weight function, are used since
the equilibrium states for the neutron transport equation are simply
constants. We then apply Gauss-Legendre quadrature to assemble
$\mathsf{A}$ and $\mathsf{B}$ for the generalized eigenvalue
problem. The rest of the details are skipped here since the
construction is relatively straightforward compared with the
linearized BGK case.

To validate our methods in this case, we compare the numerical solution
with the analytical solution with boundary data given by $\phi =
v$ for $v \in [0, 1]$. The analytical solution is known as 
\begin{equation}
f_\phi(-v) = \frac{1}{\sqrt{3}}H(v) - v, \qquad v>0 \,,
\end{equation}
where $H$ is the Chandrasekhar H-function.  In
Figure~\ref{fig:linear_v} we plot both analytical and numerical
solutions, where a second order cosine filter is used. The plot shows
good agreement of the numerical solution with the exact one. Using the
knowledge of the singularity of the solution at $v = 0$, more
sophisticated techniques can be used to post-process the Galerkin
solution. For example, Figure~\ref{fig:linear_v_zoom} shows the
result of using Gegenbauer reprojection method (with end-point
singularity) \cites{GottliebShu:97, ChenShu:14, ChenShu:15}. Excellent agreement with the exact solution is observed. 

Furthermore, the limit at $x = \infty$ of the solution to the
half-space isotropic NTE is a constant, whose amplitude agrees with
the extrapolation length. In Table~\ref{table:extrapolation_length} we
compare our numerical approximation of the extrapolation length with the
exact result, which is again in good agreement. In comparison, we note
that the approximate value for the extrapolation length obtained
in~\cite{Coron:90} is $0.71040377$ with $70$ modes, while we achieve
better results with piecewise polynomial of orders up to $12$.
\begin{table}[htp]
\caption{Numerical approximations of the extrapolation length.}
\begin{center}
\begin{tabular}{|c|c|c|c|c|c|c|c|}
4 & 0.709324539775964 & 24 & 0.710445373807707 & 44 & 0.710446026371328 & 64 & 0.710446075479882\\
8 & 0.710386430787361 & 28 & 0.710445703544666 & 48 & 0.710446044962143 & 68 & 0.710446078520678\\
12 & 0.710434523809144 & 32 & 0.710445863417934 & 52 & 0.710446057194912 & 72 & 0.710446080785171\\
16 & 0.710442451548528 & 36 & 0.710445948444682 & 56 & 0.710446065509628 & 76 & 0.710446082499459\\
20 & 0.710444603305304 & 40 & 0.710445997010591 & 60 & 0.710446071320336 & exact & 0.710446089598763
\end{tabular}
\end{center}
\label{table:extrapolation_length}
\end{table}%

\begin{figure}
\centering
\includegraphics[width = 0.5\textwidth]{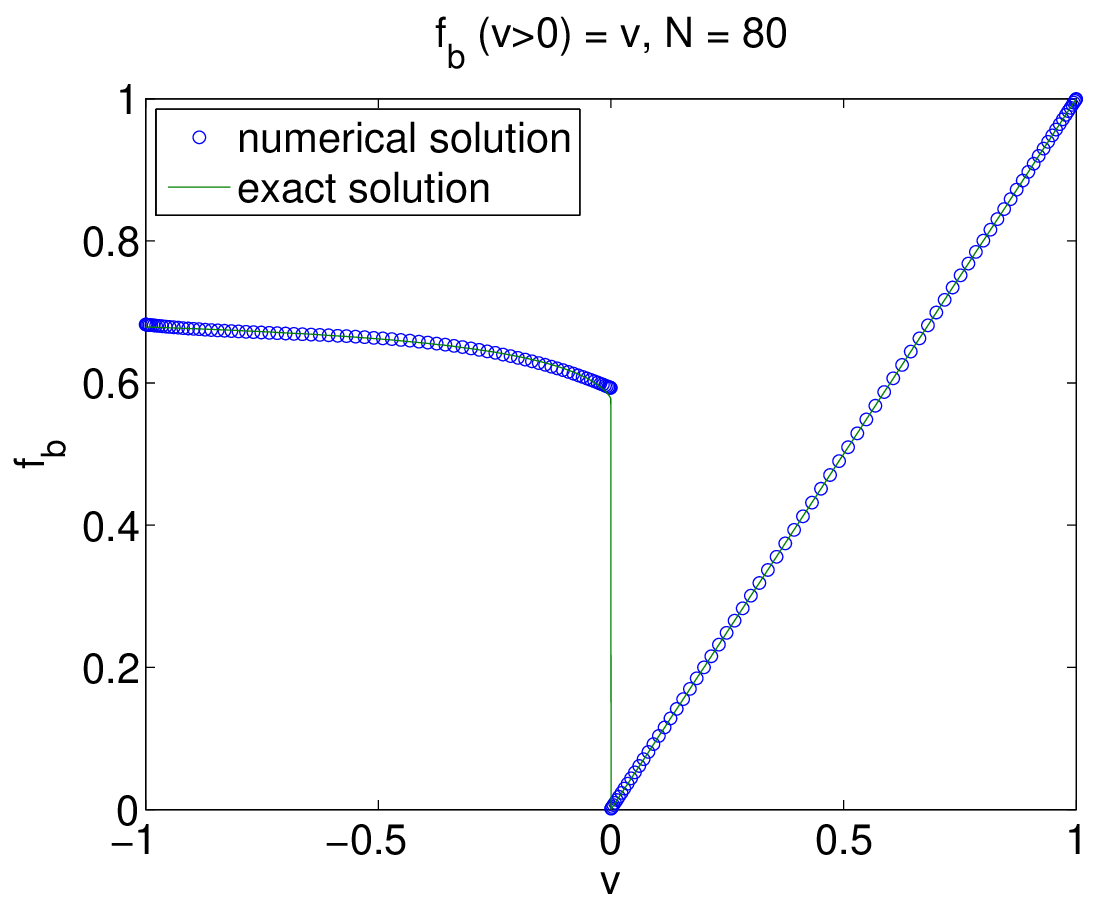}
\caption{Analytical solution and numerical solution to the isotropic neutron transport equation at $x=0$.}\label{fig:linear_v}
\end{figure}

\begin{figure}
\centering
\includegraphics[width = 0.5\textwidth]{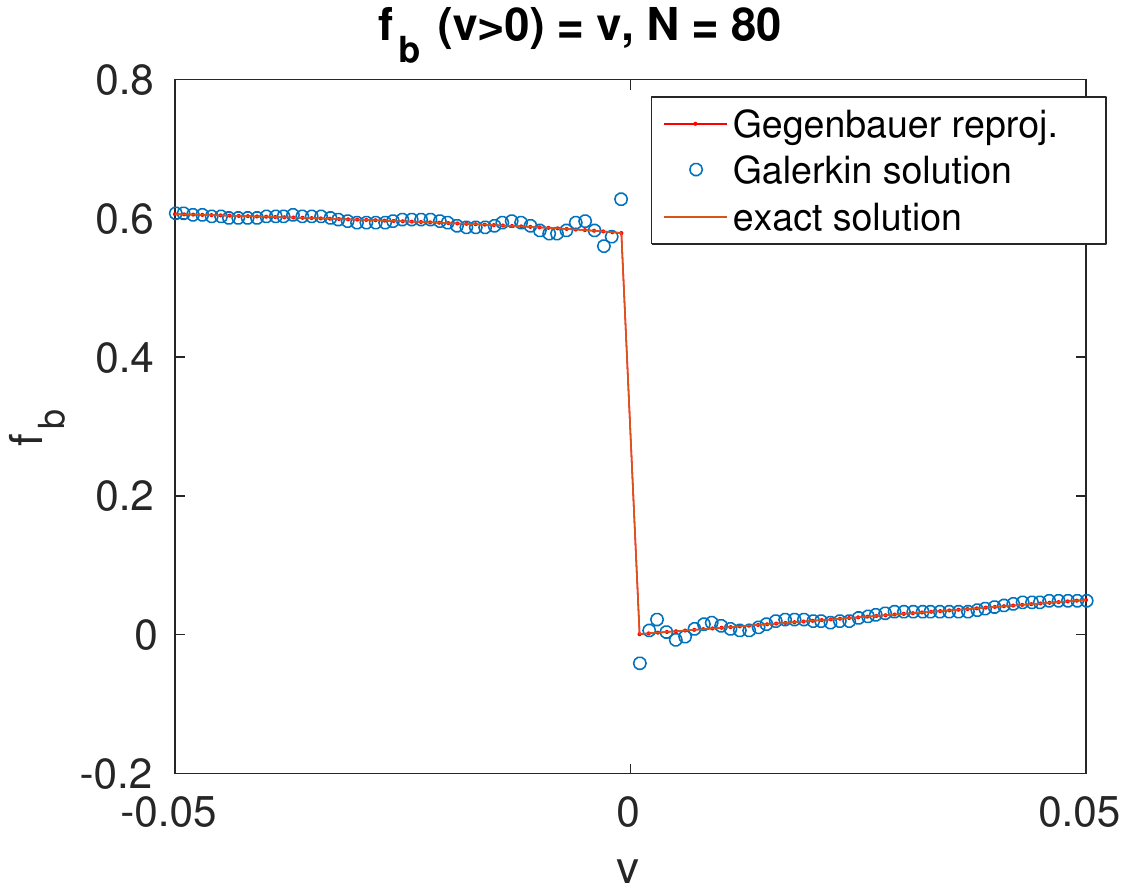}
\caption{Analytical solution, numerical Galerkin solution, and the
  Gegenbauer reprojected solution to the isotropic neutron transport
  equation at $x=0$ (zoomed in around $v = 0$).}\label{fig:linear_v_zoom}
\end{figure}


\appendix 

\section{Half-Hermite polynomial}\label{sec:ortho}

Here we derive the half-space orthogonal polynomial with weight
$\exp(-(v-u)^2)$ with $u$ a real number. The zeroth order half
space Hermite polynomial is:
\begin{equation}
  B_0 = \frac{1}{\sqrt{m_0}}\quad\text{with}\quad m_0 = \frac{\sqrt{\pi}}{2}\left(1+\erf(u)\right).
\end{equation}
The higher order polynomials are defined through recurrence relation: 
\begin{equation}\label{eq:recusion}
  \sqrt{\beta_{n+1}}B_{n+1} = (v-\alpha_n)B_{n} - \sqrt{\beta_n}B_{n-1},
\end{equation}
where $\alpha$ and $\beta$ are defined by
\begin{equation}\label{eq:deduction}
\begin{cases}
  \displaystyle \beta_{n+1} = n+\frac{1}{2} + u\alpha_n -\alpha_n^2 - \beta_n;\\
  \displaystyle \alpha_{n+1} = u - \alpha_n + \frac{1}{2\beta_{n+1}}
  \sum_{k=0}^n \alpha_k
\end{cases}
\end{equation}
with $\alpha_0 = m_1/m_0$ and $\sqrt{\beta_1} = \sqrt{m_0m_2-m_1^2}/m_0$, 
where $m_i$, $i = 0, 1, 2$  are moments of the Gaussian:
\begin{equation}
  m_i  = \int_0^\infty v^{i} e^{-(v-u)^2}\ud{v}, \qquad i = 0, 1, 2.
\end{equation}
The deduction formula are derived from the Christoffel-Darboux
identity 
\begin{equation}\label{eq:CD}
  \sum_{k=0}^n B_k^2 = \sqrt{\beta_{n+1}} \left(B'_{n+1}B_n-B_{n+1}B_n'\right)
\end{equation}
as follows. By orthogonality of $\{B_n\}$, we get
\begin{equation*}
  \alpha_n = \int_0^{\infty} v B_n^2 e^{-(v-u)^2} \ud v, \quad \text{and} \quad \sqrt{\beta_{n+1}} = \int_0^{\infty} v B_n B_{n+1} e^{-(v-u)^2} \ud v.
\end{equation*}
Integrate the identity \eqref{eq:CD} over $v$ with the weight, we get
\begin{align*}
n+1 &= \sqrt{\beta_{n+1}}\int_0^{\infty} B'_{n+1}B_n e^{-(v-u)^2} \ud v
= \int_0^{\infty} vB_{n+1}B'_{n+1} e^{-(v-u)^2} \ud v
\\ &= -\frac{1}{2} + \int_0^{\infty} v^2 B^2_{n+1} e^{-(v-u)^2} \ud v - u\alpha_n, 
\end{align*}
where the second equality is obtained by taking the inner product with
$B'_{n+1}$ of recursion equation \eqref{eq:recusion}, and the third
comes from integration by parts. From this we get the first deduction
relation in \eqref{eq:deduction}. 
Next multiply \eqref{eq:CD} with $v$ and then integrate, we obtain 
\begin{align*}
\sum_{k=0}^n\alpha_k & = \sqrt{\beta_{n+1}}\int_0^{\infty} v B'_{n+1}B_n e^{-(v-u)^2} \ud v\\
& = \sqrt{\beta_{n+1}}\left(2\int_0^{\infty} v^2 B_{n+1}B_n e^{-(v-u)^2} \ud v -2u \int_0^{\infty} vB_{n+1}B_n e^{-(v-u)^2} \ud v \right)\\
& = 2\beta_{n+1}\left(\alpha_n+\alpha_{n+1}-u\right),
\end{align*}
where the first equality comes from the fact that $\int_0^{\infty} v
  B_{n+1}B'_n e^{-(v-u)^2} \ud v=0$, the second is due to integration by parts,
and the third comes from integrating the recursion equation
\eqref{eq:recusion} multiplied by $vB_{n+1}$. This gives the other
deduction relation in \eqref{eq:deduction}.

\bibliographystyle{amsxport}
\bibliography{kinetic}

\begin{bibdiv}
\begin{biblist}

\bib{BabuskaAziz:72}{incollection}{
      author={Babu\v{s}ka, I.},
      author={Aziz, A.K.},
       title={Survey lectures on the mathematical foundations of the finite
  element method},
        date={1972},
   booktitle={{The Mathematical foundation of the Finite Element method with
  Applications to Partial Differential Equations (A.K.~Aziz (ed.))}},
   publisher={Academic Press},
     address={New York},
       pages={1\ndash 359},
}

\bib{BesseBorghol:11}{article}{
      author={Besse, C.},
      author={Borghol, S.},
      author={Goudon, T.},
      author={Lacroix-Violet, I.},
      author={Dudon, J.-P.},
       title={Hydrodynamic regimes, {K}nudsen layer, numerical schemes:
  definition of boundary fluxes},
        date={2011},
     journal={Adv. Appl. Math. Mech.},
      volume={3},
      number={5},
       pages={519\ndash 561},
}

\bib{BLP:79}{article}{
      author={Bensoussan, A.},
      author={Lions, J.L.},
      author={Papanicolaou, G.C.},
       title={Boundary-layers and homogenization of transport processes},
        date={1979},
     journal={J. Publ. RIMS Kyoto Univ.},
      volume={15},
       pages={53\ndash 157},
}

\bib{BardosSantosSentis:84}{article}{
      author={Bardos, C.},
      author={Santos, R.},
      author={Sentis, R},
       title={Diffusion approximation and computation of the critical size},
        date={1984},
     journal={Trans. Amer. Math. Soc.},
      volume={284},
      number={2},
       pages={617\ndash 649},
}

\bib{BardosYang:12}{article}{
      author={Bardos, C.},
      author={Yang, X.},
       title={The classification of well-posed kinetic boundary layer for hard
  sphere gas mixtures},
        date={2012},
     journal={Comm. Partial Differential Equations},
      volume={37},
      number={7},
       pages={1286\ndash 1314},
}

\bib{CGP12}{article}{
      author={Cheng, Y.},
      author={Gamba, I.},
      author={Proft, J.},
       title={Positivity-preserving discontinuous {G}alerkin schemes for linear
  {V}lasov-{B}boltzmann transport equations},
        date={2012},
     journal={Math. Comp.},
      volume={81},
       pages={153\ndash 190},
}

\bib{CoronGolseSulem:88}{article}{
      author={Coron, F.},
      author={Golse, F.},
      author={Sulem, C.},
       title={A classification of well-posed kinetic layer problems},
        date={1988},
     journal={Comm. Pure Appl. Math.},
      volume={41},
       pages={409\ndash 435},
}

\bib{CLY-04}{article}{
      author={Chen, C.-C.},
      author={Liu, T.-P.},
      author={Yang, T.},
       title={Existence of boundary layer solutions to the {B}oltzmann
  equation},
        date={2004},
     journal={Anal. Appl.},
      volume={2},
      number={4},
       pages={337\ndash 363},
}

\bib{Coron:90}{article}{
      author={Coron, F.},
       title={Computation of the asymptotic states for linear half space
  kinetic problems},
        date={1990},
     journal={Transport Theory Statist. Phys.},
      volume={19},
      number={2},
       pages={89\ndash 114},
}

\bib{ChenShu:14}{article}{
      author={Chen, Z.},
      author={Shu, C.-W.},
       title={Recovering exponential accuracy from collocation point values of
  smooth functions with end-point singularities},
        date={2014},
     journal={J. Comp. Appl. Math.},
      volume={265},
       pages={83\ndash 95},
}

\bib{ChenShu:15}{article}{
      author={Chen, Z.},
      author={Shu, C.-W.},
       title={Recovering exponential accuracy in {F}ourier spectral methods
  involving piecewise smooth functions with unbounded derivative
  singularities},
        date={2015},
     journal={J. Sci. Comput.},
       pages={1–21},
}

\bib{Dellacherie:03}{article}{
      author={Dellacherie, S.},
       title={Coupling of the {W}ang {C}hang-{U}hlenbeck equations with the
  multispecies {E}uler system},
        date={2003},
     journal={J. Comput. Phys.},
      volume={189},
}

\bib{DegondMas-Gallic:87}{article}{
      author={Degond, P.},
      author={Mas-Gallic, S.},
       title={Existence of solutions and diffusion approximation for a model
  {F}okker-{P}lanck equation},
        date={1987},
     journal={Transport Theory Statist. Phys.},
      volume={16},
      number={4--6},
       pages={589\ndash 636},
}

\bib{EggerSchlottbom:12}{article}{
      author={Egger, H.},
      author={Schlottbom, M.},
       title={A mixed variational framework for the radiative transfer
  equation},
        date={2012},
     journal={Math. Models Methods Appl. Sci.},
      volume={22},
       pages={1150014},
}

\bib{GolseKlar:95}{article}{
      author={Golse, F.},
      author={Klar, A.},
       title={A numerical method for computing asymptotic states and outgoing
  distributions for kinetic linear half-space problems},
        date={1995},
     journal={J. Stat. Phys.},
      volume={80},
      number={5--6},
       pages={1033\ndash 1061},
}

\bib{Golse:08}{article}{
      author={Golse, F.},
       title={Analysis of the boundary layer equation in the kinetic theory of
  gases},
        date={2008},
     journal={Bull. Inst. Math. Acad. Sin. (N.S.)},
      volume={3},
      number={1},
       pages={211\ndash 242},
}

\bib{GottliebShu:97}{article}{
      author={Gottlieb, D.},
      author={Shu, C.-W.},
       title={On the {G}ibbs phenomenon and its resolution},
        date={1997},
     journal={SIAM Rev.},
      volume={30},
       pages={644\ndash 668},
}

\bib{JinPareschiToscani:01}{article}{
      author={Jin, S.},
      author={Pareschi, L.},
      author={Toscani, G.},
       title={Uniformly accurate diffusive relaxation schemes for multiscale
  transport equations},
        date={2001},
     journal={SIAM J. Numer. Anal.},
      volume={38},
       pages={913\ndash 936},
}

\bib{LiLuSun2015}{misc}{
      author={Li, Q.},
      author={Lu, J.},
      author={Sun, W.},
       title={Half-space kinetic equations with general boundary conditions},
        note={Math. Comp., in press},
}

\bib{Marshak:47}{article}{
      author={Marshak, R.~E.},
       title={Note on the spherical harmonic method as applied to the {M}ilne
  problem for a sphere},
        date={1947},
     journal={Phys. Rev.},
      volume={71},
       pages={443\ndash 446},
}

\bib{Shizgal:81}{article}{
      author={Shizgal, B.},
       title={A {G}aussian quadrature procedure for use in the solution of the
  {B}oltzmann equation and related problems},
        date={1981},
     journal={J. Comput. Phys.},
      volume={41},
       pages={309\ndash 328},
}

\bib{UkaiYangYu:03}{article}{
      author={Ukai, S.},
      author={Yang, T.},
      author={Yu, S.-H.},
       title={Nonlinear boundary layers of the {B}oltzmann equation. {I}.
  {E}xistence},
        date={2003},
     journal={Comm. Math. Phys.},
      volume={236},
      number={3},
       pages={373\ndash 393},
}

\bib{WYY-06}{article}{
      author={Wang, W.},
      author={Yang, T.},
      author={Yang, X.},
       title={Nonlinear stability of boundary layers of the {B}oltzmann
  equation for cutoff hard potentials},
        date={2006},
     journal={J. Math. Phys},
      volume={47},
      number={8},
       pages={083301},
}

\bib{WYY-07}{article}{
      author={Wang, W.},
      author={Yang, T.},
      author={Yang, X.},
       title={Existence of boundary layers to the {B}oltzmann equation with
  cutoff soft potentials},
        date={2007},
     journal={J. Math. Phys},
      volume={48},
      number={7},
       pages={073304},
}

\end{biblist}
\end{bibdiv}

\end{document}